\documentclass[12pt,reqno]{amsart}
\usepackage[latin1]{inputenc}
\usepackage{amsmath}
\usepackage{amsfonts}
\usepackage{amssymb}
\usepackage{graphicx}
\usepackage{enumitem}
\usepackage{xcolor}
\usepackage{hyperref}
\usepackage{url}
\usepackage{amsthm}
\usepackage{bm}
\usepackage[left=3cm,right=3cm,top=2cm,bottom=2cm]{geometry}
{\theoremstyle{definition}
\newtheorem{definition}{Definition}[section]} 
\newtheorem{theorem}[definition]{Theorem}
\newenvironment{theorem*}[1]{{\bf Theorem #1} \begin{itshape}}{\end{itshape}}
\newtheorem{lemma}[definition]{Lemma}
\newtheorem{corollary}[definition]{Corollary}
\newenvironment{corollary*}[1]{{\bf Corollary #1} \begin{itshape}}{\end{itshape}}
\newtheorem{proposition}[definition]{Proposition}
\newenvironment{proposition*}[1]{{\bf Proposition #1} \begin{itshape}}{\end{itshape}}

\theoremstyle{remark}
\newtheorem{remark}[definition]{Remark}
\usepackage{geometry}
\newcommand{\N}{\mathbb{N}}
\newcommand{\upper}{\mathbb{H}}
\newcommand{\Z}{\mathbb{Z}}
\newcommand{\Q}{\mathbb{Q}}
\newcommand{\R}{\mathbb{R}}
\newcommand{\I}{\mathbb{I}}
\newcommand{\Qs}{\mathbb{Q}_{S}}
\newcommand{\Zs}{\mathbb{Z}_{S}}
\newcommand{\cM}{\mathcal{M}}

\newcommand{\U}{\mathcal{U}}

\newcommand{\Qp}{\mathbb{Q}_{p}}
\newcommand{\Zp}{\mathbb{Z}_{p}}
\newcommand{\Qv}{\mathbb{Q}_{\nu}}
\newcommand{\Zv}{\mathbb{Z}_{\nu}}
\newcommand{\cR}{\mathcal{R}}
\newcommand{\cF}{\mathcal{F}}
\newcommand{\cB}{\mathcal{B}}

\newcommand{\cU}{\mathcal{U}}
\newcommand{\cD}{\mathcal{D}}
\newcommand{\cG}{\mathcal{G}}

\newcommand{\cK}{\mathcal{K}}

\newcommand{\W}{\mathcal{W}}
\newcommand{\Bad}{\mathbf{Bad}}
\newcommand{\bx}{\mathbf{x}}
\newcommand{\by}{\mathbf{y}}
\newcommand{\ba}{\mathbf{a}}

\newcommand{\bb}{\mathbf{b}}

\newcommand{\bp}{\mathbf{p}}
\newcommand{\bd}{\mathbf{d}}

\newcommand{\bi}{\mathbf{i}}
\newcommand{\bt}{\boldsymbol{\tau}}
\newcommand{\bsigma}{\boldsymbol{\sigma}}
\newcommand{\ff}{\mathbf{F}}
\newcommand{\bff}{\mathbf{f}}

\newcommand{\dist}{\operatorname{dist}}

\newcommand{\supp}{\operatorname{supp}}

\newcommand{\diag}{\operatorname{diag}}
\newcommand{\SL}{\operatorname{SL}}
\newcommand{\Int}{\operatorname{Int}}
\newcommand{\Ext}{\operatorname{Ext}}
\newcommand{\DI}{\mathbf{DI}}

\title{Bad is null}

\author{V. Beresnevich}
\address[V. Beresnevich]{Department of Mathematics, University of York, Heslington, York, YO10
5DD, United Kingdom}
\email{victor.beresnevich@york.ac.uk}

\author{S. Datta}
\address[S. Datta]{Department of Mathematics, University of Michigan, Ann Arbor, MI 48109, USA}
\email{shreyasi1992datta@gmail.com}

\author{A. Ghosh}
\address[A. Ghosh]{School of Mathematics, Tata Institute of Fundamental Research, Homi Bhabha Road, Colaba, Mumbai 400005, India}
\email{ghosh.anish@gmail.com}

\author{B. Ward}
\address[B. Ward]{La Trobe University Bendigo Campus, Edwards Rd, Flora Hill, Bendigo, Victoria 3552, Australia}
\email{Ben.Ward@latrobe.edu.au}

\thanks{A.\ G.\ gratefully acknowledges support from a grant from the Infosys foundation, a Department of Science and Technology, Government of India, Swarnajayanti fellowship and a grant from the Department of Atomic Energy, Government of India, under project $12-R\&D-TFR-5.01-0500$.\\
B.\ W.\ acknowledges support from EPSRC Research grant EP/WS22430/1, University of York Winter funding call 2022/23, and Australian Research Council Discovery grant no. 200100994.}

\date{\today}
\begin{document}

\begin{abstract}
In this paper we develop a general framework of badly approximable points in a metric space $X$ equipped with a $\sigma$-finite doubling Borel regular measure $\mu$. We establish that under mild assumptions the $\mu$-measure of the set of badly approximable points is always zero. The framework can be applied to a variety of settings in Diophantine approximation and dynamical systems, which we also consider, including weighted and $S$-arithmetic Diophantine approximations, Diophantine approximation on manifolds and intrinsic approximations on fractals.
\end{abstract}
\setcounter{tocdepth}{1}
\maketitle
\tableofcontents\setcounter{tocdepth}{-1}
\section{Introduction}\label{Intro}

The goal of this paper is to develop a general approach that allows one to prove that, in a variety of settings,
\begin{equation}\label{goal}
\text{the set of badly approximable points is null.}
\end{equation}
Here by null, we mean measure zero with respect to the measure under consideration.
We begin with a brief overview of available tools by revisiting the classical setting of simultaneous Diophantine approximation in $\R^{n}$. Recall that Dirichlet's theorem states that for any point $\bx=(x_{1}, \dots , x_{n}) \in \R^{n}$ and any $N \in \N$ there exists $0<q\leq N$ and $(p_{1}, \dots , p_{n}) \in \Z^{n}$ such that
\begin{equation*}
\max_{1 \leq i \leq n} |qx_{i}-p_{i}|< \frac{1}{N^{1/n}} \, .
\end{equation*}
The notion of badly approximable points is underpinned by a corollary of Dirichlet's theorem, which states that for any point $\bx=(x_1,\dots,x_n)\in\R^n$
\begin{equation*}
\max_{1\le i\le n}|qx_{i}-p_{i}| < \frac{1}{q^{1/n}}
\end{equation*}
holds for infinitely many $(p_1,\dots,p_n,q)\in\Z^n\times\N$. Badly approximable points are simply the irrational points $\bx\in\R^n$ that fail to satisfy this corollary when $1$ is replaced with some arbitrarily small positive constant. Formally, $\bx$ is called badly approximable if there exists $c(\bx)>0$ such that for all $(p_1,\dots,p_n,q)\in\Z^n\times\N$
\begin{equation}\label{eq2}
\max_{1\le i\le n}|qx_{i}-p_{i}| \ge \frac{c(\bx)}{q^{1/n}}\,.
\end{equation}

The set $\Bad(n)$ of badly approximable points is intimately linked to the set $\DI(n)$ of {\em Dirichlet improvable points} for which Dirichlet's Theorem can be improved by some constant $0<c<1$. Namely, $\bx \in \DI(n)$ if and only if there exists $0<c<1$ such that for all sufficiently large $N \in \N$ there exists $0<q\leq N$ and $(p_{1}, \dots , p_{n}) \in \Z^{n}$ such that
\begin{equation*}
\max_{1 \leq i \leq n} |qx_{i}-p_{i}|< \frac{c}{N^{1/n}} \, .
\end{equation*}
In dimension one we have that $\DI(1)=\Bad(1) \cup \Q$, and in higher dimensions we have that $\Bad(n) \subset \DI(n)$ due to a result of Davenport and Schmidt \cite{MR0272722}, see \cite{MR4395950} for an overview.

The set $\Bad(n)$ has full Hausdorff dimension in $\R^n$ as was shown by Jarn\'ik \cite{J28} for $n=1$ and later by Schmidt \cite{Schmidt66,Schmidt80} for any $n$. At the same time $\Bad(n)$ is small in the sense that its Lebesgue measure is zero. In other words $\Bad(n)$ is null. The latter is very well known and we recall the following 4 different ways in which this can be obtained:

\begin{itemize}
\item[\text{(A)}]
The fact that $\Bad(n)$ is null follows from the aforementioned inclusion $\Bad(n) \subset \DI(n)$ established in \cite[Theorem~2]{MR0272722} by Davenport and Schmidt together with the fact that the set $\DI(n)$ is null, also proven by Davenport and Schmidt in a separate publication \cite[Theorem~1]{MR279040}.

\medskip

\item[\text{(B)}]
The fact that $\Bad(n)$ is null also follows from the divergence part of Khintchine's theorem \cite{K26}, which states that for any positive monotonic function $\psi:\N\to\R$ such that
\begin{equation}\label{divsum}
\sum_{q=1}^\infty\psi(q)^n=\infty
\end{equation}
we have that for almost every $\bx\in\R^n$
\begin{equation}\label{K3}
\max_{1\le i\le n}|qx_{i}-p_{i}| < \psi(q)
\end{equation}
holds for infinitely many $(p_1,\dots,p_n,q)\in\Z^n\times\N$. Indeed, let $\psi(q)={(q\log q)^{-1/n}}$. Clearly $\psi$ is monotonic and satisfies \eqref{divsum}. Further, any point $\bx$ satisfying \eqref{K3} infinitely often with this $\psi$ is not badly approximable. By Khintchine's theorem, the set of such points $\bx$ is of full measure and so its complement, which contains $\Bad(n)$, is null.

\medskip

\item[\text{(C)}]  The fact that $\Bad(n)$ is null follows from the ergodicity of the action of the one-parameter subgroup of $\SL_{n+1}(\R)$ consisting of diagonal matrices
    $$
    g_t=\diag(e^{nt},e^{-t},\dots,e^{-t})
    $$
    on the homogeneous space $\SL_{n+1}(\R)/\SL_{n+1}(\Z)$ together with the Dani correspondence, see \cite{Dani85} for details. This will be explicated in \S \ref{ergodic_methods} below in greater detail.

\medskip

\item[\text{(D)}] Also the fact that $\Bad(n)$ is null follows from the Scaling Lemma \cite[Lemma~4]{BV08}, which simply tells us that the set of points $\bx\in\R^n$ such that the system of inequalities
\begin{equation*}
\max_{1\le i\le n}|qx_{i}-p_{i}| < \frac{c}{q^{1/n}}
\end{equation*}
holds for infinitely many $(p_1,\dots,p_n,q)\in\Z^n\times\N$, has the same Lebesgue measure for all $c>0$.
\end{itemize}

\begin{remark}
As a consequence of Mahler's version of Khintchine's Transference principle \cite{Mahler39} we have that a point $\bx \in \Bad(n)$ if and only if $\bx$ is \textit{dually badly approximable}, that is
\begin{equation*}
\bx \in \Bad^{*}(n):= \left\{\bx \in \R^{n} : \begin{array}{l}
\exists\;\, c(\bx)>0 \, \text{ such that } \, |q_{1}x_{1}+ \dots + q_{n}x_{n} + p| \geq \dfrac{c(\bx)}{(\max|q_{i}|)^{n}} \\[1ex]
\text{ for all } {\bf0}\neq(q_{1}, \dots , q_{n},p) \in \Z^{n+1} \end{array}\right\}.
\end{equation*}
Thus in each of the above approaches (A)--(D) if the dual theory shows that $\Bad^{*}(n)$ is null then we can deduce that $\Bad(n)$ is also null. For example, in the dual equivalent of (B) we have, by the Khintchine-Groshev theorem (see for example \cite{BV10}), that for almost every $\bx \in \R^{n}$
\begin{equation*}
|q_{1}x_{1}+\dots + q_{n}x_{n}+p|< \frac{1}{(\max|q_{i}|)^{n}\log (\max|q_{i}|)}
\end{equation*}
  holds for infinitely many $(q_{1},\dots , q_{n},p)\in \Z^{n+1}$. Clearly any point $\bx$ satisfying the above condition is not dually badly approximable and so by Mahler's Transference principle, $\Bad(n)$ is a null set.
\end{remark}

Developing a general approach based on (A), (B) or (C) type arguments seems to have more barriers than doing so with a type (D) argument. Indeed, (A) and (B) require establishing statements that are much stronger that simply proving that Bad is null, while (C) requires an underlying Lie group structure. Within this paper we develop a general approach to proving \eqref{goal} based on a type (D) argument. First, we establish a general framework of badly approximable points in metric spaces in \S\ref{GS}, and then we consider several applications.

\section{A spoiler of applications}

In this small section we give a selection of applications of our general framework. However, the reader is encouraged to look at the general framework in \S \ref{GS} and \S \ref{gen2} and more applications in sections \S \ref{applications}, \ref{App2}, and \ref{S_bad}. In this paper we will consider badly approximable points in the more general setting of weighted Diophantine approximation, which we discuss in \S \ref{weighted} in detail. In particular, we refer to \S \ref{weighted} for definition of $\Bad(\bt)$ -- the set of $\bt$-badly approximable points in $\R^n$, where $\bt=(\tau_1,\dots,\tau_n)$ is the vector of weights. In fact, one of the main themes of this paper is to show the set of weighted badly approximable vectors has measure zero with respect to relevant measures. First, we state a slightly simplified version of one of our main applications of the general framework - Theorem~\ref{bad_realmani} below - regarding weighted badly approximable points on manifolds. In what follows $\R_+$ is the set of positive reals.

\begin{theorem}\label{main1}
Let $\bt=(\tau_{1}, \dots, \tau_{n}) \in \R^{n}_{+}$ satisfy $\tau_1+\dots+\tau_n=1$ and $\cM$ be a $C^2$ submanifold of $\R^n$. Furthermore, in the case the weights $\tau_1,\dots,\tau_n$ are not all the same, assume that almost everywhere the tangent plane to $\cM$ is not orthogonal to any of the coordinate axis. Then almost every point on $\cM$ is not in $\Bad(\bt)$.
\end{theorem}

\begin{remark}
The condition on the tangent plane is mild and can be relaxed, see Theorem~\ref{bad_realmani}. In fact, Theorem~\ref{main1} is an obvious consequence of 
Theorem~\ref{bad_realmani}, and how to deduce one from the other is explained in Remark~\ref{rem8.2}.
\end{remark}

\begin{remark}
As discussed in \S \ref{motivation_manifold_fractal} below, most of the earlier results on badly approximable points on manifolds had the assumption that the manifold is \textit{nondegenerate}. Theorem~\ref{main1} includes the degenerate case, in particular it covers affine subspaces.
\end{remark}

More generally, we consider manifolds in the $S$-arithmetic setting. When $S=\{p\}$, we have a complete analogue of Theorem \ref{main1}; 
see Theorem \ref{bad_manifold_p-adic}. In fact, the proofs are significantly more complex for $S$-arithmetic manifolds, i.e., when $S$ contains more than one place; see \S\ref{proofbad}. Accordingly, we have more restrictive assumptions, where we need conditions on the Diophantine exponent of the manifold; see Theorem \ref{Bad_nonextremal} and Corollary \ref{bad_manifold_S-adic}.

As another application of our general framework, we consider intrinsic Diophantine approximation on fractals; see \S \ref{App2}. Let $\cK=\prod_{j=1}^{n}\cK_{j}\subset \R^n$ be the attractor of an IFS, as defined in \S \ref{App2}, and $\dim\cK_j=\gamma_j.$ Let $\Bad_{\Int}(\cK;\bt)$ be the set of weighted intrinsically badly approximable points in $\cK$ as in \S \ref{App2}:
$$ \Bad_{\Int}(\cK;\bt)= \left\{ \bx =(x_i)_{i=1}^n\in \cK : \exists \, c(\bx)>0 \quad \left|x_{i}-\frac{p_{i}}{q}\right|\geq \frac{c(\bx)}{q(\log q)^{\tau_i}} \quad \forall ~\frac{\bp}{q} \in \Q^{n}\cap \cK \right\}.$$
We establish the following
\begin{theorem}
    Let $\cK$ be as above. Suppose that $\bt=(\tau_{1}, \dots , \tau_{n}) \in \R^{n}_{+}$ and
$\sum_{i=1}^{n}\tau_{i}\gamma_{i}=1$. Then $$\mu(\Bad_{\Int}(\cK;\bt))=0, $$ where $\mu$ is  the canonical self similar measure of $\cK$ as defined in Equation \eqref{mu}.
\end{theorem}
Our main abstract theorems of the general framework are Theorem \ref{bad_0_general}, and more generally Theorem \ref{bad_0} which we will not attempt to state here.
Lastly, we want to mention an application of our method in classical set-up, which one can find in \S \ref{applications}. Here we consider weighted approximations for $n\times m$ matrices in $\Q_S^{mn}$, that is in the $S$-arithmetic setting.  In this direction, we have Theorem \ref{thm6.5}. The reader is encouraged to see the classical case first to understand the complications that arise in the manifolds case in \S \ref{proofbad}, more specifically Theorem \ref{diri}.

\begin{remark}
Potential applications of the general framework are not limited to those presented in this paper and may include a variety of  problems in number theory and dynamical systems, for example, to investigate badly approximable points arising from the so-called shrinking target problem. In particular, in the follow-up paper \cite{BDGW2} we use the methods presented in this paper to study the shrinking target property for irrational rotations of tori and more generally for $\Z^d$ actions on tori \cite{BHKV}. Specifically, we generalise results of Kim \cite{MR2335077} and Shapira \cite{Shapira} by establishing the corresponding shrinking target property for rectangular targets (defined by weights) and considering the much wider $S$-arithmetic setting. 
\end{remark}

\section{Background}

\subsection{Ergodicity, Duality and Bad}\label{ergodic_methods}
We begin with a brief overview of results in various generality that have been proved using a type (C) argument, see \S\ref{Intro}.
The fact that Bad is null in a given arithmetic context is often a manifestation of the ergodicity of a related group action.
We illustrate this principle using some examples.

\subsubsection{Diagonal actions on the space of unimodular lattices.} Let $G = \SL_{n+1}(\R)$ and $\Gamma = \SL_{n+1}(\Z)$. Let $\mathcal{L}_{n+1}$ denote the space of unimodular lattices in $\R^{n+1}$; this space can be identified with the homogeneous space $\SL_{n+1}(\R)/\SL_{n+1}(\Z)$ which has an invariant probability measure inherited from the Haar measure on $\SL_{n+1}(\R)$. The quotient is non-compact and its compact subsets are described by Mahler's compactness criterion. For $\bx\in\R^n$ define the lattice
$$
\Lambda_{\bx}:=\{(q\bx-\bp,q):\, q\in\Z, \bp\in\Z^n\} = \begin{pmatrix}Id & \bx\\ 0 & 1 \end{pmatrix}\Z^{n+1}\in \mathcal{L}_{n+1}.
$$
Consider the action on $\mathcal{L}_{n+1}$ by the one-parameter subgroup
$$g_t := \diag(e^{-t}, e^{t/n},\dots, e^{t/n})
$$
of $\SL_{n+1}(\R)$. In \cite{Dani85}, Dani proved the following fact which nowadays is often referred to as {\em Dani's correspondence}.

\begin{proposition}[Dani \cite{Dani85}]\label{badly}
The vector $\bx \in \R^n$ is badly approximable if and only if the semi-orbit $\{g_t \Lambda_\bx$, $t>0\}$,  is bounded in
$\mathcal{L}_{n+1}$.
\end{proposition}

In fact, Dani proved the above result for systems of linear forms, and also provided a dynamical criterion for a vector to be singular. By Moore's ergodicity criterion, the action of $g_t$ on $\mathcal{L}_{n+1}$ is ergodic and therefore the set of points whose $g_t$ orbits are bounded has zero Haar measure. One can now use the fact that $\left\{\begin{pmatrix}Id & \bx\\ 0 & 1 \end{pmatrix}~:~\bx \in \R^n\right\}$ is the \emph{expanding horospherical} subgroup of $\{g_t\}$ to conclude that this in fact, implies that the set of $\Lambda_\bx$ has zero measure. By Dani's correspondence,  this means that $\Bad(n)$ is null. Dani's approach has proved to be extremely influential and leads itself to many generalisations. For instance, in \cite{EGL16} (see also \cite{AGGL19}), the set $\Bad_k$ of badly approximable vectors with respect to rationals coming from a fixed number field $k$ were studied. Here too, a straightforward adaptation of Dani's argument shows that $\Bad_k$ has zero measure.

\subsubsection{Group actions on homogeneous varieties.} In \cite{GGN13, GGN14, GGN15, GGN18}, a theory of intrinsic Diophantine approximation for homogeneous varieties of semisimple groups has been developed. The approach of Ghosh, Gorodnik and Nevo also uses a correspondence between dynamics of group actions and Diophantine approximation, but here the acting group is a larger, semisimple subgroup. We briefly describe this approach in the simple case of group varieties. Let $G \subset \mathbb{C}^n$ be a simple, simply connected algebraic group defined over $\Q$. Then Diophantine approximation on $G(\R)$ by points in $G(\Q)$ can be studied using the dynamics of the action of the group of finite adeles $G(\mathbb{A}_f)$ on the finite volume homogeneous space $G(\mathbb{A})/G(\mathbb{Q})$ via a duality principle similar, in spirit, to Dani's correspondence. To every point $x \in G(\R)$, one can similarly attach a point $(x, e) \in G(\mathbb{A})$. Then, it can be shown that $x$ is badly approximable if and only if the orbit $(x, e)G(\mathbb{A}_f)$ does not intersect a fixed neighbourhood of the identity coset in $G(\mathbb{A})/G(\mathbb{Q})$. By the ergodicity of the $G(\mathbb{A}_f)$-action, we can once more conclude that the set of badly approximable points has zero measure. A similar statement holds with straightforward modifications, for more general varieties of the form $G/H$.
For quadratic surfaces, a dynamical approach closer to Dani's original one can be applied \cite{MR4397037}, see also \cite{FKMS18}. This is discussed in \S 3.4.

One impetus for the work on homogeneous varieties of semisimple groups came from an old question of Lang regarding Diophantine approximation on abelian varieties. Indeed, there have been studies of metric Diophantine approximation in the projective setting \cite{GhoshHaynes16} including badly approximable points \cite{HH17}.

\subsubsection{Actions of isometries on the Gromov boundary of a hyperbolic space.}
Dani's correspondence can be viewed as a generalisation of the following fact: viewing $\R$ as the boundary of the upper half plane $\upper$, a number $x$ is badly approximable if and only if it is the end point of a geodesic $\gamma$ whose image on the modular surface $\upper/\SL_{2}(\Z)$ is bounded. The dual point of view involves the action of the discrete group of isometries (in this case $\SL_{2}(\Z)$) on its limit set in the boundary of hyperbolic space (in this case, the limit set is $\R \cup \infty$). This study was taken much further by Patterson \cite{Pat76a, Pat76b} who initiated the systematic study of metric Diophantine approximation on limit sets by isometry groups. There has been extensive, very general work in this direction, we refer the reader to \cite{BGSV18, FSU} for recent work on badly approximable points in this context. Here too, the fact that Bad is null follows from the ergodicity of the geodesic flow on the associated hyperbolic manifold. This problem is related to that of Diophantine approximation on fractals as discussed in \S \ref{motivation_manifold_fractal}.

\begin{remark}
In the approach discussed above, it is essential that the measure in question be ergodic and invariant under the corresponding group action. This fails or is very difficult to check for the main examples treated in this paper, namely for Diophantine approximation on manifolds and fractals. Our approach in the present paper, namely a general version of approach (D), is flexible enough to deal with these examples.
\end{remark}

\subsection{Weighted Diophantine approximation in $\R^n$}\label{weighted}

In this setting one allows different rates of approximation in each coordinate direction, which are typically defined by a weights vector, say $\bt=(\tau_{1}, \dots , \tau_{n})\in \R^{n}_{\ge0}$, satisfying $\tau_{1}+\dots+\tau_{n}=1$. The point $\bx\in\R^n$ is called $\bt$-badly approximable if there exists a constant $c(\bx,\bt)>0$ such that for all $(p_1,\dots,p_n, q) \in \Z^{n}\times\N$
\begin{equation*}
|qx_{i}-p_{i}| \geq \frac{c(\bx,\bt)}{q^{\tau_{i}}} \qquad \text{for some }1 \leq i \leq n\,.
\end{equation*}
Clearly, $\Bad(n)=\Bad(\tfrac1n,\dots,\tfrac1n)$. The sets $\Bad(\bt)$ of $\bt$-badly approximable points have attracted major interest in recent years in relation to Schmidt's conjecture on the intersections of the sets $\Bad(\bt)$ in $\R^2$ which was solved in \cite{BPV11}. Subsequent breathtaking progress led to showing that any countable intersections of the sets $\Bad(\bt)$ in $\R^n$ have full Hausdorff dimension for any $n$. Furthermore, it was shown that the restriction of such intersections to non-degenerate submanifolds of $\R^n$, and to certain fractals that support Ahlfors regular measures have full Hausdorff dimension. These developments remain a hot topic area and we refer the reader to \cite[\S 7.1]{BRV16}, \cite{A13,A16,BPV11,D62,KW10,PV02} and references within for the state of the art and open problems.

Naturally, investigating the dimension of the sets $\Bad(\bt)$ or their restrictions to manifolds and fractals is only of interest when we know that these sets (resp. their restrictions) are null, that is of Lebesgue measure zero in $\R^n$ (resp. zero measure on the manifold or fractal). Otherwise, the dimension is trivially full. The fact that $\Bad(\bt)$ is null is also well known as in the unweighted case of $\Bad(n)$. In fact, this fact can also be obtained in the same four ways as in the case of $\Bad(n)$:

\begin{itemize}
  \item[\text{(A)}] Kleinbock and Weiss proved that the set of \textit{weighted Dirichlet improvable} numbers is null \cite[Theorem~1.4]{MR2366229} and then Kleinbock and Rao proved that the set $\Bad(\bt)$ is contained in the set of weighted Dirichlet improvable numbers \cite[Theorem~1.2]{KleinbockRao}. Combining these two findings proves that $\Bad(\bt)$ is null.
  \item[\text{(B)}] Gallagher's generalisation \cite{G62} of Khintchine's theorem implies that $\Bad(\bt)$ is null in the same way as Khintchine's theorem implied that $\Bad(n)$ is null.
  \item[\text{(C)}] Using the ergodicity of the action of the diagonal matrices $g_t=\diag(e^{t},e^{-\tau_1t},\dots,e^{-\tau_nt})$ on $\SL_{n+1}(\R)/\SL_{n+1}(\Z)$ and a generalisation of the argument of Dani \cite{Dani85} as explained earlier implies that $\Bad(\bt)$ is null;
      \item[\text{(D)}] using the Scaling Lemma of \cite[Lemma~4]{BV08} in combination with the Cross Fibering Principle of \cite[Theorem~3]{BHV13} proves that $\Bad(\bt)$ is null in yet another way, see \cite{MR3081774} for an explicit proof.
\end{itemize}

However, these approaches are not easy to generalise when considering badly approximable points in non-linear or `more complex' structures such as submanifolds of, or fractals in, $\R^n$. Also, less is known regarding Diophantine approximation in the $p$-adic, and more generally $S$-arithmetic setting, and other natural setups.

\subsection{Diophantine approximation on manifolds and fractals}\label{motivation_manifold_fractal}

When the $n$-dimensional Lebesgue measure is replaced by another measure $\mu$ on $\R^n$, for example, a measure supported on a smooth manifold or a fractal in $\R^n$, understanding the measure theoretic structure of problems, in particular, regarding badly approximable points, becomes far more difficult and leads to Diophantine approximation of dependent quantities. \par

There is a range of results available when the measure is supported on a manifold. One of the first was due to Khintchine, who proved that $\Bad(n)$ on the Veronese curves was null \cite{MR1544723}. We remark this was proven not by any of the approaches highlighted above.
In the case of planar curves the first named author, Dickinson and Velani proved a Khintchine type theorem on any $C^{(3)}$ non-degenerate planar curve \cite{BDV07}. A dual version of the result has also been proven, namely \cite[Theorem 1.1]{BBDD99}, and a weighted version of the result was established in \cite{BV07}. In higher dimensions a divergence Khintchine type theorem was established for any analytic non-degenerate submanifold of $\R^{n}$ \cite{B12} and any non-degenerate curve \cite{MR4287738}. However, the dual Khintchine type theorem was established well before that and includes all non-degenerate manifolds in $\R^n$, see \cite[Theorem 1.6]{BBKM02}. In \cite[Theorem 2]{BBV13} an $n$-dimensional dual weighted version of a Khintchine divergence statement was proven for inhomogeneous approximation on non-degenerate manifolds. By using a type (B) argument (and Mahler's Transference principle in the dual case) one can use any of the above Khintchine type statements to deduce that the corresponding set of badly approximable points restricted to the manifold is a null set (with respect to the induced Lebesgue measure over the manifold). We should remark that Mahler's Transference principle cannot be applied in the inhomogeneous setting so one cannot deduce immediately that the inhomogeneous equivalent of $\Bad(\bt)$ on non-degenerate manifolds is null.

Regarding type (A) approach, Shah \cite[Theorem~1.1]{MR2534098} proved that $\DI(n)$ on any analytic non-degenerate curve (and, via fibering \cite{B15}, on any analytic non-degenerate manifod) is null, and this gives an alternative way to deduce that $\Bad(n)$ is null on any analytic non-degenerate manifold in $\R^n$. This has since been extended to the weighted case by Shah and Yang \cite[Theorem~1.2]{ShahYang1}. Furthermore in a separate paper the same authors proved that $\DI(n)$ intersected with a real analytic manifold contained in a hyperplane (and satisfying various other constraints) is also a null set, and thus so is $\Bad(n)$ \cite[Theorem~1.3]{ShahYang2}. Note that the results using type (A) argument were preceded by a range of results on subsets of Dirichlet improvable points on manifolds (the so-called $\varepsilon$-Dirichlet improvable points), see \cite{MR0272722, MR485713, MR1068674, MR1895732, MR2366229}. However, these results are not sufficient to deduce corresponding results on the measure of $\Bad(n)$ or $\Bad(\bt)$. \par

Observe that many of the results above, with the exception of \cite{ShahYang2}, are reliant on the intersecting manifold being non-degenerate. As an application of our main result we prove that $\Bad(\bt)$ is null when intersected with any $C^{(2)}$ (not necessarily non-degenerate) submanifold with a mild condition on the orientation of its tangent planes.

The $\mu$-measure of $\Bad(\bt)$ when $\mu$ is supported on a fractal set is more complicated. In \cite[Theorem~1.5]{EFS11} $\Bad(n)$ was proven to be a null set for any invariant $\times n$ ergodic probability measure $\mu$. An example of such a measure includes the natural canonical self-similar measure on the middle third Cantor set. In the recent breakthrough \cite{MR3953505} $\Bad(n)$ was proven to be a null set for the Hausdorff $s$-measure restricted to the attractor of an irreducible ``contracting on average'' finite iterated function systems (IFS) satisfying the open set condition (OSC), where $s$ is the Hausdorff dimension of the attractor of the IFS. Such result was proven by a type (C) argument. This has recently been generalised to $\Bad(\bt)$ for certain values of $\bt$ dependent on the contractions of the maps in the IFS associated to the self affine McMullen measure \cite{ProhaskaSertShi}. An example of such a fractal includes Bedford-McMullen carpets satisfying certain criteria (see \cite[Theorem~1.11]{ProhaskaSertShi} for more details). In \cite{KhalilLuethi} a Khinthcine type theorem, that hence implies that $\Bad(n)$ is null,  was proven for a class of (not necessarily canonical) self similar measures on rational IFS satisfying the OSC and other certain conditions on the contractions of the IFS and the probability vector on the self similar measure (see \cite[Theorem~A]{KhalilLuethi} for more details). We note in conclusion that if the measure $\mu$ is the Patterson-Sullivan measure on the (fractal) limit set of a suitable isometry group, then one can use ergodicity of the geodesic flow as remarked earlier.

\subsection{Intrinsic and extrinsic approximations}

Within the setting of Diophantine approximation of dependent quantities one can ask for further restrictions on the set $\Bad(\bt)$. One such restriction is to replace the set of rational points within the definition of $\Bad(\bt)$ by the set of rational points that lie in (or out of) the manifold or fractal. With this in mind, define
\begin{align*}
\Bad_{\Int}(\bt)&:=\left\{ \bx \in \supp \mu : \, \inf_{\frac{\bp}{q} \in \Q^{n} \cap \supp \mu} \, \max_{1\le i\le n}q^{\tau_{i}}\left|qx_{i}-p_{i}\right|  > 0 \right\}, \\
\Bad_{\Ext}(\bt)&:=\left\{ \bx \in \supp \mu : \, \inf_{\frac{\bp}{q} \in \Q^{n} \backslash \supp \mu} \, \max_{1\le i\le n}q^{\tau_{i}}\left|qx_{i}-p_{i}\right|>0 \right\}.
\end{align*}
The points in $\Bad_{\Int}(\bt)$ ({\em resp.} $\Bad_{\Ext}(\bt)$) will be called intrinsically ({\em resp.} extrinsically) badly approximable.

Observe that $\mu(\Bad(\bt))=0$ trivially implies that both $\Bad_{\Int}(\bt)$ and $\Bad_{\Ext}(\bt)$ are null sets since
\begin{equation*}
\Bad(\bt)\cap \supp \mu =\Bad_{\Int}(\bt) \cap \Bad_{\Ext}(\bt)\,.
\end{equation*}

Intrinsically badly approximable points on manifolds is a relatively unexplored topic. Kleinbock and Merrill \cite[Theorem~1.3]{MR3430242} proved a Khintchine type theorem for intrinsic approximations on the unit sphere $S^{n} \subset \R^{n+1}$ centred at the origin which implies that $\Bad_{\Int}(0)$ is null. Here $0$ means that the approximation function is $q^{0}=1$. This was generalised to non-singular rational quadric hypersurfaces with rational rank greater than or equal to $1$ in \cite[Theorem~6.3]{MR4397037}. One of the reasons there is a lack of results in this setting for general manifolds is the difficulty in establishing the distribution of rational points on a general manifold, see \cite[Example 1.1]{FKMS18} for a simple set of curves each with different Dirichlet exponent. \par

In the case of intrinsic approximations on fractals the known results are proven under additional constrains on the rational points, giving rise to a set slightly different to $\Bad_{\Int}(\bt)$. One of the first results in this setting was due to Levesley, Salp and Velani \cite{LSV07} who proved a Khintchine-type result for the set of triadic approximable points (approximating points by rationals of the form $\frac{p}{3^{k}}$, $k \in\N$) on the middle third Cantor set. In \cite{FS14} a similar result was proven with the middle third Cantor set generalised to a base $b$ missing digit set satisfying certain conditions (see \cite[Theorem~3.10]{FS14} for more details). Recently a Khintchine-type result in the standard setting of intrinsic approximations on the middle third Cantor set (approximating points are not necessarily of the form $\frac{p}{3^{k}}$, $k \in\N$) was obtained in \cite{TanWangWu}. Any of these results implies that the corresponding set of intrinsically badly approximable points is null with respect to the corresponding Hausdorff measure restricted to the Cantor set.

Recently there has been interest in the approximation of points in a fractal by sets of dense subsets other than the rational points. Namely, given an iterated function system (IFS) $\Sigma$ and a corresponding fractal set $\cK \subset \R^{n}$ one can consider how well points in $\cK$ can be approximated by images of a point $x \in \cK$ under the maps in $\Sigma$. That is, given a fixed $x \in \cK$, the approximation points are taken to be
\begin{equation*}
\cR=\{f(x) : f \text{ is a countable composition of maps in } \Sigma \}.
\end{equation*}
Note that if the maps $f \in \Sigma$ are rational preserving and the initial point $x \in \Q^{n} \cap \cK$ then $\cR$ is a sub-collection of $\Q^{n}$, and so similar results to Levesley, Salp \& Velani can be proven by considering this setup. This method has proven fruitful for a range of fractals including conformal IFS with open set condition (OSC) \cite{AB21}, and more recently conformal IFS without the OSC \cite{SBaker}. Both these papers prove a Khintchine type theorem, so the measure of the corresponding intrinsically badly approximable set can be deduced via a type (B) argument. \\






\section{General framework}\label{GS}

Diophantine approximation to real points that we have discussed in the introduction can be considered in a variety of other metric spaces. For instance, one can consider rational approximations to points in the $p$--adic space $\Q_p^n$, approximations on the limit set of a Kleinian group $G$ by points in the orbit (under the group) of a certain distinguished limit point, and so on. In this section we will introduce a general framework of badly approximable points in metric spaces which we will use in several applications.

The aim is to incorporate weighted Diophantine approximation into the framework by considering a product metric space and allowing different rates of approximation in each of the components of the product. With this in mind let us fix an integer $n \geq 1$, and for each $1 \leq i \leq n$ let $(X_{i}, d_{i}, \mu_{i})$ be a metric space equipped with a $\sigma$-finite Borel regular measure $\mu_{i}$. Let $(X,d,\mu)$ be the product space with $X=\prod_{i=1}^{n}X_{i}$, $\mu=\mu_{1}\times\cdots\times\mu_n$ being the product measure, and
\begin{equation*}
d(\bx^{(1)},\bx^{(2)})=\max_{1 \leq i \leq n}d_{i}(x_i^{(1)},x_i^{(2)})\,, \qquad \text{where }\bx^{(j)}=(x_1^{(j)},\dots,x_n^{(j)})\,\,\text{for }j=1,2.
\end{equation*}
Let $\cR:=(R_{\alpha})_{\alpha}$ be a sequence of subsets $R_{\alpha} \subset X$ indexed by $\alpha \in \N$ (instead of $\N$ one can use any other countable set).
Let $\Phi=(\phi_{1}, \dots , \phi_{n})$ be an $n$-tuple of non-negative approximation functions defined on $\alpha$'s and such that $\phi_{i}(\alpha) \to 0$ as $\alpha \to \infty$. For any constant $c>0$ let $c\Phi=(c \phi_{1}, \dots , c \phi_{n})$. Let
\begin{equation*}
\Delta(R_{\alpha}, \Phi):= \left\{ \bx \in X : \, \, \exists \, \ba \in R_{\alpha} \,\text{ such that } \, d_{i}(a_{i},x_{i})< \phi_{i}(\alpha) \,\text{ for all } 1 \leq i \leq n \right\},
\end{equation*}
where $a_{i}$ and $x_{i}$ represent the $i$th component of $\ba$ and $\bx$ respectively $(1 \leq i \leq n)$.

Define the set $\W_{\cR}(\Phi)$ of $(\Phi, \cR)$-approximable points as follows
\begin{equation*}
\W_{\cR}(\Phi):=\limsup_{\alpha \to \infty } \Delta(R_{\alpha}, \Phi)\,.
\end{equation*}
That is $\W_{\cR}(\Phi)$ is the set of points in $X$ that lie in the $\Phi(\alpha)$-neighborhood $\Delta(R_{\alpha}, \Phi)$ of $R_{\alpha}$ for infinitely many $\alpha\in\N$.

Define the set $\Bad_{\cR}(\Phi)$ of $(\Phi, \cR)$-badly approximable points in $X$ as follows
\begin{equation}\label{BadR}
\Bad_{\cR}(\Phi):=\W_{\cR}(\Phi) \backslash \bigcap_{k \in \N} \W_{\cR}\left(\tfrac{1}{k}\Phi \right).
\end{equation}
That is, the set of $(\Phi, \cR)$-approximable points $\bx$ that cannot be `improved' by an arbitrary constant i.e. $\bx \in \W_{\cR}(\Phi)$ and there exists constant $c>0$ such that $\bx \not \in \W_{\cR}(c\Phi)$. Observe that the definition above aligns with the classical definition of weighted badly approximable points in $\R^n$ when the $n$-tuple of approximation functions $\Phi$ corresponds to weighted Dirichlet's theorem, that is when $\Phi(q)=(q^{-\tau_1},\dots,q^{-\tau_n})$, see \S\ref{weighted}. If $\Phi$ does not correspond to Dirichlet's Theorem, the set of $\Phi$-badly approximable points, that is  $\Bad_{\cR}(\Phi)$, may be null for trivial reasons, see the following paragraph. However, even then this set is still of significant interest, especially due to natural and obvious relations between the sets $\Bad_{\cR}(\Phi)$ and the so-called `Exact order' sets. For instance, determining the Hausdorff dimension of $\Bad_{\cR}(\Phi)$, let alone the Exact order sets, is a challenging open problem even in the classical setting, see \cite{BGN22, Bu03, Bu08, BM11, HX22, Z12} and references within for relevant definitions and state-of-the-art.

As we already mentioned $\W_{\cR}(\Phi)$ may be null for trivial reasons. In particular, if $\mu(\W_{\cR}(\Phi))=0$ then clearly $\mu(\Bad_{\cR}(\Phi))=0$, since $\Bad_{\cR}(\Phi) \subset \W_{\cR}(\Phi)$. Similarly if  there exists some $\varepsilon>0$ such that $\W_{\cR}(\Phi^{1+\varepsilon})$ contains almost every point, then $\mu(\Bad_{\cR}(\Phi))=0$ since $\Bad_{\cR}(\Phi) \subset \W_{\cR}(\Phi) \backslash \W_{\cR}(\Phi^{1+\varepsilon})$, where $\Phi^{1+\varepsilon}=(\phi_{1}^{1+\varepsilon}, \dots , \phi_{n}^{1+\varepsilon})$. In  scenarios such as the two we have just exposed achieving goal \eqref{goal} is of course trivial. However, this goal can be also be achieved in many non-trivial situations. In particular we prove the following:

\begin{theorem} \label{bad_0_general}
 Fix $n \in \N$ and for each $1\le i\le n$ let $(X_{i}, d_{i}, \mu_{i})$ be a metric space equipped with a $\sigma$-finite Borel regular measure $\mu_i$. Let $(X,d,\mu)$ be the corresponding product space introduced above. Let $\Phi=(\phi_{1}, \dots , \phi_{n})$ be an $n$-tuple of approximation functions with each $\phi_{i}(\alpha) \to 0$ as $\alpha \to \infty$ and let $\cR=\left(R_{\alpha}\right)_{\alpha\in \N}$ be a sequence of subsets of $X$. Assume at least one of the following holds:\\[-2ex]
\begin{itemize}
\item[{\rm(i)}] $\mu$ is doubling\footnote{See \S\ref{sec3.1} for the definition.}, and $R_{\alpha} \subset \supp \mu$ for each $\alpha \in \N$;\\[-1ex]
\item[{\rm(ii)}] $\phi_1=\dots=\phi_n$, and for any $0<\delta<\rho$ there exists a constant $c(\delta,\rho)>0$ such that for each sufficiently large $\alpha \in \N$ and every point $\bx \in R_{\alpha}$ we have that
\begin{equation*}
\mu(\Delta(\bx, \delta\Phi)) \geq c(\delta,\rho)\mu(\Delta(\bx, \rho\Phi))\,.
\end{equation*}
\end{itemize}
Then
\begin{equation}\label{Badisnull}
\mu(\Bad_{\cR}(\Phi))=0.
\end{equation}
\end{theorem}

\bigskip

\subsection{Outline of the proof}\label{outline}

In view of the definition \eqref{BadR} of $\Bad_{\cR}(\Phi)$ establishing \eqref{Badisnull} can be achieved using the following simple and well known statement.

\begin{lemma} \label{measure_null}
Let $\mu$ be a measure over $X$, $(E_{k})_{k \in \N}$ be any decreasing sequence of subsets of $X$ such that
\begin{equation} \label{measure_constant}
\mu(E_{i}\setminus E_{j})=0\qquad\text{for all $i<j$. }
\end{equation}
Then for any $k \in \N$
\begin{equation*}
\mu\left(E_{k} \backslash \bigcap_{i=1}^{\infty}E_{i} \right)=0.
\end{equation*}
\end{lemma}
\begin{proof}
For completeness we give a proof. Since $E_{1} \supset E_{2} \supset E_{3} \dots$, by the sub-additivity of $\mu$,
\begin{align*}
0\le \mu\left(E_{k} \backslash \bigcap_{i=1}^{\infty}E_{i} \right)&=\mu\left(\bigcup_{i=k+1}^{\infty}E_{k} \backslash E_{i} \right)
 \le \sum_{i=k+1}^{\infty}\mu\left(E_{k} \backslash E_{i} \right)\stackrel{\eqref{measure_constant}}{=}0\,,
\end{align*}
as claimed.
\end{proof}

\smallskip

Now, returning to the outline of the proof of Theorem~\ref{bad_0_general}, in view of \eqref{BadR}, the proof of this theorem will be obtained if we can show that the sequence $E_k=\W_{\cR}(\frac1k\Phi)$ satisfies \eqref{measure_constant}, since this sequence is easily seen to be decreasing. The following definition formalises property \eqref{measure_constant} in relation to the sets $\W_{\cR}(\frac1k\Phi)$.

\begin{definition}\label{constant_invariant}\rm
Let $(X, d, \mu)$ be the product space and $\W_{\cR}(\Phi)$ be the limsup set defined in \S\ref{GS}. We will say that $\W_{\cR}(\Phi)$ is \textit{constant invariant} if for any $c>0$
\begin{equation*}
\mu\big(\W_{\cR}(\Phi)\setminus \W_{\cR}(c\Phi)\big)=0\,.
\end{equation*}
\end{definition}

The following statement can be regarded as an `abstract'  generalisation of Theorem~\ref{bad_0_general} and, as explained above, is a direct implication of Lemma~\ref{measure_null}.

\begin{theorem} \label{bad_0}
Let $(X, d, \mu)$ be the product space and $\W_{\cR}(\Phi)$ be the limsup set defined in \S\ref{GS}. Suppose that $\W_{\cR}(\Phi)$ is constant invariant. Then
\begin{equation*}
\mu\left(\Bad_{\cR}(\Phi)\right)=0.
\end{equation*}
\end{theorem}

The proof of Theorem~\ref{bad_0_general} given in this section will boil down to establishing that $\W_{\cR}(\Phi)$ is constant invariant under the conditions stated in these theorems.

The constant invariant condition formulated above was verified in the past for various limsup sets arising in Diophantine approximation. One of the first statements of this type was proven by Cassels in \cite[Lemma 9]{C50}, who showed that in the classical case ($X=\R$, $\cR=\Q$) the measure of $c\phi$-approximable points is unchanged by the choice of constant $c>0$. Since then similar results have been proven in a variety of settings. In \cite{BV08} Beresnevich and Velani proved the $n$-dimensional simultaneous version of Cassel's result, and furthermore showed that such a result held for more general $\limsup$ sets, see \cite[Lemma 1, Lemma 4]{BV08}. In \cite{MR3081774} an $n$-dimensional weighted version of Cassel's result was deduced. In \cite[Lemma 4.1, Lemma 4.2]{BLW21b} an $n$-dimensional weighted version of Cassel's result was proven in an ultrametric setting.

In \S\ref{applications} we will apply Theorem~\ref{bad_0_general} to a range of settings including $S$-arithmetic Diophantine approximation, approximation on fractals, and Diophantine approximation on manifolds. In relation to the latter, our proof enables us to by-pass the intrinsic nature of approximations required by condition (i) of Theorem~\ref{bad_0_general}.

\section{Proof of Theorem~\ref{bad_0_general}}\label{gen2}

\subsection{Auxiliary results}\label{sec3.1}
We start with auxiliary results and definitions used in the proofs. Recall that the measure $\mu$ is called {\em doubling} if there exists $\lambda>0$ and some $r_{0}>0$ such that for any ball $B=B(x,r)$ centred at $x\in\supp \mu$ of radius $0<r<r_{0}$ we have that
\begin{equation*}
\mu(2B) \leq \lambda \;\mu(B),
\end{equation*}
whereby $cB$ means the ball $B$ with radius multiplied by $c$, {\em i.e.} $cB=B(x,cr)$. Throughout, given a ball $B=B(x,r)$, $r(B)$ will stand for the radius $r$ of $B$.

The following result from geometric measure theory is often referred to as the $5r$-covering lemma, see \cite[Theorem 1.2]{Heinonen}.

\begin{lemma}[$5r$-covering lemma]
In any metric space a collection $\cF$ of uniformly bounded balls contains a disjoint subcollection $\cF'$ such that
\begin{equation*}
\bigcup_{B \in \cF} B \subseteq \bigcup_{B \in \cF'}5B.
\end{equation*}
Furthermore for any $B \in \cF$ there exists $B' \in \cF'$ such that $B' \cap B \neq \varnothing$ and $r(B') \geq \frac{1}{2}r(B)$.
\end{lemma}

We will also use Fubini's Theorem, which we recall below in the special case of integrating the characteristic function of a measurable set, see \cite[p. 233]{MR1324786} or \cite[\S2.6.2]{MR0257325}.

\begin{theorem}[Fubini's Theorem] Let $\mu_{1}$ be a $\sigma$-finite measure over $X$ and $\mu_{2}$ be a $\sigma$-finite measure over $Y$. Then $\mu_{1} \times \mu_{2}$ is a regular measure over $X \times Y$. Let $S \subseteq X \times Y$ be a $\mu_{1}\times \mu_{2}$ measurable set. Then for $\mu_1$ almost every $x\in X$ $(\;\mu_2$ almost every $y\in Y\;)$ the set
\begin{align*}
S^{x}:=\{y : (x,y) \in S \}\qquad(\text{resp. }\; S_{y}:=\{x: (x,y) \in S \}\;)
\end{align*}
is $\mu_2$-measurable $($resp. $\mu_1$-measurable$)$, and
\begin{equation*}
(\mu_{1}\times\mu_{2})(S)  = \int_X \mu_{2}(S_{x}) d \mu_{1} = \int_Y \mu_{1}(S^{y}) d \mu_{2}\;.
\end{equation*}
\end{theorem}

\medskip

\subsection{Constant invariant limsup sets: the case of balls}
We begin by generalising \cite[Lemma 9]{C50}, \cite[Lemma 1]{BV08}, \cite[Lemma 4.1]{BLW21b} on $\limsup$ sets of balls in $\R^n$ to any metric space equipped with a $\sigma$-finite Borel regular measure. The proof relies on the observation that for any $\limsup$ set of balls $\{B_{i}\}$ with radii tending to zero, the set $\{B_{i}\}$ is \textit{fine}\footnote{A collection of balls $\cB$ is \textit{fine} on a set $A$ if for any point $x \in A$ and any $r>0$ there exists $B \in \cB$ such that $x \in B$ and $r(B)<r$ i.e. each point has a subsequence of converging balls. To achieve this property one often considers Vitali Spaces, see \cite[Remark 1.13]{Heinonen} for a defintion. Examples of Vitali Spaces include $(\R^{n},\mu)$ equipped with any Radon measure or any metric space equipped with a doubling measure.}   on the limsup set and so a Lebesgue density type statement can be proven. From there a contradiction can be derived using the standard method given in the above mentioned lemmas. The proof follows the ideas presented in the proofs of \cite[Theorem 1.6]{Heinonen} and \cite[Lemma 4.1]{BLW21b}.

\begin{lemma} \label{CI_balls}
Let $(X, d, \mu)$ be a metric measure space equipped with a $\sigma$-finite Borel regular measure $\mu$. Let $(B_{i})_{i \in \N}$ be a sequence of closed balls in $X$ with the radii $r(B_{i}) \to 0$ as $i \to \infty$. Furthermore assume there exists constant $c_{1}>0$ such that for each $i \in \N$
\begin{equation} \label{c_{1}}
\mu(B_{i}) \geq c_{1}\,\mu(5B_{i})\qquad\text{for all $i\in\N$}.
\end{equation}
Let $(U_{i})_{i \in \N}$ be a sequence of $\mu$-measurable subsets of $X$ such that $U_{i} \subset B_{i}$ for each $i \in \N$ and there exists some fixed constant $c_{2}>0$ such that
\begin{equation} \label{c_{2}}
\mu(U_{i}) \geq c_{2}\,\mu(B_{i}).
\end{equation}
Then the difference of the limsup sets
\begin{equation*}
\cB=\limsup_{i \to \infty} B_{i} \qquad \text{and} \qquad \cU=\limsup_{i \to \infty} U_{i}
\end{equation*}
is null, that is
\begin{equation*}
\mu(\cB\setminus\cU)=0\,.
\end{equation*}
\end{lemma}

\begin{proof}
Define
\begin{equation*}
\cU_{j}:=\bigcup_{i \geq j}U_{i} \, , \quad \cD_{j}:=\cB \backslash \cU_{j} \, , \quad \cD:= \bigcup_{j \in \N} \cD_{j}\, .
\end{equation*}
Observe that $\cD=\cB \backslash \bigcap_{j\in\N}\cU_j=\cB \backslash \cU$. Therefore, if $\mu(\cD)=0$ we are done. Assume the contrary. Then, since $\cU_{j} \supseteq \cU_{j+1}$ for all $j \in \N$ we have that $\cD_{j} \subseteq \cD_{j+1}$ for all $j \in \N$ and, by the continuity of $\mu$ and the assumption that $\mu(\cD)>0$, there must exist some $\ell \in \N$ such that
$$
\mu(\cD_{\ell})>0.
$$

Since $\mu$ is $\sigma$-finite there exists a countable collection $(F_{j})_{j \in \N}$ of $\mu$-measurable subsets of $X$ of finite $\mu$-measure such that
\begin{equation}\label{cond0}
\cD_{\ell}\subset \bigcup_{j\in\N}F_j.
\end{equation}
Fix $j \in \N$ and consider the set $\cD_{\ell} \cap F_{j}$. Since $\mu$ is Borel regular, by the outer regularity of $\mu$, for any $\varepsilon_1>0$ there exists an open set $A$ such that
\begin{equation}\label{cond1}
\cD_{\ell} \cap F_{j} \subseteq A \qquad \text{ and } \qquad \mu(A) \leq \mu(\cD_{\ell} \cap F_{j}) + \varepsilon_1\, .
\end{equation}
Observe that $\cD_{\ell} \subseteq \cB\subset\bigcup_{k\ge i}B_k$ for every $i\in\N$. Therefore, since $r(B_i)\to0$ as $i\to\infty$, for every point $x \in \cD_{\ell} \cap F_{j}$ and any $r>0$ there exists some ball $B_{k}$ from the sequence $(B_{i})_{i \in \N}$ such that $r(B_{k})<r$, $k\ge\ell$ and $x \in B_{k}$. In particular, since, by \eqref{cond1}, $x\in A$ and $A$ is open, $B_k$ can be chosen so that $5B_k\subset A$. Thus there exists a subfamily
\begin{equation*}
\cF \subset \{B_{i}:i \in \N,\; i \geq \ell,\, 5B_i\subset A\}
\end{equation*}
 such that
\begin{equation*}
\bigcup_{B_{i} \in \cF} 5B_{i} \subseteq A \qquad \text{ and } \qquad \cD_{\ell} \cap F_{j} \subseteq \bigcup_{B_{i} \in \cF}B_{i}\,.
\end{equation*}
By the $5r$-covering lemma, there exists a disjoint subfamily $\cG \subset \cF$ such that
\begin{equation*}
\bigcup_{B_{i} \in \cG}5B_{i} \supseteq \bigcup_{B_{i} \in \cF} B_{i},
\end{equation*}
and furthermore
\begin{equation}\label{cond2}
\text{$\forall$\; $B_{k} \in \cF$ \;\;$\exists$ $B_{t} \in \cG$ such that $B_{k}\cap B_t\neq\varnothing$ and $r(B_{t}) \geq \tfrac{1}{2}r(B_{k})$.}
\end{equation}

We now show that $\cG$ almost covers $\cD_{\ell} \cap F_{j}$, in the sense that
\begin{equation} \label{cover1}
\mu\left((\cD_{\ell} \cap F_{j}) \backslash \bigcup_{B_{i} \in \cG} B_{i} \right) =0\,.
\end{equation}
Observe that
\begin{equation*}
\sum_{B_{i} \in \cG} \mu(5B_{i}) \; \overset{\eqref{c_{1}}}{\leq}\; \frac{1}{c_{1}}\sum_{B_{i} \in \cG} \mu(B_{i})
= \frac{1}{c_{1}}\mu\left(\bigcup_{B_{i} \in \cG} B_{i}\right) \le \frac{1}{c_{1}}\mu(A) \; \stackrel{\eqref{cond1}}{\le} \;
\frac{1}{c_{1}}\left(\mu(F_j)+\varepsilon_1\right) < \infty\,.
\end{equation*}
This implies that for any $\varepsilon_{2}>0$ there exists $N \in \N_{\geq\ell}$ such that
\begin{equation} \label{5r-sum}
\mu\left( \bigcup_{B_{i} \in \cG: i > N} 5B_{i} \right) \leq \sum_{B_{i} \in \cG: i > N} \mu(5B_{i}) < \varepsilon_{2}\,.
\end{equation}
Considering any point
\begin{equation*}
x \in (\cD_{\ell} \cap F_{j}) \backslash \bigcup_{B_{i} \in \cG : i \leq N}B_{i} \,
\end{equation*}
observe that, since the balls $B_{i}$ are closed, we may choose some ball $B_{k} \in \cF$ such that $x \in B_{k}$ and $B_{k}$ is pairwise disjoint from each ball in the set $\{ B_{i} \in \cG : i \leq N\}$. By \eqref{cond2}, this implies that there must exist some ball, say $B_{t} \in \{B_{i} \in \cG : i >N \}$ such that $B_{k}$ intersects $B_{t}$ and $r(B_{t}) \geq \frac{1}{2}r(B_{k})$. Hence $x \in 5B_{t}$ and therefore
\begin{equation*}
 (\cD_{\ell} \cap F_{j}) \backslash \bigcup_{B_{i} \in \cG : i \leq N}B_{i} \subseteq \bigcup_{B_{i} \in \cG : i>N}5B_{i}\,.
\end{equation*}
Combining this with \eqref{5r-sum} and letting $N \to \infty$ gives \eqref{cover1}.

 Let
 \begin{equation*}
 \lambda=\sup\left\{ \frac{\mu(\cD_{\ell} \cap F_{j} \cap B_{i})}{\mu(B_{i})} \, : B_{i} \in \cG \,  \, \text{ and } \, \mu(B_{i}) >0 \right\}.
 \end{equation*}
 Note that the above set is non-empty in view of the fact that $\mu(\cD_{\ell} \cap F_{j})>0$ and \eqref{cover1}. Then, by \eqref{cover1}, the disjointness of the collection $\cG$ and the definition of $\lambda$, we get that
 \begin{align*}
 \mu(\cD_{\ell} \cap F_{j}) & = \sum_{B_{i} \in \cG} \mu(\cD_{\ell} \cap F_{j} \cap B_{i})
 \; \leq \;\lambda \sum_{B_{i} \in \cG} \mu(B_{i}) \\
 & \leq \lambda \; \mu\left(\bigcup_{B_{i} \in \cG} B_{i}\right)
 \; \leq \;\lambda \; \mu(A)
 \; \stackrel{\eqref{cond1}}{ \leq }\;\lambda\; \big(\mu(\cD_{\ell} \cap F_{j})+\varepsilon_{1}\big)\, .
 \end{align*}
Therefore,
 \begin{equation*}
 \lambda \geq \frac{\mu(\cD_{\ell} \cap F_{j})}{\mu(\cD_{\ell} \cap F_{j})+ \varepsilon_{1}}\,.
 \end{equation*}
Hence, by choosing $\varepsilon_{1}$ small enough, we can guarantee that $\lambda>1-c_{2}$. Hence, by the definition of $\lambda$, there exists $B_{i_{0}} \in \cG$ such that $\mu(B_{i_{0}})>0$ and
 \begin{equation} \label{density_statement}
 \frac{\mu(\cD_{\ell} \cap F_{j} \cap B_{i_{0}})}{\mu(B_{i_{0}})} >1-c_{2}.
 \end{equation}
 However, by the construction of $\cD_{\ell}$, we have that
 \begin{equation}\label{cond3}
 \cD_{\ell} \cap F_{j} \cap U_{i} = \varnothing
 \end{equation}
 for any $i\geq \ell$. Then, on recalling that $B_{i_{0}} \in \cG \subset \cF$ and so $i_{0} \geq \ell$, by \eqref{cond3}, we get that
 \begin{align*}
 \mu(B_{i_{0}}) & \geq \mu(U_{i_{0}} \cap B_{i_{0}}) + \mu(\cD_{\ell} \cap F_{j} \cap B_{i_{0}}) \\
  & \overset{\eqref{c_{2}}}{\geq} c_{2} \;\mu(B_{i_{0}})+ \mu(\cD_{\ell} \cap F_{j} \cap B_{i_{0}}),
  \end{align*}
  and so we conclude that
\begin{equation*}
  \frac{\mu(\cD_{\ell} \cap F_{j} \cap B_{i_{0}})}{\mu(B_{i_{0}})} \leq 1-c_{2}
\end{equation*}
contradicting \eqref{density_statement}. Thus we must have $\mu(\cD_{\ell} \cap F_{j})=0$. Since this is true for each $j \in \N$, by \eqref{cond0}, we conclude that $\mu(\cD_{\ell})=0$, and so $\mu(\cD)=0$.
  \end{proof}

We note below two obvious but useful corollaries.

\begin{corollary}\label{CI_balls_corollary}
Let $(X, d, \mu)$ be a metric measure space equipped with a $\sigma$-finite Borel regular measure $\mu$. Let $(B_{i})_{i \in \N}$ be a sequence of balls in $X$ with $r(B_{i}) \to 0$ as $i \to \infty$. Let $0<c<C$ be given. Suppose that there exist $C'>C$ such that for $(\delta,\rho)= (C,5C')$ and $(\delta,\rho)= (c,C')$ there exists $c(\delta,\rho)>0$ satisfying
\begin{equation} \label{c(delta,rho)}
\mu(\delta B_{i})\geq c(\delta, \rho)\;\mu(\rho B_{i})\qquad\text{for all $i\in\N$}\, .
\end{equation}
Then
\begin{equation}\label{vb77}
\mu\Big(\limsup_{i \to \infty}CB_{i}\;\setminus\;\limsup_{i \to \infty}c B_{i}\Big)=0\,.
\end{equation}
\end{corollary}

\begin{proof}
Let $\overline B_i$ be the closure of $B_i$. (N.B. the balls $B_i$ could be either open or closed.) Note that $C B_{i}\subset C \overline B_{i}$ and hence, \eqref{vb77} will follow if we can show that
\begin{equation}\label{vb77+}
\mu\Big(\limsup_{i \to \infty}C\overline B_{i}\;\setminus\;\limsup_{i \to \infty}c B_{i}\Big)=0\,.
\end{equation}
To deduce the latter we will use Lemma~\ref{CI_balls}. Since $C\overline B_i\subset C'B_i$, taking $(\delta,\rho)= (C,5C')$ and $c_{1}=c(\delta,\rho)^{-1}$ in \eqref{c(delta,rho)} gives \eqref{c_{1}} in relation to $C\overline B_i$, that is
$\mu(C\overline B_{i}) \geq c_{1}\,\mu(5C\overline B_{i})$ for all $i\in\N$.
Further, taking $(\delta,\rho)= (c,C')$ and $c_{2}=c(\delta,\rho)$ in \eqref{c(delta,rho)} verifies \eqref{c_{2}} in relation to $C\overline B_i$, in which $U_i=cB_i$. Thus, by Lemma~\ref{CI_balls}, \eqref{vb77+} holds and we are done.
\end{proof}

\begin{corollary}\label{CI_balls_corollary_doubling}
Let $(X, d, \mu)$ be a metric measure space equipped with a $\sigma$-finite doubling Borel regular measure $\mu$. Let $(B_{i})_{i \in \N}$ be a sequence of balls in $X$ centred at $\supp \mu$ with $r(B_{i}) \to 0$ as $i \to \infty$. Then  \eqref{vb77} holds for any $0<c<C$.
\end{corollary}

\begin{proof}
This immediately follows from Corollary~\ref{CI_balls_corollary}. Indeed, since the centres of the balls $B_i$ are in $\supp \mu$, for any $0<\delta\le\rho$ the doubling property of $\mu$ verifies \eqref{c(delta,rho)} for some $c(\delta,\rho)>0$ depending on $\delta$ and $\rho$ only.
\end{proof}

\medskip

\subsection{Constant invariant limsup sets: the general case}

The following lemma provides an analogue of \cite[Lemma 4]{BV08} for limsup sets of neighborhoods of sets rather than limsup sets of balls (neighborhoods of singletons). Case~(i) of Lemma~\ref{CI_general} below is a generalisation of Corollary~\ref{CI_balls_corollary_doubling}, Case~(ii) of Lemma~\ref{CI_general} is a generalisation of Corollary~\ref{CI_balls_corollary}. Most of the proof follows that of \cite[Lemma 4]{BV08} where the main difference being essentially that Lemma~\ref{CI_balls} is used in place of \cite[Lemma 1]{BV08}. However, there are also additional technicalities arising from the fact that the support of $\mu$ may not be the whole metric space.

\begin{lemma} \label{CI_general}
Let $(X,d,\mu)$ be a metric measure space equipped with a $\sigma$-finite Borel regular measure $\mu$. Let  $(S_{i})_{i \in \N}$ be a sequence of subsets in $X$. Let $(\delta_{i})_{i \in \N}$ be a sequence of positive real numbers such that $\delta_{i} \to 0$ as $i \to \infty$, and let
\begin{equation*}
\Delta(S_{i}, \delta_{i}):=\{ \bx \in X: \dist(S_{i},\bx)< \delta_{i} \},
\end{equation*}
where $\dist(S_{i},\bx):=\inf \{ d(s, \bx) :s \in S_{i} \}$. Let $0<c<C$ be given. Assume either of the following:
\begin{itemize}
\item[{\rm(i)}] $\mu$ is doubling and $S_{i} \subset\supp \mu$ for each $i \in \N$;\\[-1ex]
\item[{\rm(ii)}]
there exist $C'>C$ such that for $(\delta,\rho)= (5c+C,25c+5C')$ and $(\delta,\rho)= (c,5c+C')$ there exists $c(\delta,\rho)>0$ such that for every $i \in \N$ and every $x \in S_{i}$ we have that
\begin{equation*}
\mu(\delta B(x,\delta_{i})) \geq c(\delta,\rho) \mu(\rho B(x, \delta_{i})).
\end{equation*}
\end{itemize}
Then
\begin{equation*}
\mu\left( \limsup_{i \to \infty} \Delta(S_{i}, C\delta_{i}) \; \setminus \; \limsup_{i \to \infty} \Delta(S_{i}, c\delta_{i}) \right)=0\,.
\end{equation*}
\end{lemma}

\begin{proof}
For ease of notation let
\begin{equation*}
A:=\limsup_{i \to \infty} \Delta(S_{i},c\delta_{i}), \quad U:=\limsup_{i \to \infty}\Delta(S_{i},C\delta_{i}).
\end{equation*}
The sets $\Delta(S_{i}, \delta)$ are open and therefore $A$ and $U$ are Borel sets and hence they are $\mu$-measurable. Observe that
\begin{equation} \label{lemma_eq1}
A \subseteq U.
\end{equation}
For each $i$ let $\cG_{i}$ be the collection of open balls $\{B(x, c\delta_{i}):x \in S_{i}\}$. Observe that in Case~(i) $S_{i} \subset \supp \mu$, and so $\cG_{i}$ is a collection of balls with centres in $\supp \mu$. Also $\cG_{i}$ is a $c\delta_{i}$-cover of $\Delta(S_{i}, c\delta_{i})$ and so for any $s \in \Delta(S_{i}, c\delta_{i})$ there exists an $x \in S_{i}$ such that $s \in B(x,c\delta_{i}) \in \cG_{i}$. By the $5r$-covering lemma, there exists a disjoint subcollection of balls, say $\cG'_{i}$, such that
\begin{equation} \label{lemma_eq2}
\bigcup_{B \in \cG'_{i}}5B \supset \bigcup_{B \in \cG_{i}}B=\Delta(S_{i}, c\delta_{i})\supset \bigcup_{B \in \cG'_{i}}B.
\end{equation}

Suppose that $z \in \Delta(S_{i},C\delta_{i})$, then there exists a $y \in S_{i}$ such that $z \in B(y, C\delta_{i})$. Furthermore, by \eqref{lemma_eq2}, there exists a ball $B(x, c\delta_{i}) \in \cG'_{i}$ such that $y \in B(x, 5c\delta_{i})$. By the triangle inequality, we have that $z \in B(x, (5c+C)\delta_{i})$, and so
 \begin{equation} \label{lemma_eq3}
 \Delta(S_{i}, C\delta_{i}) \subseteq \bigcup_{B \in \cG'_{i}}\kappa B\,\qquad\text{where }\kappa=(5c+C)/c\ge6\,.
 \end{equation}
  Hence
  \begin{equation}\label{vb99}
  \limsup_{i \to \infty} \bigcup_{B \in \cG'_{i}}B \overset{\eqref{lemma_eq2}}{\subseteq} A \overset{\eqref{lemma_eq1}}{\subseteq} U \overset{\eqref{lemma_eq3}}{\subseteq} \limsup_{i \to \infty} \bigcup_{B \in \cG'_{i}} \kappa B.
  \end{equation}
Let $\cG''_{i}$ be the subcollection of $\cG'_{i}$ which consists of balls $B\in \cG'_{i}$ with $\mu(\kappa B)>0$. Clearly, the balls in $\cG''_{i}$ are also disjoint. Since the balls $B$ are open, $\kappa B\cap\supp\mu=\varnothing$ if $\mu(\kappa B)=0$. Hence,
$$
\bigcup_{B \in \cG'_{i}}B \cap\supp\mu=\bigcup_{B \in \cG''_{i}}B\cap\supp\mu
$$
and
$$
\bigcup_{B \in \cG'_{i}}\kappa B \cap\supp\mu=\bigcup_{B \in \cG''_{i}}\kappa B\cap\supp\mu\,.
$$
Since $\mu$ is $\sigma$-finite, $\cG''_i$ is disjoint and every ball in $\cG''_i$ has positive $\mu$-measure, the collection $\cG''_i$ is countable. Hence $\cG'':=\bigcup_{i\in\N}\cG''_i$ is a countable collection of balls, say $(\tilde B_i)_{i\in\N}$. Since $\delta_{i} \to 0$ as $i \to \infty$, any ball $B$ can appear in at most a finite number of collections $\cG''_i$. Hence,
$$
\limsup_{i \to \infty} \bigcup_{B \in \cG'_{i}}B \cap\supp\mu=\limsup_{i\to\infty}\tilde B_i\cap\supp\mu
$$
and
$$
\limsup_{i \to \infty} \bigcup_{B \in \cG'_{i}}\kappa B \cap\supp\mu=\limsup_{i\to\infty}\kappa \tilde B_i\cap\supp\mu\,.
$$
Hence, by \eqref{vb99}, we have that
$$
(U\setminus A)\cap\supp\mu\subset (\limsup_{i\to\infty}\kappa \tilde B_i\setminus \limsup_{i\to\infty}\tilde B_i)\cap\supp\mu
$$
and the required result, that is $\mu(U\setminus A)=0$, will follow if we can show that
\begin{equation}\label{vb23}
\mu(\limsup_{i\to\infty}\kappa \tilde B_i\setminus \limsup_{i\to\infty}\tilde B_i)=0\,.
\end{equation}
In Case~(i), by construction, all the balls $\tilde B_i$ are centred at $\supp\mu$ and the measure $\mu$ is doubling. Hence, \eqref{vb23} follows on applying Corollary~\ref{CI_balls_corollary_doubling}. In Case~(ii), \eqref{vb23} follows on applying Corollary~\ref{CI_balls_corollary}.
\end{proof}

\bigskip

In \cite{BLW21b} Fubini's Theorem is applied to \cite[Lemma 4.1]{BLW21b} to obtain an $n$-dimensional result on $\limsup$ sets of `rectangles' (see \cite[Lemma 4.2]{BLW21b} for more details). We follow the same methodology to extend Lemma~\ref{CI_general} to the case of a product space.

\begin{lemma} \label{CI_product}
Let $n\in\N$. For each $1\le j\le n$ let $(X_j,d_{j},\mu_j)$ be a metric measure space equipped with a $\sigma$-finite doubling Borel regular measure $\mu_{j}$. Let $X=\prod_{j=1}^{n}X_{j}$ be the corresponding product space, $d=\max_{1 \leq j \leq n}d_{j}$ be the corresponding metric, and $\mu=\prod_{j=1}^{n}\mu_{j}$ be the corresponding product measure. Let $(S_i)_{i\in\N}$ be a sequence of subsets of $\supp \mu$ and $(\bm\delta_{i})_{i \in \N}$ be a sequence of positive $n$-tuples $\bm\delta_{i}=(\delta^{(1)}_{i},\dots,\delta^{(n)}_{i})$ such that $\delta^{(j)}_{i} \to 0$ as $i \to \infty$ for each $1 \leq j \leq n$. Let
\begin{equation*}
\Delta_n(S_{i}, \bm\delta_{i})=\{\bx \in X: \, \exists \, \ba \in S_{i} \, \, \, d_{j}(a_{j}, x_{j}) < \delta_{i}^{(j)} \, \, \forall \, \, 1 \leq j \leq n\}\,,
\end{equation*}
where $\bx=(x_1,\dots,x_n)$ and $\ba=(a_1,\dots,a_n)$.
Then, for any $\bm C=(C_{1},\dots , C_{n})$ and $\bm c=(c_{1},\dots , c_{n})$ with $0<c_j\le C_{j}$ for each $1 \leq j \leq n$
\begin{equation}\label{vb25}
\mu\left( \limsup_{i \to \infty} \Delta_n(S_{i}, \bm C\bm \delta_{i}) \;\setminus\;\limsup_{i \to \infty} \Delta_n(S_{i}, \bm c\bm\delta_{i}) \right)=0\,,
\end{equation}
where $\bm c\bm\delta_j=(c_1\delta^{(1)}_i,\dots,c_n\delta^{(n)}_i)$ and similarly $\bm C\bm\delta_j=(C_1\delta^{(1)}_i,\dots,C_n\delta^{(n)}_i)$.
\end{lemma}

\begin{proof}
Without loss of generality we will assume that $\bm c=(1,\dots,1)$ as otherwise we can replace $\bm\delta_i$ with $\bm c\bm\delta_i$. Also, since $\mu$ is $\sigma$-finite, without loss of generality we can assume it is finite, and then \eqref{vb25} is equivalent to
\begin{equation}\label{vb26}
\mu\left( \limsup_{i \to \infty} \Delta_n(S_{i}, \bm C\bm \delta_{i})\right)=\mu\left(\limsup_{i \to \infty} \Delta_n(S_{i}, \bm c\bm\delta_{i}) \right)\,.
\end{equation}
Let us begin by showing \eqref{vb26} when $\bm C=(C_{1},1 , \dots , 1)$.
Note that the general case of Lemma~\ref{CI_product} will follow inductively.
For ease of notation let
\begin{equation*}
\hat{\mu}=\prod_{j=2}^{n}\mu_{j},
 \quad \widehat{X}=\prod_{j=2}^{n}X_{j}.
\end{equation*}
For any $y \in \widehat{X}$ let
\begin{equation*}
I_{y}= \{ i\in\N: \, \exists \, \, x \in X_{1} \;\;\text{such that }\, \, (x,y) \in \Delta_n(S_{i},\delta_{i})\},
\end{equation*}
and for any $F \subseteq X$ let $F_y$ denote the fiber of $F$ at $y$, that is
\begin{equation*}
F_y=\{ x: (x,y) \in F\} \subseteq X_{1}.
\end{equation*}
Observe that
\begin{equation}\label{eqn15}
A_y:=\left( \limsup_{i \to \infty} \Delta_n(S_{i}, \bm\delta_{i}) \right)_{y}=\underset{i \in I_y}{\limsup_{i \to \infty}}\Delta_n(S_{i}, \bm\delta_{i})_{y} =:D(y).
\end{equation}
Indeed, if $x \in A_y$ then it implies there exists an infinite sequence $\{i_{k}\}$ such that
\begin{equation*}
(x,y) \in \Delta_n(S_{i_{k}}, \bm\delta_{i_{k}}) \quad \text{ for all } i_k.
\end{equation*}
Hence $\{i_k\} \subseteq I_y$ and so $x \in D(y)$.

Conversely, if $x \in D(y)$ then $D(y)$ is non-empty and so $I_y$ must be infinite. By the definition of $I_y$ and the fact that $x \in D(y)$, we have that $x \in \Delta_n(S_{i}, \bm\delta_{i})_{y}$ for infinitely many $i \in I_y$, and so $x \in A_y$. \par

Similarly, we have that
\begin{equation}\label{eqn16}
\left( \limsup_{i \to \infty} \Delta_n(S_{i},(C_{1}, 1, \dots , 1)\bm\delta_{i}) \right)_{y}=\underset{i \in I_y}{\limsup_{i \to \infty}}\Delta_n(S_{i}, (C_{1},1, \dots,1)\bm\delta_{i})_{y}.
\end{equation}

We are going to apply Lemma~\ref{CI_general}(i) to $D(y)$ with measure $\mu_{1}$. For each $y \in \supp \widehat{\mu}$ and $i \in I_{y}$ let
\begin{equation*}
S_{i}(y):=\Big\{a_{1} : \, (a_1,y)\in\Delta_{n}(\ba, \bm\delta_{i})\text{ for  some }\ba=(a_{1}, \dots , a_{n}) \in S_{i} \Big\}\, .
\end{equation*}
Thus, $S_{i}(y)$ is the projection onto $X_1$ of the part of $S_i$ which is appropriately close to the fibre $X_y=\{(x_1,y):x_1\in X_1\}$. Observe that $\ba=(a_1,\dots,a_n) \in S_{i}$ implies that $a_{1} \in \supp \mu_{1}$ (as $S_{i} \subset \supp \mu$ and $\mu$ is the product measure). Lastly observe that for each $y \in \supp \widehat{\mu}$ and $i \in I_{y}$ we have that $\Delta(S_{i}(y),\delta_{i}^{(1)})=\Delta_{n}(S_{i},\bm\delta_{i})_{y}$ and therefore
\begin{equation}\label{eq:box=rectangle0}
D(y)= \underset{i \in I_y}{\limsup_{i \to \infty}}\Delta(S_{i}(y),\delta_{i}^{(1)})\,.
\end{equation}
Similarly $\Delta(S_{i}(y), C_{1}\delta_{i}^{(1)})= \Delta_n(S_{i}, (C_{1},1, \dots,1)\bm\delta_{i})_{y}$ and therefore
\begin{equation} \label{eq:box=rectangle}
\underset{i \in I_y}{\limsup_{i \to \infty}}\Delta_n(S_{i}, (C_{1},1, \dots,1)\bm\delta_{i})_{y} =\underset{i \in I_y}{\limsup_{i \to \infty}}\Delta(S_{i}(y), C_{1}\delta_{i}^{(1)}).
\end{equation}
By Lemma~\eqref{CI_general}(i), \eqref{eq:box=rectangle0} and \eqref{eq:box=rectangle}, we get that
\begin{equation}\label{vb31}
\mu_1\Big(\underset{i \in I_y}{\limsup_{i \to \infty}}\Delta_n(S_{i}, (C_{1},1, \dots,1)\bm\delta_{i})_{y}\;\setminus D(y)\Big)=0\,.
\end{equation}

Applying Fubini's Theorem we have that
\begin{align*}
\mu \left( \limsup_{i \to \infty} \Delta_n(S_{i},\bm\delta_{i}) \right) & = \int_{\widehat{X}} \mu_{1} \left( A_y \right) d \hat{\mu} \\
& \stackrel{\eqref{eqn15}}{=} \int_{\widehat{X}} \mu_{1} \left( D(y) \right) d \hat{\mu}
\end{align*}
\begin{align*}
&\overset{\eqref{vb31}}{=} \int_{\widehat{X}} \mu_{1} \left( \underset{i \in I_y}{\limsup_{i \to \infty}}\Delta_n(S_{i}, (C_{1},1, \dots,1)\bm\delta_{i})_{y}  \right) d \hat{\mu} \\
& \stackrel{\eqref{eqn16}}{=} \int_{\widehat{X}} \mu_{1} \left( \left( \limsup_{i \to \infty} \Delta_n(S_{i},(C_{1}, 1, \dots , 1)\bm\delta_{i}) \right)_{y} \right) d \hat{\mu} \\
&=\mu \left( \limsup_{i \to \infty}\Delta_n(S_{i},(C_{1},1 , \dots , 1)\bm\delta_{i}) \right)\,.
\end{align*}
The above argument can be repeated $n-1$ more times, for $\ell=2,\dots,n-1$ each time replacing $(C_{1},C_{2}, \dots , C_{\ell-1},1, \dots , 1)$ by $(C_{1}, \dots , C_{\ell-1}, C_{\ell},1, \dots , 1)$ so that at step $\ell$ we use Fubini's theorem with
$$
\widehat{\mu}=\prod_{i\neq\ell}\mu_i\qquad\text{and}\qquad \widehat{X}=\prod_{i\neq\ell}X_i
$$
and get that
\begin{align*}
\mu \left( \limsup_{i \to \infty} \Delta_n(S_{i},(C_{1}, \dots ,C_{\ell-1}, 1, \dots , 1)\delta_{i}) \right) & =\mu \left( \limsup_{i \to \infty} \Delta_n(S_{i},(C_{1}, \dots, C_{\ell}, 1, \dots , 1)\delta_{i}) \right)\,.
\end{align*}
Putting all these equations for $\ell=1,\dots,n$ together completes the proof.
\end{proof}

\smallskip

\begin{remark}
Observe that the above technique cannot produce a version of Lemma~\ref{CI_product} similar to Lemma~\ref{CI_general}(ii). This is simply because in the use of Fubini's Theorem we consider a $\limsup$ set at $\widehat{\mu}$-almost every fibre, and so the centres of the corresponding set, while in $\supp \mu$, may no longer be contained in $\bigcup_{i} S_{i}$.
\end{remark}

\subsection{Proof of Theorem~\ref{bad_0_general}}

Given the above results the proof of Theorem~\ref{bad_0_general} follows readily on recalling the observations we have made in \S\ref{outline}. In particular, Part (i) of Theorem~\ref{bad_0_general} follows from Lemma~\ref{CI_product}, while Part~(ii) of Theorem~\ref{bad_0_general} follows from Lemma~\ref{CI_general}(ii).

\section{Application I: Classical setting}\label{applications}

For the rest of the paper we prove a variety of results on the measure of badly approximable sets. In some cases the result can be obtained by simply applying the results of the previous section to a relevant Dirichlet type result. In other applications it can become more complex. For example, this is the case in simultaneous Diophantine approximation on manifolds, when the approximating points do not lie on the manifold. See \S\ref{S_bad} for more details.

\subsection{$S$-arithmetic Diophantine approximations}\label{S-arith}

Let $S$ be a collection of valuations of $\Q$ with cardinality $k\ge1$. Let $S^{*}=S \backslash \{ \infty \}$. Let $\Qs$ denote the set of $S$-arithmetic points, that is
\begin{equation*}
\Qs:=\prod_{\nu \in S} \Qv,
\end{equation*}
where $\Qv$ is the completion of $\Q$ with respect to the valuation $|\cdot|_{\nu}$. In particular, if $\nu\in S^*$ is a prime number $p$, then $\Q_\nu=\Q_p$ is the field of $p$-adic numbers, and if $\nu=\infty$, then $\Q_\nu=\R$.
Let
\begin{equation*}
\Zs:=\prod_{\nu \in S} \Zv\,,
\end{equation*}
where for $\nu\in S^*$ we have that $\Z_\nu=\{x\in\Q_\nu:|x|_\nu\le1\}$ is the set of $\nu$-adic integers, and $\Z_\infty= [0,1)$. Let $\mu_{S}$ be the $S$-arithmetic Haar measure, normalised by $\mu_{S}(\Zs)=1$. Note that $\mu_{S}$ is simply the product of the measures $\mu_\nu$ over each $\Q_{\nu}$ normalised so that $\mu_\nu(\Z_\nu)=1$ for $\nu\in S^*$ and $\mu_\infty$, i.e.
\begin{equation*}
\mu_{S}=\prod_{\nu \in S} \mu_{\nu}\,,
\end{equation*}
where $\mu_{\infty}=\lambda$, the Lebesgue measure. Similarly, in higher dimensional settings, associated with systems of linear  forms, let $\mu_{S,m\times n}$ denote the Haar measure on $\Qs^{mn}$, normalised by $\mu_{S,m\times n}(\Zs^{mn})=1$. We will denote any $\bx\in\Q_S^{mn}$ as $\bx=(\bx_\nu)$, where $\bx_\nu=(x_{i,j,\nu})\in\Q_\nu^{mn}$, that is $(x_{i,j,\nu})$ is an $m\times n$ matrix over $\Q_\nu$ for each $\nu$.

Let $\Psi=(\psi_{i,\nu})_{1\le i \le n,\nu\in S}$ be a collection of non-negative functions $\psi_{i,\nu}:\Z^{m+n}\to \R$.
Let $W_{S,m,n}(\Psi)$ denote the set of points  $\bx\in \Zs^{mn}$, which will be called $\Psi$-approximable, such that for infinitely many $(\ba,\bb)=(a_1,\dots,a_m,b_1,\dots,b_n) \in \Z^{m+n}$ the following system holds
\begin{equation*}
|a_1x_{i,1,\nu}+\dots+a_mx_{i,m,\nu}-b_{i}|_{\nu} < \psi_{i,\nu}(\ba,\bb) \qquad
(1 \leq i\leq n,\;\nu\in S).
\end{equation*}
Define
\begin{equation*}
\Bad_{S,m,n}(\Psi)=W_{S,m,n}(\Psi)\setminus \bigcap_{\ell\in\N}W_{S,m,n}(\tfrac1\ell\Psi)\,.
\end{equation*}

\begin{theorem}\label{thm4.1}
With reference to the above setup, suppose that for all $1\le i\le n$ and $\nu\in S$.
\begin{equation}\label{vb032}
\psi_{i,\nu}(\ba,\bb)\Vert\ba\Vert_\nu^{-1}\to0\qquad\text{as}\quad\Vert(\ba,\bb)\Vert_\infty\to\infty\,.
\end{equation}
Then
$$
\mu_{S,m\times n}\big(\Bad_{S,m,n}(\Psi)\big)=0\,.
$$
\end{theorem}

\begin{proof}
Consider
$$
X=\prod_{i=1}^n\prod_{\nu\in S}\Z_\nu^m,\qquad \mu=\prod_{i=1}^n\prod_{\nu\in S}\mu_\nu^m,
\qquad R_\alpha=\prod_{i=1}^n\prod_{\nu\in S} R_{\alpha,i,\nu}\,,
$$
where
$$
R_{\alpha,i,\nu}=\left\{ (x_{i,1,\nu},\dots,x_{i,m,\nu})\in \Z_\nu^{m}: a_1x_{i,1,\nu}+\dots+a_m x_{i,m,\nu}-b_{i} =0\right\}\,,\quad\alpha=(\ba,\bb) \in \Z^{m+n}\,,
$$
and
$$
\Phi=(\phi_{i,\nu})_{1\le i \le n,\nu\in S}\,,\quad\text{where }
\phi_{i,\nu}(\ba,\bb)=\psi_{i,\nu}(\ba,\bb)\Vert\ba\Vert_\nu^{-1}\,.
$$
Then it follows from Theorem~\ref{bad_0_general}(i), that $\mu(\Bad_{\cR}(C\Phi))=0$ for any $C>0$. It remains to note that $\Bad_{S,m,n}(\Psi)\subset \Bad_{\cR}(C\Phi)$ for a suitably large constant $C>0$ since we can find absolute constants $c_1,c_2>0$ (depending on $n$, $m$ and $S$) such that for all $i$, $\nu$ and $\alpha=(\ba,\bb)$
$$
\Delta(R_{\alpha,i,\nu},c_1\phi_{i,\nu}(\ba,\bb))\;\subset\;
\{\bx_{i,\nu}\in\Z_\nu^m:|\ba{\bx_{i}}_\nu-b_{i}|_{\nu} <\psi_{i,\nu}(\ba,\bb)\}
\;\subset\; \Delta(R_{\alpha,i,\nu},c_2\phi_{i,\nu}(\ba,\bb))\,.
$$
\end{proof}

Now we specialise the above general result to the setting of weighted Diophantine approximation. We start with a version of Dirichlet's Theorem for $S$-arithmetic approximations.

\begin{lemma}\label{S_arithmetic_Dirichlet_theorem}
Let $n,m \in \N$, $S$ be a finite collection of valuations of $\Q$ and
\begin{equation} \label{omega_definition}
\omega=\begin{cases}
m+n \quad \text{ if } \infty \not \in S, \\
\quad m \quad \, \, \, \, \text{ if } \infty \in S.
\end{cases}
\end{equation}
Let $H_{1}, \dots , H_{\omega} \geq 1$ and $\psi_{i,\nu}\in(0,1)$ for $1\le i\le n$, $\nu\in S$. For $\nu\in S^*$ let $\widetilde\psi_{i,\nu}$ be the largest number of the form $\nu^{-\delta_{i,\nu}}$ with $\delta_{i,\nu}\in\Z$, not exceeding $\psi_{i,\nu}$. Also if $\infty\in S$ let $\widetilde\psi_{i,\infty}=\psi_{i,\infty}$. Suppose that 
\begin{equation} \label{approximation_condition}
\underset{1 \leq i\leq n}{\prod_{\nu \in S}}\widetilde\psi_{i,\nu}\times \prod_{\ell=1}^{\omega}H_{\ell}\ge1.
\end{equation}
Then for any $\bx=(x_{i,j,\nu})\in \Zs^{mn}$ there exists $(\ba,\bb)=(a_{1}, \dots , a_{m},b_1,\dots,b_n)\in\Z^{m+n}\setminus\{\bm0\}$ satisfying
\begin{align}
\label{vb035}
\left|a_{1}x_{i,1,\nu}+ \dots + a_{m}x_{i,m,\nu}-b_{i} \right|_{\nu}&\le\psi_{i,\nu} 
&&(1 \leq i \leq n,\;
\nu \in S^*), \\
\label{vb035B}
\left|a_{1}x_{i,1,\nu}+ \dots + a_{m}x_{i,m,\nu}-b_{i} \right|_{\nu}&<\psi_{i,\nu} 
&&(1 \leq i \leq n,\;\nu=\infty)\quad\text{if $\infty\in S$\,}, \\
\label{integer_condition}
|a_{j}| &\leq H_{j} 
&&(1 \leq j \leq m)\,,\\
\label{integer_condition2}
|b_{i}| &\leq H_{m+i}  
&&(1 \leq i \leq n)\hspace{10.5ex}\text{if $\infty\not\in S$\,.}
\end{align}
\end{lemma}

The proof of this result can be obtained in a standard way using either the pigeonhole principle or Minkowski's first theorem for convex bodies in $\R^k$, and variants of this lemma are well known, see for example \cite{BLW21b, KT07}. For completeness we provide a brief proof.

\begin{proof}
Since $\Z$ is dense in $\Z_\nu$ for any $\nu\in S^*$, we can assume without loss of generality that $x_{i,j,\nu}\in\Z$ for $\nu\in S^*$. Observe that for $\nu\in S^*$,   $\delta_{i,\nu} \in \N$ and satisfies the inequalities
\begin{equation} \label{delta_bounds}
\nu^{\delta_{i,\nu}-1} < \psi_{i,\nu}^{-1} \le \nu^{\delta_{i,\nu}}
\end{equation}
for each $1 \leq i \leq n$. 
Then, \eqref{vb035} is equivalent to 
\begin{equation}\label{vb035+}
a_{1}x_{i,1,\nu}+ \dots + a_{m}x_{i,m,\nu}-b_{i}\equiv 0\bmod{\nu^{\delta_{i,\nu}}}\qquad(1 \leq i \leq n,\quad
\nu \in S^*)\,,
\end{equation}
which is further equivalent,  by the Chinese Remainder Theorem, to
\begin{equation}\label{vb035++}
a_{1}x_{i,1,\nu}+ \dots + a_{m}x_{i,m,\nu}-b_{i}\equiv 0\bmod{\prod_{\nu\in S^*}\nu^{\delta_{i,\nu}}}\qquad(1 \leq i \leq n)\,.
\end{equation}
Let $\Lambda$ be the sublattice of $\Z^{m+n}$ consisting of integer vectors $(\ba,\bb)=(a_{1} \dots , a_{m},b_1,\dots,b_n)$ satisfying \eqref{vb035++}. It is readily seen that the covolume of $\Lambda$ is given by
$$
\mathrm{covol}(\Lambda)=\prod_{i=1}^n\prod_{\nu\in S^*}\nu^{\delta_{i,\nu}}\,.
$$
Suppose first that $\infty\not\in S$ and so \eqref{vb035B} plays no role. Then \eqref{integer_condition} and \eqref{integer_condition2} define a compact convex body $B$ in $\R^{m+n}$ symmetric about the origin of volume
$$
\textrm{vol}(B)=\prod_{\ell=1}^{m+n}2H_\ell=2^{m+n}\prod_{\ell=1}^{\omega}H_\ell\,.
$$
Then, in view of \eqref{approximation_condition} we can use Minkowski's first theorem to conclude that there exists a non-zero point $(\ba,\bb)\in\Lambda\cap B$, which will thus satisfy \eqref{vb035}, \eqref{integer_condition} and \eqref{integer_condition2}, as required.

Now suppose that $\infty\in S$ and so \eqref{integer_condition2} plays no role. Then \eqref{vb035B} and \eqref{integer_condition}, in which we replace $H_1$ by $H_1+\varepsilon$ with $\varepsilon\in(0,1)$, define a convex body $B_\varepsilon$ in $\R^{m+n}$ symmetric about the origin of volume
$$
\textrm{vol}(B_\varepsilon)>\prod_{j=1}^{m}2H_j\times\prod_{i=1}^{n}2\psi_{\infty,i}=2^{m+n} \prod_{\ell=1}^{\omega}H_\ell\prod_{i=1}^{n}\psi_{\infty,i}\,.
$$
Again, in view of \eqref{approximation_condition} we can use Minkowski's first theorem to conclude that there exists a non-zero point $(\ba,\bb)\in\Lambda\cap B_\varepsilon$. Note that the point $(\ba,\bb)$ may generally depend on $\varepsilon$, but since there are only finitely many points in $\Lambda\cap B_\varepsilon$, we can choose the same $(\ba,\bb)\neq\mathbf{0}$ in each set $\Lambda\cap B_\varepsilon$ and on letting $\varepsilon\to0$ we get the  required result, that is $(\ba,\bb)$ will satisfy \eqref{vb035}, \eqref{vb035B} and \eqref{integer_condition}.
\end{proof}

The following corollary provides a version of Dirichlet theorem with weights.

\begin{corollary}\label{cor6.3}
Let $\bt=(\tau_{i,\nu})_{1\le i\le n,\;\nu\in S}$ and $\bm\eta=(\eta_\ell)_{1\le\ell\le \omega}$ be two collections of positive real numbers $($to be referred to as weights$)$ such that
 \begin{equation} \label{S_arithmetic_tau_condition}
 \sum_{i=1}^{n} \sum_{\nu\in S}\tau_{i,\nu}=\omega=\sum_{\ell=1}^\omega \eta_\ell\,.
 \end{equation}
Then for any $\bx=(x_{i,j,\nu})\in \Zs^{mn}$ and any $H>1$ there exists $(\ba,\bb)=(a_{1}, \dots , a_{m},b_1,\dots,b_n)\in\Z^{m+n}\setminus\{\bm0\}$ satisfying\begin{align}
\label{vb042}
\left|a_{1}x_{i,1,\nu}+ \dots + a_{m}x_{i,m,\nu}-b_{i} \right|_{\nu}&\le\nu H^{-\tau_{i,\nu}} 
&&(1 \leq i \leq n,\;
\nu \in S^*)\,\\
\label{vb043}
\left|a_{1}x_{i,1,\nu}+ \dots + a_{m}x_{i,m,\nu}-b_{i} \right|_{\nu}&<H^{-\tau_{i,\nu}} 
&&(1 \leq i \leq n,\;\nu=\infty)\quad\text{if $\infty\in S$\,}\\
\label{integer_condition3}
|a_{j}| &\leq H^{\eta_j} 
&&(1 \leq j \leq m)\,,\\
\label{integer_condition4}
|b_{i}| &\leq H^{\eta_{m+i}}  
&&(1 \leq i \leq n)\hspace{10.5ex}\text{if $\infty\not\in S$\,.}
\end{align}
Furthermore, if $\infty\in S$ or 
\begin{equation}\label{vb044}
\eta_{m+i}<\sum_{\nu\in S}\tau_{i,\nu}\qquad\text{for all $1\le i\le n$}
\end{equation}
then for large enough $H$ for any $(\ba,\bb)\in\Z^{m+n}\setminus\{\bm0\}$ inequalities \eqref{vb042}---\eqref{integer_condition4} imply that $\ba\neq\bm0$.
\end{corollary}

\begin{proof}
The existence of $(\ba,\bb)\in\Z^{m+n}\setminus\{\bm0\}$ satisfying \eqref{vb042}---\eqref{integer_condition4} follows immediately from Lemma~\ref{S_arithmetic_Dirichlet_theorem}. Thus we only need to make sure that $\ba\neq\bm0$ under the conditions stated above. Assume the contrary, that is $\ba=\bm0$. Then, in the case $\infty\in S$, \eqref{vb043} implies that $\bb=\bm0$ which is impossible for $(\ba, \bb)$ must be non-zero. Now, assume that $\infty\not\in S$. Then, by \eqref{vb042}, we get that
$$
\prod_{\nu\in S}|b_i|_\nu\ll H^{-\sum_{\nu\in S}\tau_{i,\nu}}\,,
$$
which implies that $|b_i|\gg H^{\sum_{\nu\in S}\tau_{i,\nu}}$ unless $b_i=0$. Consequently, since $\bb\neq\bm0$, we get that there exists $1\le i\le n$ such that
$$
H^{\sum_{\nu\in S}\tau_{i,\nu}}\ll H^{\eta_{m+i}}
$$
which gets contradictory to \eqref{vb044} for large $H$.
\end{proof}

\medskip

In view of Corollary \ref{cor6.3} it is natural to introduce the following set of $(\bt,\bm\eta)$-badly approximable points 
\begin{equation*}
\Bad_{S,m,n}(\bt,\bm\eta)=\Z_S^{mn}\setminus \bigcap_{\ell\in\N}W_{S,m,n}(\tfrac1\ell,\bt,\bm\eta)\,,
\end{equation*}
where $W_{S,m,n}(c,\bt,\bm\eta)$ consists of $\bx\in \Zs^{mn}$ such that for infinitely many $H>1$ there exists $(\ba, \bb) \in \Z^{m+n}$ satisfying  \eqref{integer_condition3}, \eqref{integer_condition4} and 
\begin{align}
\label{vb042+}
\left|a_{1}x_{i,1,\nu}+ \dots + a_{m}x_{i,m,\nu}-b_{i} \right|_{\nu}&< c H^{-\tau_{i,\nu}} 
&&(1 \leq i \leq n,\;
\nu \in S^*)\,,\\
\label{vb043+}
\left|a_{1}x_{i,1,\nu}+ \dots + a_{m}x_{i,m,\nu}-b_{i} \right|_{\nu}&< c H^{-\tau_{i,\nu}} 
&&(1 \leq i \leq n,\;\nu=\infty)\quad\text{if $\infty\in S$\,}\,.
\end{align}

\medskip

\begin{remark}
We would like to note that in the case $\infty\not\in S$ condition \eqref{vb044} is essentially required to create any interesting theory for $\Bad_{S,m,n}(\bt,\bm\eta)$ since, if
\begin{equation}\label{vb47A}
\eta_{m+i_0}>\sum_{\nu\in S}\tau_{i_0,\nu}\qquad\text{for some $i_0$,}
\end{equation}
then for any $\bx\in\Z_S^{mn}$ and any $c>0$ one can take $\ba=\bm0$ and $\bb$ such that $b_i=0$ if $i\neq i_0$ to construct a non-zero solution $(\ba,\bb)\in \Z^{m+n}$ satisfying \eqref{integer_condition3}, \eqref{integer_condition4}, \eqref{vb042+} and \eqref{vb043+} for any large enough $H$. Thus, if \eqref{vb47A} holds, $\Bad_{S,m,n}(\bt,\bm\eta)$ is simply empty. Thus, we will assume that \eqref{vb044} holds whenever $\infty\not\in S$. We would also like to note that in the case $S=\{p\}$ consists of one prime, by \eqref{S_arithmetic_tau_condition}, condition \eqref{vb044} always holds.
\end{remark}

We will prove that $\Bad_{S,m,n}(\bt,\bm\eta)$ is null under a mild  assumption that will require further definitions. 
Let $\delta\in\R$. Given $H\in\N$,
let $W^\delta_{S,m,n}(\bm\tau,\bm\eta;H)$ be the set of $\bx\in\Z_S^{mn}$ such that there exists $(\ba,\bb)\in \Z^{m+n}\setminus\{\bm0\}$ satisfying \eqref{vb042}---\eqref{integer_condition4} and 
\begin{equation}\label{alphacondition}
\min\{\|\ba\|_\nu:\nu\in S^*\}\le H^{-\delta}\,,
\end{equation}
where 
$\|\ba\|_\nu:=\max\{|a_1|_\nu,\dots,|a_m|_\nu\}$ for $\ba=(a_1,\dots,a_m)\in\Z^m$.
Let $W^\delta_{S,m,n}(\bm\tau,\bm\eta)$ be the set of $\bx\in W^\delta_{S,m,n}(\bm\tau,\bm\eta;H)$ for infinitely many $H\in\N$.

\begin{theorem}\label{thm6.5}
Let $\bm\eta$  satisfy \eqref{S_arithmetic_tau_condition} and, in the case $\infty\not\in S$, \eqref{vb044}. Suppose that 
\begin{equation}\label{assumption}
\mu_{S,m,n}\left(W^\delta_{S,m,n}(\bm\tau,\bm\eta)\right)=0\quad\text{for some }0<\delta<\min\{\tau_{i,\nu}:1\le i\le n,\;\nu\in S^*\}\,.
\end{equation}
Then  $\Bad_{S,m,n}(\bt,\bm\eta)$ has measure $0$.
\end{theorem}

\begin{proof}
Given $(\ba,\bb)\in\Z^{m+n}\setminus\{\bm0\}$, let $H(\ba,\bb)$ be the smallest $H\in\N$ such that \eqref{integer_condition3} and \eqref{integer_condition4} hold. Let $\delta$ be as in \eqref{assumption}. For each $1\le i\le n$, $\nu\in S$ and $(\ba,\bb)\in\Z^{m+n}\setminus\{\bm0\}$ let
$$
\psi_{i,\nu}(\ba,\bb)=\left\{\begin{array}{cl}
   2\nu H(\ba,\bb)^{-\tau_i,\nu}  & \text{if }\nu\in S^*, \&\; \|\ba\|_\nu\ge H(\ba,\bb)^{-\delta}\,, \\[1ex]
   H(\ba,\bb)^{-\tau_i,\nu}  & \text{if }\nu=\infty\in S, \&\; \ba\neq\bm0\,, \\[1ex]
   0  & \text{otherwise}\,.
\end{array}\right.
$$
Then \eqref{vb032} holds. Furthermore, by Corollary~\ref{cor6.3}, we have that 
$$
\Bad_{S,m,n}(\bt,\bm\eta)\subset\Bad_{S,m,n}(\Psi) \cup W^\delta_{S,m,n}(\bm\tau,\bm\eta)\cup\bigcup_{\alpha=(\ba,\bb)\in\Z^{m+n}\setminus\{\bm0\}} R_\alpha\,,
$$
where $R_\alpha$ is the same as in the proof of Theorem~\ref{thm4.1}. Note that $\mu_{S,m,n}(R_\alpha)=0$ whenever $\ba\neq\bm0$. Then, by Theorem~\ref{thm4.1} and and assumption \eqref{assumption}, we complete the proof.
\end{proof}

\begin{remark}
We believe  \eqref{assumption} holds in all scenarios described in the statement of  Theorem~\ref{thm6.5}. In the following proposition we  demonstrate this in the classical case when all the integer parameters have the same bound, that is the weights $\eta_\ell$ are all equal, and $\infty\not\in S$. 
\end{remark}

\begin{proposition}\label{prop6.7}
Let $\bm\eta=(1,\dots,1)$, $\infty\not\in S$ and \eqref{S_arithmetic_tau_condition} and \eqref{vb044} hold. Then for any $\alpha>0$ the set $W^\alpha_{S,m,n}(\bm\tau,\bm\eta)$ has measure zero.
\end{proposition}

\begin{proof}
Define $\widetilde W^\alpha_{S,m,n}(c,\bt,\bm\eta;H)$ to consist of $\bx\in\Z_S^{mn}$ such that there exists $(\ba,\bb)\in \Z^{m+n}$ with $\ba\neq\bm0$ satisfying \eqref{integer_condition3},\eqref{integer_condition4}, \eqref{vb042+}, \eqref{vb043+},  
\begin{equation}\label{alphacondition2}
\min\{\|\ba\|_\nu:\nu\in S^*\}\le cH^{-\delta}\,.
\end{equation}
By Corollary~\ref{cor6.3} and the definitions of the sets in question, for any large enough $H\in [2^{h-1},2^h]$ we have that
$$
W^\delta_{S,m,n}(\bt,\bm\eta;H)\subset \widetilde W^\delta_{S,m,n}(c,\bt,\bm\eta;2^h)
$$
for a suitably chosen fixed constant $c>0$. Therefore, 
$$
W^\delta_{S,m,n}(\bt,\bm\eta)\subset \limsup_{h\to\infty}\widetilde W^\delta_{S,m,n}(c,\bt,\bm\eta;2^h)\,.
$$
Therefore, by the Borel-Cantelli lemma, to complete the proof it suffices to show that the sum over all sufficiently large $h$ of the measures of $\widetilde W^\delta_{S,m,n}(c,\bt,\bm\eta;2^h)$ converges.

Fix $H$ and split the points 
$(\ba,\bb)\in \Z^{m+n}$ with $\ba\neq\bm0$ satisfying \eqref{integer_condition3},\eqref{integer_condition4}
into sub-classes $\cF_{\bm t,J}$, where $\bm t=(t_\nu)_{\nu\in S}$ are non-negative integers and $\varnothing\neq J\subset\{1,\dots,m\}$, such that $(\ba,\bb)\in\cF_{\bm t,J}$ if and only if $\ba$ has non-zero coordinates $a_j$ exactly for $j\in J$,
\begin{equation}\label{vb047N}
2^{-t_\nu-1}<\|\ba\|_\nu\le 2^{-t_\nu}
\end{equation}
and \eqref{vb042+} has a solution $\bx\in\Z_S^{mn}$.
Observe that for a fixed $(\ba,\bb)\in\cF_{\bm t,J}$ the measure of $\bx$ satisfying \eqref{vb042+} is 
\begin{equation}\label{est1}
\ll \prod_{i=1}^n\prod_{\nu\in S}\min\left\{1,\frac{H^{-\tau_{i,\nu}}}{2^{-t_\nu}}\right\}=
\prod_{i=1}^n\prod_{\nu\in S}2^{t_\nu}\min\left\{2^{-t_\nu},H^{-\tau_{i,\nu}}\right\}\,.
\end{equation}
Fix any $(\ba,\bb)\in\cF_{\bm t,J}$. Then for any choice of $(i,\nu)$, by \eqref{vb042+}, we have that
$$
|b_i|_\nu\ll \max\left\{2^{-t_\nu},H^{-\tau_{i,\nu}}\right\}\,.
$$
Hence for each $1\le i\le n$ the number of different values of $b_i$ satisfying the above inequalities over all $\nu\in S$ is
\begin{equation}\label{est2}
\ll 1+ H^{\eta_{m+i}}\prod_{\nu\in S}\max\left\{2^{-t_\nu},H^{-\tau_{i,\nu}}\right\}\,.
\end{equation}
The number of different $\ba$ satisfying \eqref{vb047N} is
\begin{equation}\label{est3}
\prod_{j\in J} \left(H^{\eta_j}\prod_{\nu\in S}2^{-t_\nu}\right)
\end{equation}
and furthermore we must have that
\begin{equation}\label{J}
H^{\eta_j}\prod_{\nu\in S}2^{-t_\nu}\ge1\qquad\text{for each }j\in J\,.
\end{equation}
Then, on combining estimates \eqref{est1}, \eqref{est2} and \eqref{est3} we get that the measure of the set $S_{\bm t,J}$ consisting of $\bx$ such that there exists $(\ba,\bb)\in\cF_{\bm t,J}$ satisfying \eqref{vb042+}, is
\begin{align*}
\mu_{S,m,n}(S_{\bm t,J}) & \ll\prod_{i=1}^n\prod_{\nu\in S}2^{t_\nu}\min\left\{2^{-t_\nu},H^{-\tau_{i,\nu}}\right\} \times \\[1ex]
&\times \prod_{i=1}^n\left(1+ H^{\eta_{m+i}}\prod_{\nu\in S}\max\left\{2^{-t_\nu},H^{-\tau_{i,\nu}}\right\}\right)\times\prod_{j\in J} \left(H^{\eta_j}\prod_{\nu\in S}2^{-t_\nu}\right)\,.
\end{align*}
Since $J\neq\varnothing$, by \eqref{J} and the assumption $\bm\eta=(1,\dots,1)$ we have that
the `$1+$' term in the above expression can be omitted and we get the following upper bound for the measure
$$
\mu_{S,m,n}(S_{\bm t,J})\ll\prod_{i=1}^n\prod_{\nu\in S}H^{-\tau_{i,\nu}} \times \prod_{i=1}^n H^{\eta_{m+i}}\times\prod_{j\in J} H^{\eta_j}\prod_{j\in J}\prod_{\nu\in S}2^{-t_\nu}\le \prod_{j\in J}\prod_{\nu\in S}2^{-t_\nu} \le \prod_{\nu\in S}2^{-t_\nu}\,.
$$
By \eqref{alphacondition2}, and since the number of different subsets $J$ is finite, we get that 
$$
\sum_{J}\sum_{\bm t}\mu_{S,m,n}(S_{\bm t,J} \cap \widetilde W^\delta_{S,m,n}(c,\bt,\bm\eta;H))\ll H^{-\delta}\,.
$$
Since $\widetilde W^\delta_{S,m,n}(c,\bt,\bm\eta;H)$ is the union of the sets $S_{\bm t,J}\cap \widetilde W^\delta_{S,m,n}(c,\bt,\bm\eta;H)$ taken over $J$ and $\bm t$, we get that
$$
\sum_{h=1}^\infty\mu_{S,m,n}(\widetilde W^\delta_{S,m,n}(c,\bt,\bm\eta;2^h))\ll \sum_{h=1}^\infty 2^{-\delta h}<\infty\,.
$$
The Borel-Cantelli Lemma completes the proof.
\end{proof}

\begin{remark}
The case of $\bt=(\tau, \dots , \tau)$ and $\bm\eta=(\eta,\dots,\eta)$ satisfying \eqref{S_arithmetic_tau_condition} was already known, since there exist simultaneous Khintchine-Groshev type theorems in the $S$-arithmetic setting (see for example \cite{L55} for a linear forms version, \cite[Theorem 4]{H10} for a simultaneous $S$-arithmetic version, or \cite{J45} for a $p$-adic version). In the weighted case the $p$-adic result was also already known via Theorem 2.1 of \cite{BLW21b}. As far as we are aware an $S$-arithmetic weighted Khintchine-Groshev type Theorem is yet to be proven, so Theorem~\ref{thm6.5} together with Proposition~\ref{prop6.7} is new. We note that a related Khintchine-Groshev type theorem is probably attainable via a combination of methods introduced in the papers cited above. However, our technique seems much easier and more flexible to use, for example, one could allow the weights $\bt$ and $\bm\eta$ to depend on $H$.
 \end{remark}

\section{Application II: Approximations on Fractals}\label{App2}

In 1984 Mahler \cite{M84} asked how well elements of Cantor sets $\cK$ can be approximated by either
\begin{enumerate}
\item rational points contained in the Cantor set, \label{intrinsic}
\item rational points not in the Cantor set. \label{extrinsic}
\end{enumerate}

Among various answers and extensions to these questions, see for example \cite{BFR11, FS14, LSV07}, it has been shown that the set of badly approximable points intersected with the Cantor set has full Hausdorff dimension on $\cK$ \cite{KW05, KTV06}. Furthermore it was shown that in the simultaneous case that the Hausdorff $\dim \cK$-measure of the badly approximable points intersected with $\cK$ was null \cite[Corollary 1.10]{EFS11}. \par
In this section we consider a slightly different setup. Rather than the set of classical badly approximable points intersected with $\cK$ we consider the set of $(\Psi, \cR)$-badly approximable points where $\cR$ is the collection of rational points contained within $\cK$. That is, we consider Mahler's question of type (1). Our result holds for Cantor sets that are the attractor of a self similar iterated function system. \par

Fix $r \in \N$ and let $\Lambda_{j}$ be a finite collection of mappings $f^{(j)}_{i}:\I \to \I$, for $\I=[0,1]$, of the form
\begin{equation*}
f^{(j)}_{i}(x)=\frac{x+a_{i}}{r} \qquad \text{ with each } \, \, 0 \leq a_{i} \leq r-1
\end{equation*}
for $1 \leq j \leq n$. Let $\cK_{j}$ be the unique attractor of $\Lambda_{j}$, that is
\begin{equation*}
\cK_{j}=\bigcup_{i \in \Lambda_{j}} f^{(j)}_{i}(\cK_{j}).
\end{equation*}
Observe that $\Lambda_{j}$ is a family of contracting similarities, thus $\cK_{j}$ is a self similar set and
\begin{equation*}
\dim \cK_{j}=\frac{\log \#\Lambda_{j}}{\log r}=\gamma_{j},
\end{equation*}
see for example \cite[Theorem 9.3]{F14}. Let $\Lambda=\prod_{j=1}^{n}\Lambda_{j}$ and so for each $\bi=(i_{1}, \dots , i_{n})\in \Lambda$ associate the mapping $\bff_{\bi}:\I^{n} \to \I^{n}$ with
\begin{equation*}
\bff_{\bi}\left( \left(\begin{array}{c} x_{1} \\ x_{2} \\ \vdots \\ x_{n} \end{array}\right) \right)=\left( \begin{array}{c} \frac{x_{1}+a_{i_1}}{r} \\ \frac{x_{2}+a_{i_2}}{r} \\ \vdots \\ \frac{x_{n}+a_{i_n}}{r} \end{array}\right).
\end{equation*}
Let $\cK$ be the attractor of the IFS $\Lambda$ and observe that $\cK=\prod_{j=1}^{n}\cK_{j}$ and
\begin{equation*}
\dim \cK=\frac{\log \#\Lambda}{\log r}=\frac{\log \prod_{j=1}^{n}\#\Lambda_{j}}{\log r}=\sum_{j=1}^{n}\gamma_{j}=\gamma.
\end{equation*}
Let $\mu$ be the unique Borel probability measure supported on $\cK$ satisfying
\begin{equation}\label{mu}
\mu(A)=\sum_{i \in \Lambda} r^{\gamma}\mu\left(\bff_{i}^{-1}(A)\right)
\end{equation}
for all Borel subsets $A \subset \R^{n}$. Call such a measure the canonical self similar measure of $\cK$. Note that $\mu=\prod_{j=1}^{n}\mu_{j}$ where each $\mu_{j}$ is the canonical self similar measure of $\cK_{j}$ since
\begin{align*}
\mu\left(\prod_{j=1}^{n}A_{j}\right)&=\prod_{j=1}^{n}\left( \sum_{i \in \Lambda_{j}}r^{\gamma_{j}}\mu_{j}\left(f_{i}^{-1}(A_{j})\right) \right)\\
&=\sum_{(i_{1}, \dots , i_{n}) \in \Lambda} r^{\sum\gamma_{j}}\mu\left( \prod_{j=1}^{n}f_{i_{j}}^{-1}(A_{j}) \right)\\
&=\sum_{\bi \in \gamma}r^{\gamma}\mu\left(\bff_{\bi}^{-1}\left(\prod_{j=1}^{n}A_{j} \right)\right).
\end{align*}
Lastly, as proven by Mauldin and Urba\'{n}ski \cite{MU96}, since $\Lambda$ is an iterated function system satisfying the open set condition there exists some constant $c>1$ such that for any $\bx \in \cK$ and $0<r<\frac{1}{2}$
\begin{equation*}
c^{-1} \leq \frac{\mu(B(\bx, r))}{r^{\gamma}} \leq c.
\end{equation*}
Thus $\mu$ is doubling.

For an $n$-tuple of approximation functions $\Psi=(\psi_1,\dots,\psi_n)$ let
\begin{equation*}
\W_{\Int}(\cK;\Psi):=\left\{ \bx \in \cK: \left|x_{i}-\frac{p_{i}}{q} \right|<\frac{\psi_{i}(q)}{q} \quad \text{ for i.m. } \, \, \frac{\bp}{q} \in \Q^{n} \cap \cK \right\}.
\end{equation*}
That is, the set of points $\bx \in \cK$ that are $\Psi$-approximable by infinitely many rationals contained in $\cK$. Define 
\begin{equation*}
\Bad_{\Int}(\cK;\Psi):=\W_{\Int}(\cK;\Psi)\setminus \bigcap_{\ell\in \N} \W_{\Int}\left(\cK;\tfrac{1}{\ell}\Psi\right)\,.
\end{equation*}
Thus
$$ \Bad_{\cK}(\Psi)= \left\{ \bx \in \W_{\Int}(\cK;\Psi) : \exists \, c>0 \quad \forall\;\frac{\bp}{q} \in \Q^{n}\cap \cK\quad \max_{1\le i\le n}\frac{q}{\psi_{i}(q)}\left|x_{i}-\frac{p_{i}}{q}\right|\geq c\right\}.$$
\begin{theorem}\label{thm4.7}
For any $n$-tuple of approximation functions $\Psi$  we have that
\begin{equation*}
\mu(\Bad_{\Int}(\cK;\Psi))=0.
\end{equation*}
\end{theorem}

\begin{proof}
This follows easily on applying Theorem~\ref{bad_0_general}(i) with
$$
X=\prod_{i=1}^{n}\cK_{i}\;\;(\text{\em i.e}\;\;X_i=\cK_i),\qquad \mu=\prod_{i=1}^n\mu_i,
$$
$$
R_\alpha=\bp/q\in\Q^n\cap\cK \quad\text{where }
\alpha=(\bp,q) \in \Z^{n}\times\N\,,
$$
and
$$
\Phi=(\phi_{i})_{1\le i \le n}\,,\quad\text{where }
\phi_{i}(\alpha)=\psi_{i}(q)/q\,.
$$
\end{proof}

Now let us specialise the above general result to a simplified setting of weighted Diophantine approximation. We start with a version of Dirichlet's Theorem with weights analogous to \cite[Theorem 2.1]{BFR11}, which we will prove toward the end of this subsection following methodology similar to that of \cite{FS14}.

\begin{theorem}  \label{fractal_dirichlet}
Let $\cK$ be constructed as above. Suppose that $\bt=(\tau_{1}, \dots , \tau_{n}) \in \R^{n}_{+}$ and
\begin{equation} \label{fractal_summation}
\sum_{i=1}^{n}\tau_{i}\gamma_{i}=1.
\end{equation}
Then for any $\bx=(x_{1}, \dots , x_{n}) \in \cK$ and any $Q>r$ there exists $\frac{\bp}{q}=\left(\frac{p_{1}}{q}, \dots , \frac{p_{n}}{q}\right) \in \Q^{n} \cap \cK$ with $q \leq Q$ such that
\begin{equation*}
\left|x_{i}-\frac{p_{i}}{q}\right|< \frac{C}{q(\log Q)^{\tau_{i}}}
\end{equation*}
for some constant $C>0$ independent of $\bx$.
\end{theorem}

\begin{corollary}\label{fractal_dirichlet_cor}
Let $\cK$ be constructed as above, and $\bt=(\tau_{1}, \dots , \tau_{n}) \in \R^{n}_{+}$ satisfy
\eqref{fractal_summation}. Then there exists $C>0$ such that for any $\bx=(x_{1}, \dots , x_{n}) \in \cK\setminus\Q^n$ there exist infinitely many $\frac{\bp}{q}=\left(\frac{p_{1}}{q}, \dots , \frac{p_{n}}{q}\right) \in \Q^{n} \cap \cK$ such that
\begin{equation*}
\left|x_{i}-\frac{p_{i}}{q}\right|< \frac{C}{q(\log q)^{\tau_{i}}}\qquad(1\le i\le n)\,.
\end{equation*}
\end{corollary}

In view of this corollary it then makes perfect sense to introduce the set $\Bad_{\Int}(\cK;\bt)$ of weighted intrinsically badly approximable points in $\cK$ to consist of all $\bx\in\cK$ such that there exists $c(\bx)>0$ such that for all
$\frac{\bp}{q}=\left(\frac{p_{1}}{q}, \dots , \frac{p_{n}}{q}\right) \in \Q^{n} \cap \cK$ we have that
\begin{equation*}
\left|x_{i}-\frac{p_{i}}{q}\right|\ge \frac{c(\bx)}{q(\log q)^{\tau_{i}}}\quad\text{for some } 1\le i\le n\,.
\end{equation*}
That is, for any $\bt$ satisfying \eqref{fractal_summation},
\begin{equation*}
    \Bad_{\Int}(\cK;\bt)=\left\{ \bx \in \cK : \exists \, c(\bx)>0 \quad \left|x_{i}-\frac{p_{i}}{q}\right|\geq c(\bx)\frac{(\log q)^{-\tau_{i}}}{q} \quad \text{for all }\frac{\bp}{q} \in \Q^{n}\cap \cK \right\}\,.
\end{equation*}

By Theorem~\ref{thm4.7} and Corollary~\ref{fractal_dirichlet_cor} we have the following

\begin{corollary}
Let $\bt=(\tau_{1}, \dots , \tau_{n}) \in \R^{n}_{+}$ satisfy
\eqref{fractal_summation}. Then
\begin{equation*}
\mu(\Bad_{\Int}(\cK;\bt))=0\,.
\end{equation*}
\end{corollary}

Now, let us turn to the proof of Theorem~\ref{fractal_dirichlet}. We begin with a slightly altered version of the \textit{fractal pigeonhole principle} (see \cite[Lemma 2.2]{FS14}).
\begin{lemma} \label{weighted_fractal_pigeonhole}
Let $(\bx_{j})_{j=0}^{N}$ be a finite sequence of points in $\cK$ where $\bx_{j}=(x_{j,1}, \dots , x_{j,n})$. Then there exist distinct integers $0 \leq u < v \leq N$ such that
\begin{equation*}
|x_{u,i}-x_{v,i}|<d_{i} \qquad 1 \leq i \leq n
\end{equation*}
for values $d_{i}>0$ satisfying
\begin{equation} \label{sidelenghts_condition}
\prod_{i=1}^{n}d_{i}^{\gamma_{i}} \geq \frac{C_{1}}{N}
\end{equation}
for some constant $0<C_{1}<\infty$ dependent only on $\cK$.
\end{lemma}
\begin{proof}
We prove the contrapositive to Lemma~\ref{weighted_fractal_pigeonhole}. Let the sequence $(\bx_{j})_{j=1}^{N}$ be a $\bd=(d_{1}, \dots , d_{n})$-separated sequence i.e. for each $\bx_{j}=(x_{j,1}, \dots , x_{j,n})$ in the sequence we have that $|x_{u,i}-x_{v,i}| \geq d_{i}$ for all $1 \leq i \leq n$ and all $0 \leq u < v \leq N$. Then the collection of rectangles
\begin{equation*}
\Delta\left(\bx_{j}, \frac{\bd}{2} \right)=\prod_{i=1}^{n}B\left( x_{j,i},\frac{d_{i}}{2} \right),
\end{equation*}
with centres in $(\bx_{j})_{j=1}^{N}$, are disjoint. Thus
\begin{align*}
1=\mu(\cK) & \geq \sum_{j=1}^{N}\mu\left(\Delta\left(\bx_{j}, \frac{\bd}{2} \right) \right) \\
&= \sum_{j=1}^{N} \prod_{i=1}^{n}\mu_{i}\left( B\left(x_{j,i}, \frac{d_{i}}{2} \right) \right) \\
& \asymp \sum_{j=1}^{N}\prod_{i=1}^{n}\left(\frac{d_{i}}{2}\right)^{\gamma_{i}} \\
& \asymp N \prod_{i=1}^{n}d_{i}^{\gamma_{i}}.
\end{align*}
Thus we are forced to conclude there exists constant $C_{1}< \infty$ such that
\begin{equation*}
 N\prod_{i=1}^{n}d_{i}^{\gamma_{i}} < C_{1}.
 \end{equation*}
 \end{proof}

\begin{proof}[Proof of Theorem~\ref{fractal_dirichlet}]
Let $\Lambda^{\N}$ be the collection of all possible sequences of $\bi \in \Lambda$, both infinite and finite. Define $\Lambda^{\N}_{j}$ similarly for each $1 \leq j \leq n$. We can write each $\bsigma \in \Lambda^{\N}$ in two key forms
\begin{equation*}
\bsigma=(\sigma_{1}, \dots , \sigma_{n})=(\bsigma_{1}, \bsigma_{2}, \dots ),
\end{equation*}
where the first sequence is a finite collection of possibly infinite words with each $\sigma_{j}$ belonging to $\Lambda_{j}^{\N}$ respectively. The second form is a possibly infinite collection of finite words $\bsigma_{i}$ each belonging to $\Lambda$. For a sequence $\sigma_{j}=(\sigma_{j,1}, \sigma_{j,2}, \dots ) \in \Lambda^{\N}_{j}$ and integers $r<s$ let
\begin{equation*}
{}^{(r)}\sigma_{j}=(\sigma_{j,r+1}, \sigma_{j, r+2}, \dots), \quad \sigma_{j}^{(s)}=(\sigma_{j,1}, \dots , \sigma_{j,s} ), \quad {}^{(r)}\sigma_{j}^{(s)}=(\sigma_{j,r+1}, \dots , \sigma_{j,s}).
\end{equation*}
If $\sigma_{j}$ is of finite length $t$ and $r>t$ then let ${}^{(r)}\sigma_{j}=0$ where $f_{0}^{(j)}(x)=x$ is the identity map. Adopt similar notation for $\bsigma \in \Lambda^{\N}$, that is
\begin{equation*}
{}^{(r)}\bsigma^{(s)}=(\bsigma_{r+1}, \dots , \bsigma_{s}).
\end{equation*}
Observe that also
\begin{equation*}
{}^{(r)}\bsigma^{(s)}=({}^{(r)}\sigma_{1}^{(s)}, \dots , {}^{(r)}\sigma_{n}^{(s)}).
\end{equation*}
We recall the following functions as defined in \cite{FS14}:
 \begin{itemize}
 \item Define the \textit{shift map} $\omega:\Lambda^{\N} \to \Lambda^{\N}$ by $\omega((\bsigma_{1}, \bsigma_{2}, \dots) )= (\bsigma_{2}, \dots )$, or equivalently $\omega(\bsigma)={}^{(1)}\bsigma$.
 \item Define the \textit{limit point} of a sequence $\bsigma \in \Lambda^{\N}$ as the unique point $\bx \in \cK$ such that
\begin{equation*}
 \lim_{s \to \infty} \bff_{\bsigma^{(s)}}(\bx)=\bx.
 \end{equation*}
 \end{itemize}
 We now prove Theorem~\ref{fractal_dirichlet}. Take $\bx=(x_{1}, \dots x_{n}) \in \cK$ and let $\bsigma \in \Lambda^{\N}$ be the sequence with limit point $\bx$. Fix some $N \in \N$ to be determined later and consider the sequence of points
 \begin{equation*}
 \left\{\bff_{{}^{(d)}\bsigma}(\bx) \right\}_{d=0}^{N}.
 \end{equation*}
 Choose
 \begin{equation*}
 d_{j}=C_{1}^{1/n}N^{-\tau_{j}} \qquad 1 \leq j \leq n
 \end{equation*}
 for some weight vector $\bt=(\tau_{1}, \dots , \tau_{n}) \in \R^{n}_{+}$ satisfying
 \begin{equation*}
 \sum_{j=1}^{n}\tau_{j}\gamma_{j}=1.
 \end{equation*}
 Such choices of $d_{j}$ means that Lemma~\ref{weighted_fractal_pigeonhole} is applicable, so we conclude there exists two integers $0 \leq u < v\leq N$ such that
 \begin{equation*}
 \left|f_{{}^{(u)}\sigma_{j}}^{(j)}(x_{j})-f_{{}^{(v)}\sigma_{j}}^{(j)}(x_{j}) \right|<d_{j} \qquad 1 \leq j \leq n.
 \end{equation*}
The following calculations are essentially the same as those in \cite{FS14}, we include them here for completeness. For now fix some $1 \leq j \leq n$. Observe that
\begin{equation*}
f^{(j)}_{\sigma_{j}^{(u)}} \left( f^{(j)}_{{}^{(u)}\sigma_{j}}(x_{j}) \right)=x_{j}, \qquad f^{(j)}_{{}^{(u)}\sigma_{j}^{(v)}}\left(f^{(j)}_{{}^{(v)}\sigma_{j}}(x_{j})\right)=f^{(j)}_{{}^{(u)}\sigma_{j}}(x_{j}),
\end{equation*}
and that
\begin{align*}
f^{(j)}_{\sigma_{j}^{(u)}}(x_{j})&=\frac{x_{j}+\sum_{i=1}^{u}a_{\sigma_{j,i}}r^{u-i}}{r^{u}}, \\
f^{(j)}_{{}^{(u)}\sigma_{j}^{(v)}}(x_{j})&=\frac{x_{j}+\sum_{i=u+1}^{v}a_{\sigma_{j,i}}r^{v-i}}{r^{v-u}}.
\end{align*}
Note that $f^{(j)}_{{}^{(u)}\sigma_{j}^{(v)}}$ has a unique fixed point, furthermore, it is a rational point of the form
\begin{equation*}
\frac{p_{j}'}{q'}=\frac{\sum_{i=u+1}^{v}a_{\sigma_{j,i}}r^{v-i}}{r^{v-u}-1}.
\end{equation*}
Let
\begin{equation*}
\frac{p_{j}}{q}=f^{(j)}_{\sigma_{j}^{(u)}}\left(\frac{p_{j}'}{q'}\right)=\frac{(r^{v-u}-1)\sum_{i=1}^{u}a_{\sigma_{j,i}}r^{u-i} + \sum_{i=u+1}^{v}a_{\sigma_{j,i}}r^{v-i}}{r^{u}(r^{v-u}-1)}.
\end{equation*}
Note that $\frac{p_{j}}{q} \in \cK_{j}$, since $f^{(j)}_{\sigma_{j}^{(u)}} \left( f^{(j)^{m}}_{{}^{(u)}\sigma_{j}^{(v)}}\left(\frac{p_{j}'}{q'}\right)\right)=\frac{p_{j}}{q}$ for all $m \in \N$, and that $q<r^{v}$. To bound the distance between $x_{j}$ and $\frac{p_{j}}{q}$ observe that
\begin{align*}
\left| f^{(j)}_{{}^{(u)}\sigma_{j}}(x_{j})-\frac{p_{j}'}{q'}\right|&=\left| f^{(j)}_{{}^{(u)}\sigma_{j}^{(v)}}\left(f^{(j)}_{{}^{(u)}\sigma_{j}}(x_{j})\right) - f^{(j)}_{{}^{(u)}\sigma_{j}^{(v)}}\left(\frac{p_{j}'}{q'}\right) \right|\\
&=\frac{1}{r^{v-u}}\left| f^{(j)}_{{}^{(v)}\sigma_{j}}(x_{j})-\frac{p_{j}'}{q'}\right|, \\
& \leq \frac{1}{r^{v-u}}\left( \left| f^{(j)}_{{}^{(u)}\sigma_{j}}(x_{j})-\frac{p_{j}'}{q'} \right| + \left| f^{(j)}_{{}^{(v)}\sigma_{j}}(x_{j})- f^{(j)}_{{}^{(u)}\sigma_{j}}(x_{j})\right| \right),
\end{align*}
which can be evaluated to
\begin{equation*}
\left| f^{(j)}_{{}^{(u)}\sigma_{j}}(x_{j})-\frac{p_{j}'}{q'}\right| \leq \frac{1}{r^{v-u}-1}d_{j}.
\end{equation*}
Applying $f^{(j)}_{\sigma_{j}^{(u)}}$ to the left hand side we get that
\begin{equation*}
\left| x_{j}-\frac{p_{j}}{q}\right|=\frac{1}{r^{u}}\left| f^{(j)}_{{}^{(u)}\sigma_{j}}(x_{j})-\frac{p_{j}'}{q'}\right| <\frac{1}{r^{u}(r^{v-u}-1)}d_{j}=\frac{d_{j}}{q}.
\end{equation*}
Observe by the earlier remark that $q<r^{v}<r^{N}$, so if we choose $N=\lfloor \log_{r}Q \rfloor$, then $q \leq Q$, and
\begin{equation*}
d_{j} =C_{1}^{1/n}N^{-\tau_{j}}\leq C_{1}^{1/n}C_{2}\log(Q)^{-\tau_{j}}
\end{equation*}
for some sufficiently large $C_{2}$. Taking $C=C_{1}^{1/n}C_{2}$ completes the proof along the $j^{th}$ axis. Lastly, note that the above calculation can be completed for each $1\leq j \leq n$, and furthermore that $q$ remains the same in each case since $q$ is dependent only on $r, u$ and $v$.
\end{proof}

\section{Application III: Bad on manifolds}\label{S_bad}

In this section we provide the statements of main results on badly approximable points on manifolds. We will restrict ourselves to weighted  simultaneous approximations and will consider both real and $p$-adic cases, and will also delve into the $S$-arithmetic case. 

{\em Notation.} We will use slightly simplified  notation introduced in \S \ref{S-arith}. Since $m=1$ for the rest of the paper, we will write $W^{(S)}_{n}(\Psi)$ and $\Bad^{(S)}_{n}(\Psi)$ instead of $W_{S,1,n}(\Psi)$ and $\Bad_{S,1,n}(\Psi)$ respectively. Furthermore, in the real case, that is when $S=\{\infty\}$ we will write $W_{n}(\Psi)$ and $\Bad_{n}(\Psi)$ for these sets, and when $S=\{p\}$ consists of one prime, we will write $W^{(p)}_{n}(\Psi)$ and $\Bad^{(p)}_{n}(\Psi)$ for these sets.
When $\Psi=(\psi_{,i,\nu})$ are of the form $\psi_{i,\nu}(q)=q^{-\tau_{i,\nu}}$ for some $\bm\tau=(\tau_{i\nu})_{1\le i\le n,\nu\in S}$, we will write $\Bad^{(S)}_{n}(\bm\tau)$ instead of $\Bad^{(S)}_{n}(\Psi)$, etc. If $S$ contains only 1 element, then we write $\tau_i$ instead of $\tau_{i,\nu}$.

The weights $\bm\tau$ will always be subject to the following conditions, which are \eqref{S_arithmetic_tau_condition} and \eqref{vb044} rewritten for the setup of this section ($m=1$, $\bm\eta=(1,\dots,1)$):
\begin{equation}\label{tau_standard}\sum_{i=1}^n\sum_{\nu\in S}\tau_{i,\nu}=\left\{\begin{array}{cl}
n+1 & \quad\text{if } \infty\notin S,\\
1 & \quad \text{if } \infty \in S\,,
    \end{array}\right.
    \end{equation}
\begin{equation}\label{vb044C}
\qquad\qquad\sum_{\nu\in S}\tau_{i,\nu}>1\qquad\text{for all $1\le i\le n$},\qquad\text{if }\infty\not\in S\,.
\end{equation}

In what follows, $\bff=(\bff_\nu)_{\nu\in S}$ will stand for a collection of maps used to parameterise a manifold in the $S$-arithmetic case. Specifically, $\bff_\nu=(f_{1,\nu},\cdots,f_{m,\nu}):\U_\nu \to \Z^{m}_\nu$ is a $C^2$ map defined on an open subset $\U_\nu\subseteq \Z^{d}_\nu$. We let $\cU=\prod_{\nu\in S} \cU_\nu$, and write $\bx\in\cU$ as $(\bx_\nu)_{\nu\in S}$, where $\bx_\nu\in \cU_\nu$. We write $\bff(\bx)$ for $(\bff_\nu(\bx_\nu))_{\nu\in S}$.
For $\bx \in \U$ let $\ff_\nu(\bx)=(\bx_\nu, \bff_\nu(\bx_\nu))$, and $\ff(\bx):=(\ff_\nu(\bx_\nu))$. Thus $\ff_\nu$ can be viewed as a parameterisation of $\mathcal{M}_\nu=\ff_\nu(\U_\nu)$, and $\ff$ can be viewed as a parameterisation of $\mathcal{M}=\prod_{\nu\in S}\mathcal{M}_\nu$. Here for $\nu=\infty$, $\Z_\nu=(0,1)$.

Our first theorem, which proof will be given in \S\ref{proofR}, concerns the totally real case.

\begin{theorem}\label{bad_realmani}
    Let $\bff=(f_{1}, \dots , f_{m}):\U\subset\R^d \to \R^{m}$ be a $C^2$ map defined on an open subset $\U$, and, as before, $\ff(\bx):=(\bx,\bff(\bx))$. Let $\bt=(\tau_{1}, \dots, \tau_{n}) \in \R^{n}_{+}$ satisfy 
    \begin{equation}\label{maxlessmin} \sum_{i=1}^{n} \tau_{i}=1, \quad \text{ and } \quad \min_{1 \leq i \leq d} \tau_{i} \geq \max_{1 \leq j \leq m} \tau_{d+j}.
    \end{equation} Then
\begin{equation*}
	\mu_{d}\left( \ff^{-1}(\Bad_{n}(\bt)) \right)=0.
	\end{equation*}
\end{theorem}

\begin{remark}\label{rem8.2}
Note that the right had side of  \eqref{maxlessmin} is a parametrisation specific condition which is generally not required, since, by using the Inverse Functions Theorem, one can choose a different parametrisation of the manifold with the parameter space $\cU$ chosen within the $d$ coordinates that correspond to the $d$ maximal weights. This requires an obvious condition on the tangent plane to the manifold and leads to Theorem~\ref{main1}.
\end{remark}

\begin{remark}
Theorem~\ref{bad_realmani} includes the case of affine subspaces, that is the case of affine maps $\bff$. Recently, Huang \cite{Huang22} showed that an affine subspace $\mathcal{L}$ is Khintchine type if and only if $\omega(A)<n$, where $A$ a matrix associated with the parametrisation of  $\mathcal{L}$ and $\omega(A)$ is its Diophantine exponent. In partciular, Huang's theorem implies that the set of badly approximable points on such  affine subspaces is null. Our Theorem~\ref{bad_realmani} shows that for any affine subspace, the set of badly approximable points on this subspace is null, and extends this to the weighted case. One can easily obtain  examples of affine subspaces that are not of Khintchine type but on which the set of (weighted) badly approximable points is null.
\end{remark}

Our second type of applications for Bad on manifolds concerns the $p$-adic setting. The following theorem will be obtained as a consequence of a more general result for $S$-arithmetic approximations, namely  Corollary \ref{bad_manifold_S-adic}.

\begin{theorem} \label{bad_manifold_p-adic}
Let $\bff:\U \to \Zp^{m}$ be a $C^2$ map defined on an open subset $\U\subseteq \Zp^{d}$, and for $\bx \in \U$ let $\ff(\bx)=(\bx, \bff(\bx))$. Suppose that $\bt=(\tau_{1}, \dots , \tau_{n}) \in \R^{n}_{+}$ with each $\tau_{i}>1$,
 \begin{equation}\label{tau_p}
 \sum_{i=1}^{n} \tau_{i}=n+1, \quad \text{ and } \quad \min_{1 \leq i \leq d} \tau_{i} \geq \max_{1 \leq j \leq m} \tau_{d+j}.
 \end{equation}
 Then
\begin{equation*}
\mu_{d}\left( \ff^{-1}(\Bad^{(p)}_{n}(\bt)) \right)=0.
\end{equation*}
\end{theorem}

\medskip

More generally, we consider the $S$-arithmetic case, where $S$ is a finite set that contains more than one prime. In this latter setting the proofs of our results are significantly more complex and in fact the manifolds in question are required to satisfy more restrictive assumptions, namely assumptions on their Diophantine exponents. Consequently, prior to stating our $S$-arithmetic results we give some auxiliary definitions. Although our $S$-arithmetic results will deal with the case $\infty\not\in S$, we introduce definitions in full generality for completeness.

\begin{definition}
    Let $\bt=\bt_\nu=(\tau_{i,\nu}), \tau_{i,\nu}>0 ~\forall \nu\in S, i=1,\cdots,n$, that satisfies Equation \eqref{tau_standard}.
    For any $\by\in \Q_S^n$, we define \begin{equation}
        \omega_{\bt}(\by):=\sup\left\{\omega~:~\max_{i,\nu}\vert b_0 y_{i,\nu}+b_i\vert_\nu^{1/\tau_{i,\nu}}\leq \left\{\begin{aligned}&\frac{1}{\Vert \tilde\bb\Vert_\infty^{\omega}}, \infty\notin S\\
                        & \frac{1}{\vert b_0\vert_\infty^{\omega}}, \infty\in S\end{aligned}\right.
        \text{ for i.m. $\tilde\bb=(b_0,\bb)\in\Z^{n+1}$}\right\}.
    \end{equation}
\end{definition}
\noindent Here, and elsewhere, i.m. stands for `infinitely many'.

\medskip

\begin{definition}\label{exponent_tau}
    Let $S$ be a finite set of valuations that may or may not contain $\infty$. Suppose $n=m+d$, and  $\bt$ be as in \eqref{tau_standard}. We define the Diophantine exponent of the manifold $\cM$, parameterised by a map $\ff$ as described above, relative to the weights $\bm\tau$ as follows
    \begin{equation}\label{exponent_manifolds}
        \omega_{\bt}(\mathcal{M}):=\sup\{\omega ~:~\omega_{\bt}(\ff(\bx))\leq \omega \text{ for } \mu_{S,d} \text{ almost every } \bx\in\U\}.
    \end{equation} 
\end{definition}
\noindent Here $\mu_{S,d}$ is the Haar measure that was defined in \S \ref{S-arith}.

\medskip

\begin{definition}\label{def8.7}
     We call $\mathcal{M}$ to be {\em $\bt$-extremal} if $\omega_{\bt}(\mathcal{M})=1.$ Moreover, we call  $\mathcal{M}$ to be {\em strongly extremal} if it is $\bt$-extremal for any $\bt$ satisfying Equation \eqref{tau_standard}.
\end{definition}

\begin{remark}
Strong extremality is usually associated with the so-called multiplicative form of Diophantine approximations, see for example \cite{KM98}. However, it is also well known that the strong (multiplicative) form of extremality is  equivalent to weighted extremality for arbitrary choice of weights. For example, in the real setting this was established in \cite{MR3346961}. This explains our choice for the terminology introduced in Definition~\ref{def8.7}.
\end{remark}

\begin{remark}
Our goal is to establish that the measure of $\ff^{-1}(\Bad^{(S)}_{n}(\bt))$ is zero. Hence, without loss of generality, we will assume that $\omega_{\bt}(\cM)=1$, since otherwise, $\ff^{-1}(\Bad^{(S)}_{n}(\bt))$ is trivially null.
\end{remark}

Now we can state one of the main theorems that concerns badly approximable vectors on $S$-arithmetic manifolds.
For the rest of the section let $\infty\notin S$, and $\bt\in\R_{+}^{nl}$, where $l=\#S$, satisfying
   Equation \eqref{tau_standard} and,
	\begin{equation}\label{tau1}
	 \min_{1 \leq i \leq d} \tau_{i,\nu} \geq \max_{1 \leq j \leq m} \tau_{d+j,\nu}, \quad \forall \nu\in S.\end{equation}

\begin{theorem}\label{Bad_nonextremal}
   Let $\infty\notin S$, and $\bt\in\R_{+}^{nl}$ satisfy
   Equations \eqref{tau_standard}, \eqref{tau1} and \begin{equation}\label{tau_2}
  \frac{1}{n+1}\sum_{i=1}^n\tau_{i,\nu}<\tau_{j,\nu}\qquad (1\le j\le n,\; \nu\in S). 
	\end{equation}\label{vb39S}.
    For all $\nu\in S$, let $\bt^\nu\in\R_{+}^{n(l-1)}$ be defined as \begin{equation}\label{tau_nu}
\tau^\nu_{i,\sigma}:=\frac{\tau_{i,\sigma}}{1-\frac{\sum_{i=1}^n \tau_{i,\nu}}{n+1}}\qquad (\sigma\in S\setminus \{\nu\},\; 1\le i\le n),
    \end{equation} which by definition satisfies Equation \eqref{tau_standard}.
    Suppose for all $\nu\in S$ such that $\min_{i=1}^n\tau_{i,\nu}\leq 1$ we have that
    \begin{equation}\label{exponent_less101}
        \omega_{\bt^\nu}\left(\prod_{\sigma\in S\setminus\{\nu\}}\mathcal{M}_\sigma\right)<\frac{1-1/(n+1)\sum_{i=1}^n\tau_{i,\nu}}{1-\min_{i=1}^n\tau_{i,\nu}},
    \end{equation}
    then
    \begin{equation*}
	\mu_{S,d}\left( \ff^{-1}(\Bad^{(S)}_{n}(\bt)) \right)=0.
	\end{equation*}
\end{theorem}

    Note Equation \eqref{tau_2} makes the right hand side of Equation \eqref{exponent_less101} strictly greater than $1$. Hence, we have the following corollary.
    \begin{corollary} \label{bad_manifold_S-adic}
	Let $\infty\notin S$. Let $\bt=(\tau_{1,\nu}, \dots , \tau_{n,\nu}) \in \R^{nl}_{+}$ satisfies Equations \eqref{tau_standard}, \eqref{tau1} and \eqref{tau_2}.
  Suppose for each $\nu\in S$ with $\min_{i=1}^n\tau_{i,\nu}\leq 1$ the manifold $\prod_{\sigma\in S\setminus \nu}\mathcal{M}_\sigma$ is strongly extremal.
	Then
	\begin{equation*}
	\mu_{S,d}\left( \ff^{-1}(\Bad^{(S)}_{n}(\bt)) \right)=0.
	\end{equation*}
\end{corollary}

Note that in the case $S=\{p\}$, the weights $\tau_i$ are all $>1$ and thus there is no condition on the manifold in the above corollary, which thus implies Theorem \ref{bad_manifold_p-adic}.

\medskip

 	\begin{remark}
 	Equation \eqref{tau_2} is a weaker assumption than saying that the weights are `equally distributed' over all valuations, i.e., $\sum_{i=1}^n\tau_{i,\nu}=\frac{n+1}{l}$ for all $\nu\in S$, where $S$ contains $l$ many valuations. On the other hand, Equation \eqref{tau_2} implies Equation \eqref{vb044C}. For $S=\{p\}$ both Equations \eqref{tau_2} and \eqref{vb044C} are the same as $\tau_{i}>1.
  $\end{remark}

  \begin{theorem} \label{bad_manifold_S-adic_nonextremal}
	Let $\infty\notin S$ and $\bt=(\tau_{1,\nu}, \dots , \tau_{n,\nu}) \in \R^{nl}_{+}$ satisfy Equation \eqref{tau_standard}, \eqref{tau1} and \eqref{vb044C}. Let $0<\alpha<\min\{\min_{\nu\in S}\min_{i=1}^n\tau_{i,\nu},1\}$ and $\bt^{\nu,\alpha}\in\R_{+}^{nl}$ be such that
 \begin{equation}\label{taunu}
 \begin{aligned}
     &\tau^{\nu,\alpha}_{i,\nu}:=\frac{(\tau_{i,\nu}-\alpha)(n+1)}{n+1- n\alpha},\\
     & \tau^{\nu,\alpha}_{i,\sigma}:=\frac{\tau_{i,\sigma}(n+1)}{n+1- n\alpha} ~ \qquad \forall \sigma\in S\setminus \{\nu\},
      \end{aligned}
 \end{equation}
 which by definition satisfies Equation \eqref{tau_standard}.
Suppose for every $\nu\in S$ such that $\min_{i=1}^n\tau_{i,\nu}\leq 1$ we have that
\begin{equation}\label{exponent_less2}
    \omega_{\bt^{\nu,\alpha}}(\mathcal{M})<1+\frac{\alpha}{(1-\alpha)(n+1)}
\end{equation}
for some $0<\alpha<\min_{i,\nu}\tau_{i,\nu}$. Then
	\begin{equation*}
	\mu_{S,d}\left( \ff^{-1}(\Bad^{(S)}_{n}(\bt)) \right)=0.
	\end{equation*}
\end{theorem}
 As a consequence of the above theorem we have the following corollary.
\begin{corollary} \label{bad_manifold_S-adic_extremal}
	Let $\infty\notin S$.
 Suppose that $\bt=(\tau_{1,\nu}, \dots , \tau_{n,\nu}) \in \R^{nl}_{+}$ satisfying Equation \eqref{tau_standard}, \eqref{tau1} and \eqref{vb044C}.
	Suppose $\mathcal{M}$ is strongly  extremal, then
	\begin{equation*}
	\mu_{S,d}\left( \ff^{-1}(\Bad^{(S)}_{n}(\bt)) \right)=0.
	\end{equation*}
\end{corollary}

\section{Proofs: Bad on $S$-arithmetic manifolds, $\infty\notin S$}\label{proofbad}

In this section we prove Theorems~\ref{Bad_nonextremal} and \ref{bad_manifold_S-adic_nonextremal}, and begin with two auxiliary statements. The first fact one is simply a very well known consequence of the fact that all the functions under consideration are $C^2$, in the real case this is a consequence of Taylor's formula, in the $p$-adic case appropriate theory can be found in \cite{S06}. 

\begin{lemma}\label{DQE}
	Let $F$ be either $\Q_p$ or $\R$. Let $\mathbf{g}: V \to F^m$ be defined on an open subset $V\subset\Qp^{d}$ and suppose that $\mathbf{g}$ is $C^2$ at $\bx\in V$. Then there exists constants $C_{\bx}>0$ and $\varepsilon>0$ and  $\partial_i \mathbf{g}(\bx)\in F$ $(1\le i\le d)$, which are referred to as partial derivatives of $\mathbf{g}$ at $\bx$, such that for any $\by \in B(\bx,\varepsilon)\subset V$
	\begin{equation}\label{eqn007}
	\left\Vert\mathbf{g}(\by)-\mathbf{g}(\bx)-\sum_{i=1}^{d}\partial_i \mathbf{g}(\bx)(y_i-x_i)\right\Vert_{F}< C_{\bx}\max_{1 \leq i \leq d} |y_{i}-x_{i}\vert_{F}^{2}\,,
	\end{equation}
 where $\Vert \cdot\Vert_F$ is the standard norm in $F^m$.
	Furthermore, if $\mathbf{g}$ is $C^2$ on $V$ then $C_{\bx}$ is bounded on any compact subset of $V$.
\end{lemma}

Note that in the ultrametric case  ($F=\Q_p$), by \eqref{eqn007}, we also have
\begin{equation}\label{minus}\vert \mathbf{g}(\by)-\mathbf{g}(\bx)\vert_F\leq \max_{i=1}^d\{\vert \partial_i\mathbf{g}(\bx)\vert_F \vert y_i-x_i\vert_F, C_{\bx}\vert y_i-x_i\vert_F^2\}\,.
\end{equation}

The following lemma is a variant of Minkowski's theorem for systems of linear forms for the $S$-arithmetic setting. For singleton $S$ it can explicitly be found in \cite[Lemma 3.3]{BLW21b} and in general it can be viewed as a version of Lemma~\ref{S_arithmetic_Dirichlet_theorem}, or  Corollary~\ref{cor6.3}.
Throughout the rest of this section we assume that $\infty\notin S.$

\begin{lemma}\label{Minlinear}
Let $\infty\not\in S$ and let $L_{i,\nu}(\bx)$ $(1\le i\le n)$ be linear forms in $\bx=(x_0,\cdots,x_n)\in\Z^{n+1}$, where coefficients of $L_{i,\nu}$ belong to $\Z_\nu$ for all $i,\nu$. Let $\bt=(\tau_{i,\nu})\in\R_{+}^{nl}$ satisfy $\sum_{i}\sum_{\nu}\tau_{i,\nu}=n+1$ and $\sigma=(\sigma_{i,\nu})\in\R^{nl}$ be such that $\sum_{i} \sigma_{i,\nu}=n$ for every $\nu\in S$. Then there exists $H_\sigma$ such that for $H_0,\cdots,H_{n}$ integers with $T^{n+1}:=(H_0+1)\cdots(H_n+1)\geq H_\sigma$ there exists $\bx\in\Z^{n+1}_{\neq\bm0}$ such that
 \begin{equation}\label{linear}
\vert L_{i,\nu}(\bx)\vert_\nu<\nu^{\sigma_{i,\nu}}T^{-\tau_{i,\nu}}\qquad (1\le i\le n,\; \nu\in S)
\end{equation}
and \begin{equation}\label{height}
\vert x_i\vert_\infty\leq H_i\qquad (0\le i\le n)\;.
\end{equation}
\end{lemma}

The proof can be obtained in a standard way using either the pigeonhole principle as in \cite[Lemma 3.3]{BLW21b} 
or Minkowski's theorem for convex bodies as we did in Lemma~\ref{S_arithmetic_Dirichlet_theorem} above. For completeness we provide a brief argument, this time using the pigeonhole principle.

\begin{proof} 
Let $T_\varepsilon=T-\varepsilon$ and
define unique integers $\delta_{i,\nu}$ such that $$
\nu^{\delta_{i,\nu}-1}\leq \nu^{-\sigma_{i,\nu}}T_{\varepsilon}^{\tau_{i,\nu}}<\nu^{\delta_{i,\nu}}\qquad \forall i,\nu\,,
$$
which for sufficiently large $T_{\varepsilon}$ will be positive.
Note that $(L_{i,\nu}(\bx))_{\nu\in S}\in\Z_S$ for every $1\le i\le n$ and $\bx\in\Z^{n+1}$. Split $\Z_S^n$ into the subsets $S(\ba)=\prod_{i,\nu}B(a_i,\nu^{-\delta_{i,\nu}})$, where $\ba=(a_0,\cdots,a_n)\in\Z^{n+1}$ and $0\leq a_i\leq \prod_{\nu\in S} \nu^{\delta_{i,\nu}}.$ Observe that the sets $S(\ba)$ are disjoint and covering $\Z_S^{n}$, and that the numbers of these sets is at most 
$$
\prod_{i,\nu} \nu^{\delta_{i,\nu}}\leq \prod_{\nu\in S}\nu^{1-\sum_{i=1}^n\sigma_{i,\nu}} T_{\varepsilon}^{\sum_{i,\nu}\tau_{i,\nu}}= \left(\prod_{\nu}\nu^{1-n}\right) T_\varepsilon^{(n+1)}<T^{n+1}.
$$ 
On the other hand, there are $T^{n+1}$ different integer points  $\bx\in\Z^{n+1}$ satisfying \eqref{height}. Hence there should be two different integer points $\bx_1\neq \bx_2$ under consideration such that $L(\bx_1), L(\bx_2)$ belong to same set $S(\ba)$. Therefore, for $\bx=\bx_1-\bx_2\neq \mathbf{0}$ we have that
$$\vert L_{i,\nu}(\bx)\vert_\nu\leq \nu^{-\delta_{i,\nu}}<\nu^{\sigma_{i,\nu}}T_{\varepsilon}^{-\tau_{i,\nu}}\qquad \forall i,\nu.$$ Since there are finitely many integer points $\bx$ satisfying \eqref{height}, without loss of generality we can assume that there is an unique $\bx$ which works for all $T_\varepsilon$, $\varepsilon\to 0$.
\end{proof}

The proofs of the results stated in the previous section will require a Dirichlet-type result on manifolds.  In the case  $S=\{p\}$ a result of this ilk appears in \cite[Theorem~7.1]{BLW21b}, albeit it is not sufficiently `flexible' in terms of constants to be used for our purpose. 
Furthermore, our goal is to establish its generalisation for the $S$-arithmetic setting, in which certain complications arise due to working with both $\vert\cdot\vert_p$ (potentially for several primes $p$) and $\vert\cdot\vert_\infty$ norm at the same time.
The complications arise due to the following 

\noindent{\bf Observation}\;\!: Given integers $a_0\neq0$ and $a_1$, the inequality $\vert a_0x-a_1\vert_p<\varepsilon$ implies (in fact equivalent to) $\vert x-\frac{a_1}{a_0}\vert_p< \varepsilon \vert a_0\vert_p^{-1}$. Thus, after `canceling out' $a_0$ the right-hand side in the implied inequality may become much larger than $\varepsilon$, potentially larger than $1$, since $|a_0|_p$ could be very small and so $|a_0|_p^{-1}$ could be very large -- in the extreme case, as large as $|a_0|$. 

This observation gives raise to complications while proving Dirichlet type theorems for manifolds in the $S$-arithmetic setting, where $S$ contains multiple valuations. Worse still, the above observation means that the diameters of the hyper-rectangles arising from a Dirichlet type theorem in the $S$-arithmetic setting may not tend to zero, rendering the use of the general framework of \S\ref{GS} inapplicable. We deal with this obvious obstacle by imposing appropriate conditions on the exponents of the manifolds in question. 

Let $C=C_{\bx}>0$ the constant that satisfy Lemma~\ref{DQE} for all $\bff_{\nu}$ simultaneously. By $a_n\ll b_n$, we mean $a_n\leq K b_n,$ where $K>0$ does not depend on $n.$

\begin{theorem} \label{diri}
	Let $\bff=(\bff_\nu)_{\nu\in S}=(f_{1,\nu}, \dots , f_{m,\nu}):\U:=\prod_{\nu\in S}\U_\nu \to \Zs^{m}$ be a map defined on an open subset $\U$, where $\U_\nu \subseteq \Z_\nu^{d}$ open, $\bx_\nu\in\U_\nu\setminus \Q^d$ and suppose that $\bff_\nu$ is $C^2$ at $\bx_\nu$ for all $\nu\in S$. let $\lambda_\nu$ be given by
	\begin{align}
	\max\left\{1,\,\underset{1 \leq j \leq m}{\max_{1 \leq i \leq d}} \left| \frac{\partial f_{j,\nu}}{\partial x_{i,\nu}}(\bx_\nu) \right|_{\nu}\right\}=\nu^{\lambda_\nu}\,. \label{first_order_constant2}
	\end{align}
 Let $0<\alpha_\nu<\min_{i=1}^n\{\tau_{i,\nu},1\}$ for all $\nu\in S$. Let $c=(c_1,\cdots,c_m)\in\Z^m$, $n=m+d$, and $\bt=(\tau_{1,\nu}, \dots, \tau_{n,\nu}) \in \R^{nl}_{+}$ such that $\bt$ satisfies Equations \eqref{tau_standard}, \eqref{tau1} and \eqref{vb044C} . Then either of the following will happen: 
 \begin{itemize}
 	\item[{\rm\bf Case 1:}] There exist infinitely many $\tilde\bb=(b_0,\cdots,b_n)\in\Z^{n+1}$, such that  $\forall\nu\in S, i=1,\cdots,n$, $\vert b_0\vert_\nu\geq  \Vert\tilde\bb\Vert_\infty^{-\alpha_\nu}$,  and
	\begin{equation}
	\begin{cases}
	\left|x_{i,\nu} -\frac{b_i}{b_0}\right|_{\nu} < \nu^{(n+m\lambda_\nu+\sum_{j=1}^{m}c_{j})/d}  \vert b_0\vert_\nu^{-1}\Vert\tilde\bb\Vert_\infty^{-\tau_{i,\nu}} \qquad\qquad\qquad\quad (1 \leq i \leq d), \\[2ex]
	\left|f_{j,\nu}\left( \frac{b_{1}}{b_{0}}, \dots , \frac{b_{d}}{b_{0}} \right) - \frac{b_{d+j}}{b_0} \right|_{\nu} < \nu^{-c_j} \vert b_0\vert_\nu^{-1} \Vert \tilde\bb\Vert_\infty^{-\tau_{d+j,\nu}} \, \qquad\quad (1 \leq j \leq m), \label{tau_{2}2} \\[2ex]
	\end{cases}
	\end{equation}
	\item[{\rm\bf Case 2:}] There exist some $\nu\in S$ with $\min_{i=1}^n\tau_{i,\nu}\leq 1$ and $K_{\lambda,c}>0$ such that for infinitely many $\tilde\bb=(b_0,\cdots,b_n)\in\Z^{n+1}$ and for all $\sigma\in S\setminus\{\nu\},$ \begin{align}
	\begin{cases}
	|b_{0}x_{i,\sigma}-b_{i}|_{\sigma} &<K_{\lambda,c} \Vert \tilde\bb\Vert_\infty^{-\tau_{i,\sigma}/(1-\alpha_\nu)} \quad (1 \leq i \leq d), \\[1.5ex]
	\displaystyle\left| b_{0}f_{j,\sigma}(\bx_\sigma)-b_{d+j} \right|_{\sigma} &< K_{\lambda,c}\Vert \tilde\bb\Vert_\infty^{-\tau_{d+j,\sigma}/(1-\alpha_\nu)} \hspace*{8ex} (1 \leq j \leq m), \label{excep2thm}
	\end{cases}
	\end{align}
	and
	\begin{align}
	\begin{cases}
	|b_{0}x_{i,\nu}-b_{i}|_{\nu} &< K_{\lambda,c} \Vert \tilde\bb\Vert_\infty^{-(\tau_{i,\nu}-\alpha_\nu)/(1-\alpha_\nu)} \quad (1 \leq i \leq d), \\[1.5ex]
	\displaystyle\left| b_{0}f_{j,\nu}(\bx_\nu)-b_{d+j} \right|_{\nu} &< K_{\lambda,c}  \Vert \tilde\bb\Vert_\infty^{-(\tau_{d+j,\nu}-\alpha_\nu)/(1-\alpha_\nu)} \quad  (1 \leq j \leq m), \hspace*{8ex} \label{excep_11}.
	\end{cases}
	\end{align}
	
	where $K_{\lambda,c}>0$ is a constant that depend on $c=(c_j), \lambda=(\lambda_\nu).$
	
	\end{itemize}
\end{theorem}

\begin{proof}
 Let $0<\varepsilon<1$ be the constant that satisfy Lemma~\ref{DQE} for all $\bff_{\nu}$ simultaneously. In particular, we have that $B(\bx,\varepsilon) \subseteq \prod_{\nu\in S}\U_\nu$.
	By Lemma \ref{Minlinear} with $\sigma_\nu=(\nu,\cdots,\nu)$,
 $H_0=\dots=H_n=H$ and $T=H+1$, for any integer $H \ge H_{\sigma}^{1/(n+1)}$, and $\nu\in S$ the following system
	\begin{align}
	\begin{cases}
	|b_{0}x_{i,\nu}-b_{i}|_{\nu} &< \nu^{(n+m\lambda_\nu+\sum_{j=1}^{m}c_{j})/d} H^{-\tau_{i,\nu}} \quad (1 \leq i \leq d), \\[1.5ex]
	\displaystyle\left| b_{0}f_{j,\nu}(\bx_\nu)-\sum_{i=1}^{d}\partial_i f_{j,\nu}(\bx_\nu) \left( b_0x_{i,\nu}-b_{i} \right)-b_{d+j} \right|_{\nu} &< \nu^{-c_j} H^{-\tau_{d+j,\nu}} \hspace*{8ex} (1 \leq j \leq m), \label{ineq2D}\\[1.5ex]
	\underset{0\leq i\leq n}{\max} |b_{i}|_\infty  & \leq H
	\end{cases}
	\end{align}
	has a non-zero integer solution $\tilde\bb:=(b_{0},b_{1},\dots, b_{n}) \in \Z^{n+1}$.

 Let $\tau_{\min,\nu}:=\min_{1 \leq i \leq d}\tau_{i,\nu}$, $\tau_{\max,\nu}:=\max_{1 \leq j \leq m} \tau_{d+j,\nu}, \tau_{\min,\nu}^\star=\min_{1\leq i\leq n}\tau_{i,\nu}$. The rest of the proof will be represented by a series of claims that we shall establish one-by-one.

 \medskip

\textbf{Claim A:} We claim that $b_0\neq 0$ in Equation \eqref{ineq2D}. Suppose the contrary, that is $b_{0}=0$. Then by the first inequality of \eqref{ineq2D} we have that $|b_{i}|_{\nu}<\nu H^{-\tau_{i,\nu}}$
for $i=1\cdots,d$ and $\forall \nu\in S$. Therefore, $\prod_{\nu\in S}\vert b_i\vert_\nu<H^{-1}$ by Equation \eqref{vb044C}. As $|b_{i}|_\infty \leq H$ and $H>H_{0}$, we have that $b_{i}=0$ for $1 \leq i \leq d$. Considering the second set of inequalities of \eqref{ineq2D}, for each $1 \leq j \leq m$ we have that $|b_{d+j}|_{\nu}<\nu H^{-\tau_{d+j,\nu}}$ for all $\nu\in S$ which yields $b_{d+j}=0$ for each $1 \leq j \leq m$, by Equation \eqref{vb044C}. Thus $(b_{0}, b_{1}, \dots, b_{n})=\boldsymbol{0}$, a contradiction. So we must have that $b_{0}\neq 0$.
	

\medskip

 \textbf{Claim B:} We verify that there are infinitely many $\tilde\bb$ that appear in Equation \eqref{ineq2D} for $\bx, $ with $\bx_\nu\notin \Q^d$. Suppose the contrary. Since the right-hand side of Equation \eqref{tau_{2}2} goes to $0$ as $H\to\infty$, the left-hand side is fixed. Hence $ b_0x_{i,\nu}-b_i=0$ for all $i=1,\cdots,d,$ which is a contradiction to the fact that $\bx_\nu\notin\Q^d.$

Hence there are infinitely many $\tilde\bb\in\Z^{n+1}$ such that
\begin{align}
	\begin{cases}
	|b_{0}x_{i,\nu}-b_{i}|_{\nu} &< \nu^{(n+m\lambda_\nu+\sum_{j=1}^{m}c_{j})/d} \Vert \tilde\bb\Vert_\infty^{-\tau_{i,\nu}} \; (1 \leq i \leq d), \\[1.5ex]
	\displaystyle\left| b_{0}f_{j,\nu}(\bx_\nu)-\sum_{i=1}^{d}\partial_i f_{j,\nu}(\bx_\nu) \left( b_0x_{i,\nu}-b_{i} \right)-b_{d+j} \right|_{\nu} &< \nu^{-c_j} \Vert \tilde\bb\Vert_\infty^{-\tau_{d+j,\nu}} \hspace*{8ex} (1 \leq j \leq m), \label{ineq22}
	\end{cases}
	\end{align}

\noindent\textbf{Case 1:} Suppose for arbitrarily large $H$, there exists $\tilde\bb$ satisfying Equation \eqref{ineq2D} (and hence Equation \eqref{ineq22}) are such that for all $\nu\in S$, \begin{equation}\vert b_0\vert_\nu\geq \Vert \tilde\bb\Vert_\infty^{-\alpha_\nu}.\end{equation}
The same reasoning as in Claim B says there are infinitely many such $\tilde\bb.$
Then we get that for all $\nu\in S,$
	$$
	\left( \frac{b_{1}}{b_{0}}, \dots , \frac{b_{d}}{b_{0}} \right)\in B(\bx_\nu,\varepsilon)\subseteq  \U_\nu,
	$$
	since $\alpha_\nu<\tau_{i,\nu}$ for all $\nu\in S, i=1,\cdots,n$.  Thus $f_{j,\nu}\left(\frac{b_{1}}{b_{0}}, \dots , \frac{b_{d}}{b_{0}}\right)$ is well defined.
	
	Using the fact that each $f_{j,\nu}$ is $C^2$ at $\bx_\nu$, by Lemma~\ref{DQE}, we get that for  infinitely many $\tilde\bb$ appearing in Equation \eqref{ineq22},
	\begin{align} \label{f_tau2}
	\left|\left( f_{j,\nu} \left(\frac{b_{1}}{b_{0}}, \dots , \frac{b_{d}}{b_{0}} \right)-f_{j,\nu}(\bx_\nu)-\sum_{i=1}^d \partial_i f_{j,\nu}(\bx_\nu) \left(\frac{b_{i}}{b_{0}}-x_{i,\nu} \right)\right) \right|_{\nu} & < C\max_{1 \leq i \leq d} \left| \frac{b_{i}}{b_{0}}-x_{i,\nu} \right|_{\nu}^{2}\nonumber\\
 &<\nu^{-c_{j}}\frac{1}{\vert b_0\vert_\nu}\Vert \tilde\bb\Vert_\infty^{-\tau_{d+j,\nu}}
	\end{align}
	for each $1 \leq j \leq m$.

	The last inequality follows since
	\begin{align*}
	C\max_{1 \leq i \leq d} \left| \frac{b_{i}}{b_{0}}-x_{i} \right|_{\nu}^{2} & \;\stackrel{\eqref{ineq22}}{<}\; C \nu^{(2n+2m\lambda+ 2\Sigma_{c})/d} \frac{1}{\vert b_0\vert_\nu^2}\Vert \tilde\bb\Vert_\infty^{-2\tau_{\min,\nu}}\\
	& = \frac{1}{\vert b_0\vert_\nu} C \nu^{(2n+2m\lambda + 2\Sigma_{c})/d}\frac{1}{\vert b_0\vert_\nu \Vert \tilde\bb\Vert_\infty^{\tau_{min,\nu}}}\Vert \tilde\bb\Vert_\infty^{-\tau_{\min,\nu}}\\[1ex]
	& \stackrel{\star}{\le}\frac{1}{\vert b_0\vert_\nu} \nu^{-c_{j}}\Vert \tilde\bb\Vert_\infty^{-\tau_{\max,\nu}} \le \nu^{-c_{j}}\frac{1}{\vert b_0\vert_\nu}\Vert \tilde\bb\Vert_\infty^{-\tau_{d+j,\nu}}\,.
	\end{align*}
Here $\star$ follows from the fact that $\vert b_0\vert_\nu\Vert \tilde\bb\Vert_\infty^{\tau_{i,\nu}}\geq \Vert \tilde\bb\Vert_\infty^{-\alpha_\nu+\tau_{i,\nu}}\to \infty$ as $\Vert \tilde\bb\Vert_\infty\to\infty$ since $\alpha_\nu<\tau_{i,\nu}.$
	
	Combining \eqref{ineq22} and \eqref{f_tau2} we obtain for all $\nu\in S$ and $j=1,\cdots,m$,
	\begin{equation*}
	\left| f_{j,\nu} \left( \frac{b_{1}}{b_{0}}, \dots , \frac{b_{d}}{b_{0}} \right)-\frac{b_{d+j}}{b_0} \right|_{\nu}\leq  \nu^{-c_j}\vert b_0\vert_\nu^{-1}\Vert \tilde\bb\Vert_\infty^{-\tau_{d+j,\nu}}.
	\end{equation*}
	for each $1 \leq j \leq m$. This verifies the second set of inequalities in \eqref{tau_{2}2}, while the first set of inequalities in \eqref{tau_{2}2} follow directly from the first set of inequalities in  \eqref{ineq22}.

\medskip
 
	\textbf{Claim C:}
We claim that if for some $\nu\in S$, $\min_{i=1}^n\tau_{i,\nu}>1$ then there are infinitely many $\tilde\bb$ satisfying Equation \eqref{ineq22}, such that $\vert b_0\vert_\nu=1$.
From Equation \eqref{ineq2D}, infinitely many $\tilde\bb\in\Z^{n+1}$
$$\vert b_i\vert_\nu\leq \max\{\nu^{(n+m\lambda_\nu+\sum_{j=1}^{m}c_{j})/d} \Vert \tilde\bb\Vert_\infty^{-\tau_{i,\nu}},  \vert b_0\vert_\nu^{-1}\}=\vert b_0\vert_\nu^{-1}, i=1,\cdots,d.$$ Also,
$$\vert b_{d+j}\vert_\nu\leq \max_{i=1}^d\{\nu^{\lambda_\nu}\nu^{(n+m\lambda_\nu+\sum_{j=1}^{m}c_{j})/d} \Vert \tilde\bb\Vert_\infty^{-\tau_{i,\nu}}, \nu^{-c_j}\Vert \tilde\bb\Vert_\infty^{-\tau_{d+j,\nu}}, \vert b_0\vert_\nu^{-1}\}=\vert b_0\vert_\nu^{-1}, j=1,\cdots, m.$$
Hence we can construct $\tilde\bb'=\vert b_0\vert_\nu\tilde\bb\in\Z^{n+1}$. Then Equation \eqref{ineq22} is satisfied by $\tilde\bb'.$ This is because
\begin{equation}\label{ineq222}
	\begin{cases}
	|b'_{0}x_{i,\nu}-b'_{i}|_{\nu} < \nu^{(n+m\lambda_\nu+\sum_{j=1}^{m}c_{j})/d} \vert b_0\vert_\nu^{\tau_{i,\nu}-1} \Vert \tilde\bb'\Vert_\infty^{-\tau_{i,\nu}} \quad (1 \leq i \leq d), \\[1.5ex]
	\displaystyle\left| b'_{0}f_{j,\nu}(\bx_\nu)-\sum_{i=1}^{d}\partial_i f_{j,\nu}(\bx_\nu) \left( b'_0x_{i,\nu}-b'_{i} \right)-b'_{d+j} \right|_{\nu} \!\!< \nu^{-c_j} \vert b_0\vert_\nu^{\tau_{d+j,\nu}-1} \Vert \tilde\bb'\Vert_\infty^{-\tau_{d+j,\nu}} \hspace*{1.5ex} (1 \leq j \leq m)
	\end{cases}
	\end{equation}
and $\vert b_0\vert_\nu^{\tau_{i,\nu}-1}<1,$ for $i=1,\cdots,n$. And for any $\sigma\in S\setminus\{\nu\},$ since $\vert\vert b_0\vert_\nu\vert_\sigma=1,$ the inequalities in Equation \eqref{ineq22} are satisfied by $\tilde{\bb'}$. Also, note that on the right-hand side of Equation \eqref{ineq222}, $\vert b_0\vert_\nu^{\tau_{i,\nu}-1}\Vert \tilde\bb'\Vert_\infty^{-\tau_{i,\nu}}=\vert b_0\vert_\nu^{-1} \Vert \tilde\bb\Vert_\infty^{-\tau_{i,\nu}}\leq \Vert \tilde\bb\Vert_\infty^{1-\tau_{i,\nu}}\to 0 $ as $\Vert \tilde\bb\Vert_\infty\to\infty.$ Hence there are infinitely many $\tilde{\bb}'$. Also note that $\vert b_0'\vert_\nu=1\geq \frac{1}{\Vert\tilde\bb'\Vert_\infty}\geq \frac{1}{\Vert\tilde\bb'\Vert^{\alpha_\nu}_\infty}.$

\medskip

\noindent\textbf{Case 2: } We consider the opposite situation to Case 1. This means all large $H\geq 1$, there exists some $\nu\in S$ such that $\vert b_0\vert _\nu<\Vert \tilde\bb\Vert_\infty^{-\alpha_\nu}$  for all $\tilde\bb\in\Z^{n+1}$ satisfying Equations \eqref{ineq2D}. Since $S$ is finite, this means for some $\nu\in S$, $\vert b_0\vert _\nu<\Vert \tilde\bb\Vert_\infty^{-\alpha_\nu}$  for all $\tilde\bb\in\Z^{n+1}$ appearing in Equations \eqref{ineq22}. 
 Using the conclusion of Claim C, $\min_{i=1}^n\tau_{i,\nu}\leq 1$. Then using the inequalities from Equation \eqref{ineq22}, we have
	$$\vert b_i\vert_\nu\ll \max\left\{\frac{1}{\Vert \tilde\bb\Vert_\infty^{\min_{i}\tau_{i,\nu}}},\frac{1}{\Vert \tilde\bb\Vert_\infty^{\alpha_\nu}}\right\}, \quad i=1,\cdots,n.$$
	Hence $\vert b_i\vert_\nu\leq K^\nu_{\lambda, c} \Vert \tilde\bb\Vert_\infty^{-\alpha_\nu} ~~~\forall i=1,\cdots,n$ for some $K^{\nu}_{\lambda,c}>0$.
 
 Now suppose $k$ is the unique integer such that $\nu^k\leq \frac{\Vert \tilde\bb\Vert_\infty^{\alpha_\nu}}{K^\nu_{\lambda, c} }\leq \nu^{k+1}.$ Since $\nu^k\leq \vert b_i\vert_\nu^{-1}$, we take $b_i':=\nu^{-k}b_i\in \Z$ for $i=1,\cdots,n.$ Then 
 $$
	\Vert \tilde\bb'\Vert_\infty\leq \nu K^\nu_{\lambda,c}\Vert \tilde\bb\Vert^{1-\alpha_\nu}_\infty.$$ Hence, for all $\sigma\in S\setminus\{\nu\}$ we have \begin{align}
	\begin{cases}
	|b'_{0}x_{i,\sigma}-b'_{i}|_{\sigma} &< K_{\lambda,c} \Vert \tilde\bb'\Vert_\infty^{-\tau_{i,\sigma}/(1-\alpha_\nu)} \quad (1 \leq i \leq d), \\[1.5ex]
	\displaystyle\left| b'_{0}f_{j,\sigma}(\bx_\sigma)-\sum_{i=1}^{d}\partial_i f_{j,\sigma}(\bx_\sigma) \left( b'_0x_{i,\sigma}-b'_{i} \right)-b'_{d+j} \right|_{\sigma} &< K_{\lambda,c}  \Vert \tilde\bb'\Vert_\infty^{-\tau_{d+j,\sigma}/(1-\alpha_\nu)} \quad  (1 \leq j \leq m), \label{excep_16}
	\end{cases}
	\end{align}
	
and

\begin{align}
\begin{cases}
|b'_{0}x_{i,\nu}-b'_{i}|_{\nu} &\!\!\!\!< K_{\lambda,c} \Vert \tilde\bb'\Vert_\infty^{-(\tau_{i,\nu}-\alpha_\nu)/(1-\alpha_\nu)} \quad (1 \leq i \leq d), \\[1.5ex]
\displaystyle\left| b'_{0}f_{j,\nu}(\bx_\nu)-\sum_{i=1}^{d}\partial_i f_{j,\nu}(\bx_\sigma) \left( b'_0x_{i,\nu}-b'_{i} \right)-b'_{d+j} \right|_{\nu} &\!\!\!\!< K_{\lambda,c}  \Vert \tilde\bb'\Vert_\infty^{-(\tau_{d+j,\nu}-\alpha_\nu)/(1-\alpha_\nu)} \;\;  (1 \leq j \leq m), \label{excep_17}
\end{cases}
\end{align}
	where $K_{\lambda,c}>0$ is a constant that depend on $\lambda,c$.
	
\medskip

\textbf{Claim D:} We claim that there are infinitely many $\tilde\bb'$ as constructed above. Note that
	$\vert b_i'\vert_\infty= \frac{\vert b_i\vert_\infty}{\nu^k}\asymp \frac{\vert b_i\vert_\infty}{\Vert \tilde\bb\Vert_\infty^{\alpha_\nu}}$ for all $i=1,\cdots,n$.
	This implies
	$$
	\Vert \tilde\bb'\Vert_\infty=\max_{i=0}^n\vert b_i'\vert_\infty\asymp \frac{\max_{i=0}^n\vert b_i\vert_\infty}{\Vert \tilde\bb\Vert_\infty^{\alpha_\nu}}=\Vert \tilde\bb\Vert_\infty^{1-\alpha_\nu}.
	$$
	Since $\alpha_\nu<1$
	and in Claim B we showed there are infinitely many $\tilde\bb$, so infinitely many $\tilde\bb'$ is guaranteed.
	
Lastly by Equation \eqref{tau_standard}, there are infinitely many  $\tilde\bb'=(b_0',\cdots,b_n')\in\Z^{n+1}$ such that
	for all $\sigma\in S\setminus\{\nu\}$,  \begin{align}
	\begin{cases}
	|b'_{0}x_{i,\sigma}-b'_{i}|_{\sigma} &<K_{\lambda,c} \Vert \tilde\bb'\Vert_\infty^{-\tau_{i,\sigma}/(1-\alpha_\nu)} \quad (1 \leq i \leq d), \\[1.5ex]
	\displaystyle\left| b'_{0}f_{j,\sigma}(\bx_\sigma)-b'_{d+j} \right|_{\sigma} &< K_{\lambda,c}\Vert \tilde\bb'\Vert_\infty^{-\tau_{d+j,\sigma}/(1-\alpha_\nu)} \hspace*{8ex} (1 \leq j \leq m), \label{excep2}
	\end{cases}
	\end{align}
and

\begin{align}
\begin{cases}
|b'_{0}x_{i,\nu}-b'_{i}|_{\nu} &< K_{\lambda,c} \Vert \tilde\bb'\Vert_\infty^{-(\tau_{i,\nu}-\alpha_\nu)/(1-\alpha_\nu)} \quad (1 \leq i \leq d), \\[1.5ex]
\displaystyle\left| b'_{0}f_{j,\nu}(\bx_\nu)-b'_{d+j} \right|_{\nu} &< K_{\lambda,c}  \Vert \tilde\bb'\Vert_\infty^{-(\tau_{d+j,\nu}-\alpha_\nu)/(1-\alpha_\nu)} \quad  (1 \leq j \leq m), \hspace*{8ex} \label{excep_1}.
\end{cases}
\end{align}

\end{proof}

\medskip

\begin{corollary}\label{coro1}
	Let us assume the same notation and assumption as in Theorem \ref{Bad_nonextremal}. Let $\lambda_\nu$ be as in Theorem \ref{diri}, which depends on $\bx$, and $c=(c_1,\cdots, c_m)\in \Z^m$. Then there exists $0<\alpha_{\nu}<\min_{i=1}^n\{\tau_{i,\nu},1\}$ for all $\nu\in S$, such that for $\mu_{S,d}$ almost every $\bx\in\U$, there are infinitely many $\tilde\bb=(b_0,\cdots,b_n)\in\Z^{n+1}$ with $\vert b_0\vert_\nu\geq  \Vert\tilde\bb\Vert_\infty^{-\alpha_\nu}$ $\forall\nu\in S$ and
	\begin{equation}
	\begin{cases}
	\left|x_{i,\nu} -\frac{b_i}{b_0}\right|_{\nu} < \nu^{(n+m\lambda+\sum_{j=1}^{m}c_{j})/d}  \vert b_0\vert_\nu^{-1}\Vert\tilde\bb\Vert_\infty^{-\tau_{i,\nu}} \qquad\qquad\quad (1 \leq i \leq d), \\[2ex]
	\left|f_{j,\nu}\left( \frac{b_{1}}{b_{0}}, \dots , \frac{b_{d}}{b_{0}} \right) - \frac{b_{d+j}}{b_0} \right|_{\nu} < \nu^{-c_j} \vert b_0\vert_\nu^{-1} \Vert \tilde\bb\Vert_\infty^{-\tau_{d+j,\nu}} \, \qquad\quad (1 \leq j \leq m).  \\[2ex]
	\end{cases}
	\end{equation}

\end{corollary}

\begin{proof}
 For all $\nu\in S$ with $\min_{i=1}^n\tau_{i,\nu}\leq 1,$ we choose $\varepsilon_\nu>0$ such that
    \begin{equation}\label{exponent_less10}
        \omega_{\bt^\nu}\left(\prod_{\sigma\in S\setminus\{\nu\}}\mathcal{M}_\sigma\right)<\frac{1-1/(n+1)\sum_{i=1}^n\tau_{i,\nu}}{1-\min_{i=1}^n\tau_{i,\nu}+\varepsilon_\nu},
    \end{equation} where $\bt^\nu$ is defined in Equation \eqref{tau_nu}.
    This is possible by the assumption of the manifold.
    Then for these $\nu$, we choose $\alpha_\nu>0$ such that $\min_{i=1}^n\tau_{i,\nu}-\varepsilon_\nu<\alpha_\nu<\min^n_{i=1}\tau_{i,\nu}.$ For $\nu\in S$ such that $\min_{i=1}^n\tau_{i,\nu}> 1,$ we choose  some $\alpha_\nu$ such that $0<\alpha_\nu<1$. We now apply Theorem \ref{diri} with these $\alpha_\nu.$

   Suppose $\bx$ satisfies the second possibility of Theorem \ref{diri}, which means there exists some $\nu\in S$ with $\min_{i=1}^n\tau_{i,\nu}\leq 1$, and $K_{\lambda,c}>0$ such that for infinitely many $\tilde\bb=(b_0,\cdots,b_n)\in\Z^{n+1}$ and for all $\sigma\in S\setminus\{\nu\}$ \begin{align}
	\begin{cases}
	|b_{0}x_{i,\sigma}-b_{i}|_{\sigma}^{1/{\tau^\nu_{i,\sigma}}} &<K_{\lambda,c} \Vert \tilde\bb\Vert_\infty^{-\frac{(1-\frac{\sum_{i=1}^n \tau_{i,\nu}}{n+1})}{(1-\alpha_\nu)}} \quad (1 \leq i \leq d), \\[1.5ex]
	\displaystyle\left| b_{0}f_{j,\sigma}(\bx_\sigma)-b_{d+j} \right|_{\sigma}^{1/{\tau^\nu_{i,\sigma}}} &< K_{\lambda,c}\Vert \tilde\bb\Vert_\infty^{-\frac{(1-\frac{\sum_{i=1}^n \tau_{i,\nu}}{n+1})}{(1-\alpha_\nu)}} \hspace*{8ex} (1 \leq j \leq m).
	\end{cases}
	\end{align}
 Now note that $$\frac{(1-\frac{\sum_{i=1}^n \tau_{i,\nu}}{n+1})}{(1-\alpha_\nu)}> \frac{1-1/(n+1)\sum_{i=1}^n\tau_{i,\nu}}{1-\min_{i=1}^n\tau_{i,\nu}+\varepsilon_\nu}. $$ By Equation \eqref{exponent_less10}, we have that the set of such $\bx$ has  $\mu_{S,d}$ measure zero.
 \end{proof}

\medskip

\begin{corollary}\label{coro12}
	Let us take the same notation and assumption as in Theorem \ref{bad_manifold_S-adic_nonextremal}. Let $\lambda_\nu$ be as in Theorem \ref{diri}, which depends on $\bx$, and $c=(c_1,\cdots, c_m)\in \Z^m$. Then there exist $0<\alpha_{\nu}<\min_{i=1}^n\{\tau_{i,\nu},1\}$ for all $\nu\in S$, such that the same conclusion as in Corollary \ref{coro1} follows.	
\end{corollary}

\begin{proof}
    Let $\alpha_\nu=\alpha$ for every $\nu\in S$, where $0<\alpha<\min_{\nu\in S}\min_{i=1}^n\{\tau_{i,\nu},1\}$. Now we apply Theorem \ref{diri} with this choice of $\alpha.$
Suppose $\bx$ satisfies the second possibility of Theorem \ref{diri}, which means there exists some $\nu\in S$ with $\min_{i=1}^n\tau_{i,\nu}\leq 1$, and $K_{\lambda,c}>0$ such that for infinitely many $\tilde\bb=(b_0,\cdots,b_n)\in\Z^{n+1}$ and for all $\sigma\in S$, \begin{align}
	\begin{cases}
	|b_{0}x_{i,\sigma}-b_{i}|_{\sigma}^{1/{\tau^{\nu,\alpha}_{i,\sigma}}} &<K_{\lambda,c} \Vert \tilde\bb\Vert_\infty^{-\frac{n+1-n\alpha}{(n+1)(1-\alpha)}} \quad (1 \leq i \leq d), \\[1.5ex]
	\displaystyle\left| b_{0}f_{j,\sigma}(\bx_\sigma)-b_{d+j} \right|_{\sigma}^{1/{\tau^{\nu,\alpha}_{i,\sigma}}} &< K_{\lambda,c}\Vert \tilde\bb\Vert_\infty^{-\frac{n+1-n\alpha}{(n+1)(1-\alpha)}} \hspace*{8ex} (1 \leq j \leq m),
	\end{cases}
	\end{align}
where $\tau_{i,\sigma}^{\nu,\alpha}$ is defined in Equation \eqref{taunu}. Since $\frac{n+1-n\alpha}{(n+1)(1-\alpha)}=1+ \frac{\alpha}{(1-\alpha)(n+1)}$, by Equation \eqref{exponent_less2}, we have that the set of such $\bx$ has $\mu_{S,d}$ measure zero.
\end{proof}

\medskip

\begin{corollary} \label{coldiri+2}
Assume the conclusion of the previous two corollaries; Corollary \ref{coro1} and Corollary \ref{coro12}. Let $\delta_{1}, \dots , \delta_{n} >0$ be any constants. Then for $\mu_{S,d}$ almost every $\bx \in \U$ the following system
	\begin{equation} \label{cas+2}
	\begin{cases}
	\left|x_{i,\nu} -\frac{b_{i}}{b_{0}}\right|_{\nu} < \delta_{i} \vert b_0\vert_\nu^{-1}\Vert \tilde\bb\Vert_\infty^{-\tau_{i,\nu}} & (1 \leq i \leq d), \nu\in S, \\[2ex]
	\left|f_{j,\nu}\left(\frac{b_{1}}{b_{0}}, \dots , \frac{b_{d}}{b_{0}} \right) - \frac{b_{d+j}}{b_{0}} \right|_{\nu} < \delta_{d+j}\vert b_0\vert_\nu^{-1}\Vert \tilde\bb\Vert_\infty^{-\tau_{d+j,\nu}}  \, & (1 \leq j \leq m),\nu\in S\,,
	\end{cases}
	\end{equation}
has infinitely many integer solutions $\tilde\bb=(b_{0}, \dots, b_{n}) \in \Z^{n+1}$ such that $\vert b_0\vert_\nu\geq  \Vert\tilde\bb\Vert_\infty^{-\alpha_\nu}$ for all $\nu\in S$. Here $\alpha_\nu, \forall \nu\in S$ be as in the previous two corollaries.
\end{corollary}

\begin{proof}
Suppose $0<\alpha_{\nu}<\min_{i=1}^n\{\tau_{i,\nu},1\}, \nu\in S$ be as in Corollary \ref{coro1}, and Corollary \ref{coro12}.
	Define the set of integer points
	\begin{equation}\label{eqn047}
	S_{\bt}(\delta_{d+1}, \dots , \delta_{n})=\left\{(b_{0},\dots, b_{n}) \in \Z^{n+1}: \begin{array}{l}
	(\frac{b_1}{b_0},\cdots,\frac{b_d}{b_0})\in\U, \vert b_0\vert_\nu\geq  \Vert\tilde\bb\Vert_\infty^{-\alpha_\nu} ~~\forall\nu\in S, \\[1ex]
	\left| f_{j,\nu}\left(\frac{b_{1}}{b_{0}}, \dots , \frac{b_{d}}{b_{0}} \right) -\frac{b_{d+j}}{b_{0}}\right|_{\nu}\!\!<\delta_{d+j}\vert b_0\vert_\nu^{-1}\Vert \tilde\bb\Vert_\infty^{-\tau_{d+j,\nu}} \\[1.5ex]
	\end{array}
	\right\}
	\end{equation}
	and for each $\tilde\bb\in S_{\bt}(\delta_{d+1},\cdots,\delta_n)$ and $\delta_1,\cdots,\delta_d>0$ consider the following sets
	\begin{equation}\label{eqn048B}
	B_{\tilde\bb}(\bt;\delta_{1}, \dots , \delta_{d})=\left\{ \bx \in \Z_S^{d}: \left|x_{i,\nu}-\frac{b_{i}}{b_{0}}\right|_{\nu}<\delta_{i}\vert b_0\vert_\nu^{-1}\Vert \tilde\bb\Vert_\infty^{-\tau_{i,\nu}}, (1 \leq i \leq d ) \right\}.
	\end{equation}
	
	If $\delta_{d+1}= \dots = \delta_{n}=\delta$ then we may write $S_{\bt}(\delta)$ for ease of notation, and similarly if $\delta_{1}= \dots =\delta_{d}=\delta$ then we write $B_{\tilde\bb}(\bt; \delta)$. Choose $c_{1}, \dots c_{m}$ in Corollary~\ref{coro1} such that $\nu^{-c_{j}} \leq \delta_{d+j}$ for each $1 \leq j \leq m$. Then, by Corollary~\ref{coro1},
	\begin{equation}\label{eqn050}
	\mu_{S,d}\left( \bigcup_{\delta>0} \limsup_{\tilde\bb \in S_{\bt}(\delta_{d+1}, \dots , \delta_{n})} B_{\tilde\bb}(\bt;\delta) \right)=\mu_{S,d}(\U).
	\end{equation}
	These are Borel sets and therefore measurable. Hence, by the continuity of measure, we have that
	\begin{equation}\label{eqn051}
	\lim_{\delta\to+\infty}\mu_{S, d}\left(\underset{\tilde\bb \in S_{\bt}(\delta_{d+1}, \dots , \delta_{n})}\limsup B_{\tilde\bb}(\bt;\delta)\right)=
	\mu_{S, d}\left(\bigcup_{\delta>0}\underset{\tilde\bb \in S_{\bt}(\delta_{d+1}, \dots , \delta_{n})}\limsup B_{\tilde\bb}(\bt;\delta)\right)=
	\mu_{S, d}(\U)\,.
	\end{equation}
	By Corollary~\ref{CI_balls_corollary}, every limsup set in \eqref{eqn051} is of the same measure. Hence,
	$$
	\mu_{S, d}\left(\underset{\tilde\bb \in S_{\bt}(\delta_{d+1}, \dots , \delta_{n})}\limsup B_{\tilde\bb}(\bt;\delta)\right)=
	\mu_{S, d}(\U)
	$$
	for every $\delta>0$. Explicitly, taking $\delta=\min_{1 \leq i \leq d}\delta_{i}$ then we obtain our result.
\end{proof}

\medskip

\begin{proposition}\label{limsup_F-1}
Let us take the same notation and assumption as in either of Theorem \ref{Bad_nonextremal} and Theorem \ref{bad_manifold_S-adic_nonextremal}.
 Suppose that  $\U^{*}$ be a subset of $\U$ and the partial derivatives of $\bff$ and the constants $C_{\bx}$ arising from \eqref{eqn007}  are bounded over $\U^{*}$ and define
\begin{equation*}
D:= \sup_{\bx \in \U^{*}} \underset{1 \leq j \leq m}{\max_{1 \leq i \leq d}} \left\{ \left|\partial_i f_{j,\nu}(\bx)\right|_{\nu}, C_{\bx} \right\}.
\end{equation*}
 Let $\bt=(\tau_{1}, \dots , \tau_{n})\in \R^{n}_{+}$ and let $S_{\bt}(\delta_{d+1}, \dots, \delta_{n})$ and $B_{\ba}(\bt, \delta_{1}, \dots , \delta_{d})$ be defined by \eqref{eqn047} and \eqref{eqn048B} respectively. Then
\begin{equation} \label{link}
\U^{*} \cap \limsup_{\ba \in S_{\bt}(\delta)}B_{\ba}(\bt; D^{-1}\delta) \, \,  \subseteq \, \,  \U^{*} \cap \ff^{-1}\left( \W_{S,1,n}(\delta ; \bt) \right) \, \, \subseteq \, \,  \U^{*} \cap \limsup_{\ba \in S_{\bt}(D\delta)}B_{\ba}(\bt; \delta).
\end{equation}
\end{proposition}

\begin{proof}
Suppose $\bx \in \U^{*} \cap \ff^{-1}\left( W_{S,1,n}(\delta ; \bt) \right)$. Then there exist infinitely many integer solutions $(b_{0}, \dots , b_{n}) \in \Z^{n+1}$ to
\begin{equation} \label{cases52}
\begin{cases}
\left|b_0x_{i,\nu}-b_{i} \right|_{\nu}<\delta \Vert \tilde\bb\Vert_\infty^{-\tau_{i,\nu}} \quad (1 \leq i \leq d), \\
\left| b_0f_{j,\nu}(\bx_\nu)-b_{d+j} \right|_{\nu}<\delta \Vert \tilde\bb\Vert_\infty^{-\tau_{d+j,\nu}} \quad (1 \leq j \leq m).
\end{cases}
\end{equation} Moreover, using the same methods as in the proof of Corollary \ref{coro1} and Corollary \ref{coro12}, one can also guarantee that for almost every $\bx$, the solutions $\tilde\bb$ in \eqref{cases52}, $\vert b_0\vert_\nu\geq \Vert \tilde\bb\Vert_\infty^{\alpha_\nu}.$

Conversely, if $\bx \in \limsup_{\tilde\bb \in S_{\bt}(\delta_{d+1}, \dots , \delta_{n})}B_{\tilde\bb}(\bt;\delta_{1}, \dots , \delta_{d})$, then there exist infinitely many integer solutions $(b_{0}, \dots , b_{n}) \in \Z^{n+1}$ to \eqref{cas+2}. Observe that when $\delta_{1}= \dots = \delta_{d}=\delta$ then the first set of inequalities of \eqref{cas+2} and \eqref{cases52} are equivalent, and if $\delta_{1}= \dots = \delta_{d}=D^{-1}\delta$ then the first set of inequalities of \eqref{cas+2} are contained in the first set of inequalities of \eqref{cases52} since $D \geq 1$. And so the inequalities of \eqref{link} are done over the independent variables.

For the dependent variables, we use that $\bff$ is $C^2$ over all $\U^{*}$. For any $\bx \in \U^{*}$ we have that
 \begin{align} \label{DQE_result}
 \left|f_{j,\nu}(\bx_\nu)-f_{j,\nu}\left(\frac{b_{1}}{b_{0}}, \dots , \frac{b_{d}}{b_{0}} \right)\right|_{\nu} & <
\max\left\{\max_{1\le i\le d}\left|\partial_if_{j,\nu}(\bx_\nu)\right|_\nu\max_{1 \leq i \leq d}\left|x_{i,\nu}-\frac{b_{i}}{b_{0}} \right|_{\nu},\right.\nonumber\\[2ex] &\hspace{24ex} \left.C\max_{1 \leq i \leq d}\left|x_{i,\nu}-\frac{b_{i}}{b_{0}} \right|_{\nu}^2 \right\}\nonumber \\
&<D\max_{1 \leq i \leq d}\left| x_{i,\nu}-\frac{b_{i}}{b_{0}} \right|_{\nu},
\end{align}
by the definition of $D$.

Hence, if $\bx \in \U^{*} \cap \ff^{-1}\left( W_{S,1,n}(\delta ; \bt) \right)$ there exist infinitely many $\tilde\bb\in\Z^{n+1}$ with $\vert b_0\vert _\nu\geq \Vert \tilde\bb\Vert_\infty^{\alpha_\nu},$
\begin{equation*}\begin{aligned}
\left|f_{j,\nu}\left(\frac{b_{1}}{b_{0}}, \dots \frac{b_{d}}{b_{0}}  \right)-\frac{b_{d+j}}{b_0}\right|_{\nu}\leq &\max\left\{\left|f_{j,\nu}(\bx_\nu)-f_{j,\nu}\left(\frac{b_{1}}{b_{0}}, \dots , \frac{b_{d}}{b_{0}} \right)\right|_{\nu},
\left|\frac{b_{d+j}}{b_0}-f_{j,\nu}\left(\bx_\nu \right)\right|_{\nu}
\right\} \\& <
 D\delta \vert b_0\vert_\nu^{-1}\Vert \tilde\bb\Vert_\infty^{-\tau_{d+j,\nu}}, \forall\nu\in S\end{aligned}
 \end{equation*}
 where the last inequality follows by the first row of \eqref{cases52}, by \eqref{DQE_result}, and by \eqref{tau_standard}. Hence $\U^{*} \cap \ff^{-1}\left( \W_{S,1,n}(\delta ; \bt) \right) \, \, \subseteq \, \,  \U^{*} \cap \limsup_{\tilde\bb \in S_{\bt}(D\delta)}B_{\tilde\bb}(\bt; \delta)$. \par
 If $\bx \in \U^{*} \cap \limsup_{\tilde\bb \in S_{\bt}(\delta)}B_{\tilde\bb}(\bt; D^{-1}\delta)$, then
 \begin{equation*}\begin{aligned}
\left|f_{j,\nu}\left(\bx \right)-\frac{b_{d+j}}{b_0}\right|_{\nu} \le &\max\left\{\left|f_{j,\nu}\left(\frac{b_{1}}{b_{0}}, \dots , \frac{b_{d}}{b_{0}} \right)- f_{j,\nu}(\bx_\nu)\right|_{\nu}, \left|\frac{b_{d+j}}{b_0}- f_{j,\nu}\left(\frac{b_{1}}{b_{0}}, \dots \frac{b_{d}}{b_{0}} \right)\right|_{\nu}
\right\}\\  &<
 \delta \vert b_0\vert_\nu^{-1}\Vert \tilde\bb\Vert_\infty^{-\tau_{d+j,\nu}},\end{aligned}
 \end{equation*}
 when we apply \eqref{DQE_result} and \eqref{cas+2} with $\delta_{1}=\dots= \delta_{d}=D^{-1}\delta$ and $\delta_{d+1}= \dots = \delta_{n}=\delta$. Hence we have that $ \U^{*} \cap \limsup_{\tilde\bb \in S_{\bt}(\delta)}B_{\tilde\bb}(\bt; D^{-1}\delta) \, \,  \subseteq \, \,  \U^{*} \cap \ff^{-1}\left( \W_{S,1,n}(\delta ; \bt) \right) $.
 \end{proof}

\medskip

\subsection{Proof of Theorem \ref{Bad_nonextremal} and Theorem \ref{bad_manifold_S-adic_nonextremal}}

Using Proposition  \ref{limsup_F-1}, Corollary \ref{coldiri+2}, Corollary \ref{coro1}, and Corollary \ref{coro12}, we have that \begin{equation*}
 \mu_{S,d}\left( \ff^{-1}(\W_{1, n}(\delta ; \bt)) \right)=\mu_{S,d}(\U)
	\end{equation*}
	for any $\delta>0$.
Hence subsets $\{ \ff^{-1}\left(\W_{1, n}\left( \frac{1}{k} ; \bt \right) \right) \}_{k \in \N}$ have same measure as $\U$. Therefore
	\begin{equation*}
	\mu_{S,d}\left( \ff^{-1}(\Bad_{n}(\bt)) \right)=\mu_{S,d}\left( \U \backslash \bigcap_{k \in \N}\ff^{-1}\left(\W_{1, n}\left( \frac{1}{k}; \bt\right)\right) \right)=0
	\end{equation*}
	completing the proof.

\section{Proofs: Bad on real manifolds, $S=\{\infty\}$}\label{proofR}

This section gives a proof of Theorem \ref{bad_realmani}, which, although being similar in nature, is much simpler than the proofs of the $S$-arithmetic results and should help the reader to wind down after the elaborate $S$-arithmetic calculations.

\begin{theorem}
Let $\bff=(f_{1}, \dots , f_{m}):\U\subset\R^d \to \R^{m}$ be a map defined on an open subset $\U$, $\bx\in\U\setminus \Q^d$ and suppose that $\bff$ is $C^2$ at $\bx$. let $\lambda$ be given by
	\begin{align}
	\max\left\{1,\,\underset{1 \leq j \leq m}{\max_{1 \leq i \leq d}} \left| \frac{\partial f_{j}}{\partial x_{i}}(\bx) \right|_{\infty}\right\}=e^{\lambda}\,. \label{first_order_constant_R}
	\end{align}
 Let $0<\alpha<\tau_{i}$ and $i=1,\cdots,n$. Let $c=(c_1,\cdots,c_m)\in\Z^m$, $n=m+d$, and $\bt=(\tau_{1}, \dots, \tau_{n}) \in \R^{n}_{+}$ such that $\bt$ satisfies \eqref{tau_standard} and $\max_{j=1}^m\tau_{d+j}\leq\min_{i=1}^d\tau_{i}$.
 	There exist infinitely many $\tilde\bb=(b_0,\cdots,b_n)\in\Z^{n+1}$,
	\begin{equation}
	\begin{cases}
	\left|x_{i} -\frac{b_i}{b_0}\right|_{\infty} < e^{(n+m\lambda+\sum_{j=1}^{m}c_{j})/d}   \vert b_0\vert_\infty^{-(\tau_{i}+1)} \qquad\qquad\qquad\quad (1 \leq i \leq d), \\[2ex]
	\left|f_{j}\left( \frac{b_{1}}{b_{0}}, \dots , \frac{b_{d}}{b_{0}} \right) - \frac{b_{d+j}}{b_0} \right|_{\infty} < e^{-c_j}  \vert b_0\vert_\infty^{-(\tau_{d+j}+1)} \, \qquad\quad (1 \leq j \leq m), \label{tau_{2R}} \\[2ex]
	\end{cases}
	\end{equation}

\end{theorem}
\begin{proof}
     By Minkowski's convex body theorem there exists $H_{\sigma}>0$ such that for any integer $H \ge H_{\sigma}^{1/(n+1)}$ the following system
	\begin{align}
	\begin{cases}
	|b_{0}x_{i}-b_{i}|_{\infty} &< e^{(n+m\lambda+\sum_{j=1}^{m}c_{j})/d} H^{-\tau_{i}} \quad (1 \leq i \leq d), \\[1.5ex]
	\displaystyle\left| b_{0}f_{j}(\bx)-\sum_{i=1}^{d}\partial_i f_{j}(\bx) \left( b_0x_{i}-b_{i} \right)-b_{d+j} \right|_{\infty} &< e^{-c_j} H^{-\tau_{d+j}} \hspace*{8ex} (1 \leq j \leq m), \label{ineq2R}\\[1.5ex]
	|b_{0}|_\infty  & \leq H
	\end{cases}
	\end{align}
 $$
 \vert \bff(\bx)-\bff(\by)-\sum_{i=1}^d \partial_{i}\bff(\bx)(x_i-y_i)\vert_\infty<C_{\bx}\max_{i=1}^d \vert x_i-y_i\vert_\infty^2.
 $$
Hence using \eqref{ineq2R},
for  infinitely many $\tilde\bb$ appearing in Equation \eqref{ineq2R},
	\begin{align} \label{f_R}
	\left|\left( f_{j} \left(\frac{b_{1}}{b_{0}}, \dots , \frac{b_{d}}{b_{0}} \right)-f_{j}(\bx)-\sum_{i=1}^d \partial_i f_{j}(\bx) \left(\frac{b_{i}}{b_{0}}-x_{i} \right)\right) \right|_{\infty} & < C\max_{1 \leq i \leq d} \left| \frac{b_{i}}{b_{0}}-x_{i} \right|_{\infty}^{2}\nonumber\\
 & < C e^{(n+m\lambda+\sum_{j=1}^{m}c_{j})/d} \vert b_0\vert_\infty^{-2(\tau_{i}+1)} \nonumber\\
 &<e^{-c_{j}}\vert b_0\vert_\infty^{-(\tau_{d+j}+1)}
	\end{align}
	for each $1 \leq j \leq m$ since $\max_{j=1}^m\tau_{d+j}\leq\min_{i=1}^d\tau_{i}$.
 Therefore combining Equation \eqref{f_R} and \eqref{ineq2R} we have infinitely many $\tilde\bb$ appearing in Equation \eqref{ineq2R}, such that
 $$\left\vert b_0 f_{j}\left(\frac{b_1}{b_0},\cdots,\frac{b_d}{b_0}\right)-b_{d+j}\right\vert_\infty <e^{-c_j}\vert b_0\vert_\infty^{-\tau_{d+j}}.$$ This completes the proof.
\end{proof}

Now similarly to Corollary~\ref{coldiri+2}, we have the following corollary.

\begin{corollary}\label{coldiri+R}
Let us take the same set-up as in the previous theorem.
Let $\delta_{1}, \dots , \delta_{n} >0$ be any constants. Then for almost every $\bx \in \U$ the following system
	\begin{equation} \label{cas+R}
	\begin{cases}
	\left|x_{i} -\frac{b_{i}}{b_{0}}\right|_{\infty} < \delta_{i} \vert b_0\vert_\infty^{-(1+\tau_{i})} & (1 \leq i \leq d), \\[2ex]
	\left|f_{j}\left(\frac{b_{1}}{b_{0}}, \dots , \frac{b_{d}}{b_{0}} \right) - \frac{b_{d+j}}{b_{0}} \right|_{\infty} < \delta_{d+j}\vert b_0\vert_\infty^{-(\tau_{d+j}+1)}  \, & (1 \leq j \leq m)\,,
	\end{cases}
	\end{equation}
has infinitely many integer solutions $\tilde\bb=(b_{0}, \dots, b_{n}) \in \Z^{n+1}$.
\end{corollary}

Hence using a similar version to Proposition  \ref{limsup_F-1}, and using Lemma \ref{coldiri+R} we have that \begin{equation*}
 \mu_{d}\left( \ff^{-1}(\W_{1, n}(\delta ; \bt)) \right)=\mu_{d}(\U)
	\end{equation*}
	for any $\delta>0$.
Hence subsets $\{ \ff^{-1}\left(\W_{1, n}\left( \frac{1}{k} ; \bt \right) \right) \}_{k \in \N}$ of $\U$ have the same measure as $\U$. Therefore
	\begin{equation*}
	\mu_{d}\left( \ff^{-1}(\Bad_{n}(\bt)) \right)=\mu_{d}\left( \U \backslash \bigcap_{k \in \N}\ff^{-1}\left(\W_{1, n}\left( \frac{1}{k}; \bt\right)\right) \right)=0
	\end{equation*}
	completing the proof.

\subsection*{Acknowledgements} SD thanks Subhajit Jana and Ralf Spatzier for several helpful remarks which have improved the presentation of this paper. She also thanks University of York, for warm hospitality, where some part of this work was done.

\bibliographystyle{plain}
\bibliography{badnull}
\end{document}